\documentclass{article}
\usepackage{amsmath,  amsfonts, amsthm, latexsym, amssymb}
\usepackage{a4wide} 
 \numberwithin{equation}{section}
\newtheorem{theorem}{Theorem}[section]

\newtheorem{proposition}[theorem]{Proposition}
\newtheorem{remark}[theorem]{Remark}
\newtheorem{lemma}[theorem]{Lemma}
 
\newtheorem{corollary}[theorem]{Corollary}
\def\eps{\varepsilon}
\newcommand{\cA}{{\cal A}}

\newcommand{\cE}{{\cal E}}
\newcommand{\cK}{{\cal K}}
\newcommand{\di}{\operatorname{\text{div}}}

\newcommand{\R}{{\mathbb R}}

\titlepage
\title{Concentration of ground states in stationary \\ Mean-Field Games systems}
\author{Annalisa Cesaroni and Marco Cirant }
\date{ }

\begin{document}

\maketitle

\begin{abstract}
 In this paper we provide the  existence of classical solutions to stationary mean field game systems in the whole space $\R^N$, 
with coercive  potential and aggregating local coupling, under general conditions on the Hamiltonian.
The only  structural assumption we make is   on the growth  at infinity of the  coupling term in terms of the growth of the Hamiltonian.
This result is obtained  using a variational approach based on the analysis of the non-convex energy associated to the system. 
Finally, we show that in the vanishing viscosity limit mass concentrates around the flattest minima of the potential. We also describe  
the asymptotic shape of the rescaled solutions in the vanishing viscosity limit, in particular proving the existence of ground states,
 i.e. classical solutions to mean field game systems in the whole space without potential, and with aggregating coupling. 
\medskip

\noindent
{\footnotesize \textbf{AMS-Subject Classification}}. {\footnotesize 35J50, 49N70, 35J47, 91A13, 35B25}
\end{abstract}

\tableofcontents

\section{Introduction} 
We consider a class of ergodic Mean-Field Games systems set on the whole space $\R^N$ with  unbounded decreasing coupling: our problem is, given $\eps >0$ and $M>0$, to find a constant $\lambda \in \R$ for which there exists a couple $(u,m)\in C^{2}(\R^N) \times W^{1,p}(\R^N)$, for any $p>1$, solving
\begin{equation}\label{mfg}\begin{cases}
-\eps \Delta u+ H(\nabla u)+\lambda= f(m)+V(x) \\
- \eps\Delta m-\di(m \nabla H(\nabla u) )=0 & \text{on $\R^N$,} \\   \int_{\R^N} m=M. \end{cases}
\end{equation}The aim of this work is two-fold. Firstly, for any fixed $\eps >0$, we prove the existence of classical ground states of \eqref{mfg}. Secondly, we study their behavior in the vanishing viscosity limit $\eps \to 0$.

The Hamiltonian $H:\R^N\to\R$ is  strictly convex,  $H\in C^2(\R^N\setminus \{0\})$ and has superlinear growth: we assume that there exist $C_H > 0$, $K>0$ and $\gamma > 1$
such that, for all $p\in\R^N$,
\begin{equation}\begin{split}\label{Hass}
&C_H |p|^{\gamma} - K \le H(p) 
\le C_H |p|^{\gamma}, \\ 
&\nabla H(p)\cdot p-H(p)\geq K^{-1} |p|^\gamma -K \quad
{\mbox{   and  }} \quad   |\nabla H(p)|\leq K|p|^{\gamma-1}. \end{split}\end{equation}
The coupling term  $f:[0, +\infty)\to\R $ 
is  a locally Lipschitz continuous function such
 that  there exist $C_f>0$ and $K>0$ for which 
\begin{equation}\label{assFlocal}  -C_fm^{\alpha}-K\leq  f( m)
\leq -C_f m^{\alpha}+K,\end{equation}
with \begin{equation}\label{alpha}   0<\alpha<\frac{\gamma}{N(\gamma-1)}=\frac{\gamma'}{N},
 \end{equation} where $\gamma'=\frac{\gamma}{\gamma-1}$ is the conjugate exponent of $\gamma$.\\
  Finally, we assume that the potential $V$ is a locally H\"older continuous function, and that  there exist $b>0$ and a constant $C_V>0$ such that 
\begin{equation}\label{vass}
C_V^{-1} (\max\{ |x|-C_V, 0\})^{b} \le V(x) \le C_V (1 + |x|)^{b}. 
\end{equation} Note that the requirement of $V$ to be non-negative is not crucial, we just need it to be bounded from below.

Mean-Field Games (MFG) is a recent theory that models the behaviour of a very large number of indistinguishable rational agents aiming at minimizing a common cost. The theory was introduced in the seminal works by Lasry, Lions \cite{LL061, LL062, LL07, Lcol} and by Huang, Caines, Malham\'e \cite{HCM06}, and has been rapidly growing during the last decade due to its mathematical challenges and several potential applications (from economics and finance, to engineering and models of social systems).
In the ergodic MFG setting, the dynamics of a typical agent is given by the controlled stochastic differential equation
\[
d X_s = -v_s ds + \sqrt{2 \eps} \, d B_s, \ s>0,
\]
where $v_s$ is the control and $B_s$ is a Brownian motion, with initial state given by a random variable $X_0$.
The cost (of long-time average form) is given by
\[
\lim_{T \to \infty } \frac{1}{T} \mathbb{E} \int_{0}^T [L(v_s) + V(X_s) + f(m(X_s))] ds,
\]
where the Lagrangian $L$ is the Legendre transform of $H$ (see \eqref{Leg}) and $m(x)$ denotes the density of population of small agents at a position $x \in \R^N$. A typical agent minimizes his own cost, and the density of its corresponding distribution law $\mathcal{L}(X_s)$ converges  as time $s \to \infty$ to a stationary density $\mu$, which is independent of the initial distribution $\mathcal{L}(X_0)$. In an equilibrium regime, $\mu$ coincides with the population density $m$. This equilibrium is encoded from the PDE viewpoint in \eqref{mfg}: a solution $u$ of the Hamilton-Jacobi-Bellman equation gives an optimal control for the typical agent in feedback form $\nabla H(\nabla u(\cdot))$, and the Kolmogorov equation provides the density $m$ of the agents playing in an optimal way. 

The two key points of our setting are the following: firstly, the cost is monotonically {\it decreasing} with respect to the population distribution $m$, namely agents are attracted toward congested areas. A large part of the MFG literature focuses on the study of systems with competition, namely when the coupling in the cost is monotonically increasing; this assumption is essential if one seeks for uniqueness of equilibria, and it is in general crucial in many existence and regularity arguments, see, e.g \cite{gpv}, and references therein. On the other hand, models with aggregation like \eqref{mfg} have been considered in few cases, see \cite{notebari, c16, CirO, cpre, Go16}.

Secondly, the state of a typical agent here is the {\it whole euclidean space} $\R^N$. Usually, the analysis of \eqref{mfg} is carried out in the periodic setting, in order to avoid boundary issues and the non-compactness of $\R^N$. Few investigations are available in the truly non-periodic setting: see \cite{Porretta} for time-dependent problems, \cite{Arapostathis} for the case of bounded controls, \cite{GoPi16} for some regularity results and \cite{BarPri} for the Linear-Quadratic framework. We observe that the non-compact setting is even more delicate for stationary (ergodic) problems like \eqref{mfg}: a stable long-time regime of a typical player is ensured if the Brownian motion is compensated by the optimal velocity $v_s$. In other words, if a force that drives players to bounded states is missing, dissipation eventually leads their distribution to vanish on the whole $\R^N$. This phenomenon is impossible if the state space is compact. The main issue here is that the behaviour of the optimal velocity $v_s(\cdot) = \nabla H(\nabla u(\cdot))$ is a priori unknown, and depends in an implicit way on $V$ and the distribution $m$ itself.  Note that $V(\cdot)$ represents the spatial preference of a single agent; if it grows as $|x| \to \infty$, it discourages agents to be far away from the origin. At the PDE level, this will compensate the lack of compactness of $\R^N$.  Let us mention that even without the coupling term $f(m^\alpha)$, the ergodic control problem in unbounded domains has received a considerable attention, see e.g. \cite{bm, ichi11, ichi15} and references therein.

In our analysis, we exploit the variational nature of the system \eqref{mfg}, which has been pointed out already in the first papers on MFG, see \cite{LL07}, or the more recent work \cite{MesSil}. Indeed,
solutions to \eqref{mfg} can be put in correspondence with critical points of  
the following energy \begin{equation}\label{energiaintro}
\mathcal{E}(m, w) := \begin{cases}
\displaystyle \int_{\R^N} m L\left(-\frac{w}{m}\right) + V(x)m  + F(m) \; dx & \text{ if $(m,w) \in\mathcal{K}_{\eps, M}$}, \\
+\infty & \text{otherwise},
\end{cases}
\end{equation}
where  $F(m)=\int_0^m f(n)dn$ for $m\geq 0$ and $F(m)=0$ for $m\leq 0$ and \begin{equation}\label{dati}
L\left(-\frac{w}{m}\right):=\begin{cases}  \sup_{p\in\R^N}(-\frac{p\cdot w}{m}-H(p)) & {\mbox{ if }}m>0,\\ 
0 & {\mbox{ if }}m=0, w=0,\\ +\infty & \text{otherwise.} \end{cases} \end{equation} Note that $m L(-\cdot/m)$ reads as the Legendre transform of $m H(\cdot)$.  
 The  constraint set  is defined as  \begin{equation}\begin{split} 
\label{kcalconstraint}
\mathcal{K}_{\eps, M}: =\,&\left\{ (m,w) \in L^1(\R^N)\cap L^q(\R^N)
\times L^{1}(\R^N)\; {\mbox{ s.t.}} \right. \\ &\quad
\eps \int_{\R^N}  m  (-\Delta \varphi) \, dx = \int_{\R^N} w \cdot \nabla \varphi \, dx 
\quad \forall \varphi \in C_0^\infty(\R^N), \\ &\quad \left.
\int_{\R^N} m \, dx = M, \quad \text{$m \ge 0$ a.e.}  \right\}\qquad\text{ with } q=\begin{cases} \frac{N}{N-\gamma'+1} & \gamma'\leq N \\ \gamma' & \gamma'>N.\end{cases}
\end{split}\end{equation} 

%
%The relation between the energy, the constraint set and the solutions to \eqref{mfg} can be observed by differentiating $\mathcal{E}$ around its critical points, and by introducing a suitable adjoint equation for $u$ involving $m, w$. By the definition of $L$ and the constraints, $w=-m\nabla H(\nabla u)$, and such adjoint equation turns out to be the HJB equation in \eqref{mfg}. Note that, by studying the problem in terms of the variables $m, w$, one has that $(m,w)\to m L\left(-\frac{w}{m}\right)$ is a convex function, and the differential constraint  $\eps \Delta m=\di(w)$ is linear.

Under assumption  \eqref{assFlocal} on  the coupling term, the energy $\mathcal{E}$ is not convex. Condition \eqref{alpha} is necessary for  the problem $e_\eps(M):=\min_{(m,w)\in \mathcal{K}_{\eps, M}} \mathcal{E}(m, w)$ to be well-posed. Indeed, consider any $(m_0, w_0) \in \cK_{\eps,M}$ such that $m_0$ has compact support. An easy computation shows that if $\alpha > \gamma'/N$, then
\[
\cE(\sigma^{-N}m_0(\sigma^{-1} \cdot), \sigma^{-(N+1)}w_0(\sigma^{-1} \cdot))\to-\infty
\]
as $\sigma \to 0$, so $\cE$ is not bounded from below on $\cK_{\eps,M}$. We show that \eqref{alpha} is indeed sufficient for $e_\eps(M)$ to be finite, and allows to look for {\it ground states} of \eqref{mfg}. This will be accomplished by a study of the Sobolev regularity of the Kolmogorov equation, see in particular Section \ref{s:kolmo_reg}. Note that the critical case $\alpha = \gamma'/N$ is more delicate, and requires additional analysis.  We also mention that another critical exponent is intrinsic in \eqref{mfg}: if $\alpha > \gamma'/(N-\gamma')$, one has to expect non-existence of solutions (see \cite{c16}).
 We refer to our case as the \emph{subcritical case}, in analogy with the $L^2$-subcritical regime in nonlinear 
Schr\"odinger equations with prescribed mass (see \cite[Remark 2.9]{c16} for additional comments). 
The analogy can be made precise in the  purely quadratic framework, that is when $H(p)=\frac{1}{2}|p|^2$. Indeed, as observed in \cite{LL061, LL062}, the so-called Hopf-Cole transformation permits to reduce  the number of unknowns in the system.  Setting $v^2(x):=m(x)=  c e^{-\frac{u(x)}{\eps}}$,  with $c$ normalizing constant, then $v$ is a solution to 
\[-2\eps^2\Delta v+ (V(x)-\lambda) v= -f(v^2)v\] with $\int_{\R^N} v^2(x)dx=M$. Then the energy reads $\mathcal{E}(v)=\int_{\R^N} \eps^2 |\nabla v|^2 +\frac{1}{2}V(x) v^2+\frac{1}{2 } F(v^{2}) dx$.

In our approach,  to construct solutions to \eqref{mfg}, we look for  minimizers $(m,w)\in\mathcal{K}_{\eps, M}$ of the energy \eqref{energiaintro}. These minimizers can be obtained by classical direct methods, by using in particular estimates and compactness in some $L^p$ space for elements $(m,w)$ in $\mathcal{K}_{\eps, M}$ with bounded action, i.e.  which satisfy $\int_{\R^N} m L\left(-\frac{w}{m}\right) dx\leq C$, obtained in Section \ref{s:kolmo_reg}.  Then, the existence of a solution $(u_\eps, \lambda_\eps)$ of the HJB equation in \eqref{mfg} is obtained by considering another functional with linearized coupling (around the minimizer) and the associated dual functional in the sense of Fenchel-Rockafellar (as in \cite{BC16}). One has to take care of the interplay between $u$ and $m$ as $|x| \to \infty$.  To handle the lack of a priori regularity on the function $m$, we first regularize the problem, by applying standard  regularizing convolution kernels on the coupling (see Section \ref{sectionreg}). We construct minimizers $(m_k, w_k)$ of the regularized energy and associated solutions $(u_k, m_k)$ of the regularized version of \eqref{mfg}. Then, in order to come back to the initial problem, we provide some new a priori uniform $L^\infty$ bounds on $m_k$, which in turn imply a priori uniform bounds on $|\nabla u_k|$ and (local) H\"older regularity of $m_k$ that is uniform in $k$.  This key a priori bound is provided by Theorem \ref{stimalinfinito} 

 Note that we will consider classical solutions to this system (with a slight abuse of terminology), that is $(u,m)\in C^{2 }(\R^N)\times W^{1, p}(\R^N)$, for  all $p>1$. The existence result, proved in Section \ref{secex}, is the following. 
\begin{theorem} \label{exthm} Under the assumptions \eqref{Hass}, \eqref{assFlocal}, \eqref{alpha} and \eqref{vass}, for every $\eps>0$ there exists a classical solution $(u_\eps,m_\eps,\lambda_\eps)\in C^{2}(\R^N) \times W^{1,p}(\R^N)\times\R$, 
for all $p>1$, to \eqref{mfg}.  Moreover, $(m_\eps, -m_\eps \nabla H (\nabla u_\eps))$ is a minimizer in the set $\mathcal{K}_{\eps, M}$ of
the energy \eqref{energiaintro}.  
 \end{theorem} 
 
 We observe (see Remarks \ref{remarkene}, \ref{rem2}) that Theorem \ref{exthm} holds under 
 more general conditions on $H$ and  $f$, that is, 
if there exist $C_H$, $C_f>0$ and $K>0$ such that
\begin{equation}\label{assFlocal2}
C_H^{-1} |p|^{\gamma} - K \le H(p) 
\le C_H (|p|^{\gamma} + 1),\qquad 
 -C_fm^{\alpha}-K\leq  f( m)
\leq C_f^{-1}m^{\alpha}+K,\end{equation}
where  $\alpha$ satisfies \eqref{alpha}. 

In the second part of the work, in Section \ref{secco}, we analyze the behavior of the triple $(u_\eps, \lambda_\eps, m_\eps)$ coming from a minimizer of $\cE$ as $\eps \to 0$, under the assumptions \eqref{Hass}, \eqref{assFlocal}. From the viewpoint of the model, this amounts to remove the Brownian noise from the agents' dynamics. Heuristically, if the diffusion becomes negligible, one should observe aggregation of players (induced by the decreasing monotonicity of coupling in the cost) towards minima of the potential $V$, that are the preferred sites.  Moreover, in the case $V$ has a finite number of minima and polynomial behavior (that is, when \eqref{Vpoly} holds) we specialize the result showing that the limit procedure  selects the more stable minima of $V$, implying e.g. full convergence in the case that there exists a unique flattest minimum. 

In order to bring as much as possible information to the limit, we consider an appropriate rescaling of $m, u$, namely 
\begin{equation}\label{risca}
\bar m_\eps(\cdot)= \eps^{\frac{N \gamma'}{\gamma'-\alpha N}}  m(\eps^{\frac{\gamma'}{\gamma'-\alpha N}} \cdot+x_\eps), \quad
\bar{u} _\eps(\cdot) = \eps^{\frac{N \alpha(\gamma'-1)-\gamma'}{\gamma'-\alpha N}}\left({u} (\eps^{\frac{\gamma'}{\gamma'-\alpha N}}  \cdot+x_\eps)-{u} (x_\eps)\right),
\end{equation}
for all $\eps > 0$. The rescaling is designed so that $(\bar u_\eps, \bar m_\eps)$ solves a MFG system where the non-linearities have the same behavior of the original ones, i.e. $H_\eps \sim |p|^\gamma$ as $p \to \infty$, but the coefficient in front of the Laplacian is equal to one for all $\eps$, see \eqref{mfgresc1}. Moreover, the couple $\bar u_\eps, \bar m_\eps$ is associated to a minimizer of a rescaled energy $\cE_\eps$, see \eqref{energiarisc}. It turns out that in this rescaling process, the potential $V$ becomes
\[
V_\eps(\cdot) = \eps^{\frac{N\alpha \gamma'}{\gamma'-\alpha N}}V(\eps^{\frac{\gamma'}{\gamma'-\alpha N}} \cdot),
\]
and vanishes (locally) as $\eps \to 0$. Therefore, as one passes to the limit, the potential cannot compensate anymore the lack of compactness of $\R^N$, and the convergence of $\bar m_\eps$ in $L^1(\R^N)$ has to be proven by other methods. Heuristically, the aggregating force should be strong enough to overcome the dissipation effect, but the clustering point can be hard to predict by lack of spatial preference. This is why we also have to translate in \eqref{risca} by $x_\eps$. We will select $x_\eps$ to be the minimum of $u_\eps$: heuristically, being $u_\eps$ the value function, this is the point where most of the players should be located. In order to recover compactness for the sequence $\bar m_\eps$, we implement some ideas of the celebrated {\it concentration-compactness} method \cite{Lions84}. This principle states intuitively that if loss of compactness occurs, $\bar m_\eps$ splits in (at least) two parts which are going infinitely far away from each other, that is 
\begin{equation}\label{msplit}
\bar m_\eps \sim \chi_{B_R(0)} \bar m_\eps + \chi_{\R^N \setminus B_{2R}(0)} \bar m_\eps,
\end{equation}
with $R \to \infty$,  $\int \chi_{B_R(0)} \bar m_\eps \sim a$ and $\int \chi_{\R^N \setminus B_{2R}(0)} \bar m_\eps \sim M - a$ for some $a \in (0,M)$ (a third possibility might happen, but it is easily ruled out here by local estimates). This induces a splitting in the energy $\cE$, that is
\begin{equation}\label{split}
\inf_{\int m = M} \cE_\eps \gtrsim \inf_{\int m = a} \cE_\eps + \inf_{\int m = M-a} \cE_\eps.
\end{equation}
One then exploits a special feature of $\cE_\eps$, which is called sub-additivity:
\[
\inf_{\int m = M} \cE_\eps < \inf_{\int m = a} \cE_\eps + \inf_{\int m = M-a} \cE_\eps,
\]
that makes \eqref{split} impossible. While sub-additivity is easy to prove for $\cE_\eps$ (see Lemma \ref{subadd}), the splitting \eqref{split} requires technical work, in particular due to the presence of the term $mL(-w/m)$ in $\cE_\eps$, that becomes increasingly singular as $m$ approaches zero (a simple cut-off as in \eqref{msplit} is not useful). The property \eqref{split} is proven in Theorem \ref{nodicotomy}. It relies on the Brezis-Lieb lemma and a perturbation argument. 
The $L^1$ convergence of $\bar m_\eps$ enables us to obtain the full convergence of $(\bar u_\eps, \bar m_\eps)$ to a limit MFG system. By a uniform control of the decay of $\bar m_\eps$ as $|x| \to \infty$, that comes from a Lyapunov function built upon $\bar u_\eps$,  energy arguments and   the crucial $L^\infty$ estimate of Theorem \ref{stimalinfinito}, we are also able to keep track of  $x_\eps$. In terms of the non-rescaled density $m_\eps$, $x_\eps$ is the point around which most of the mass is located.

The second main result of this work is stated in the following two theorems. The first one is about concentration of $m_\eps$.
  \begin{theorem} \label{concthm} Under the assumptions of Theorem \ref{exthm}, there exist sequences $\eps \to 0$ and $x_\eps$,  such that for all $\eta>0$ there exists $R$ and $\eps_0$ for which for all $\eps<\eps_0$, 
\[
\int_{|x - x_\eps| \le  R\eps^{\frac{\gamma'}{\gamma'-\alpha N}}} m_\eps \, dx \ge M-\eta.
\]
Moreover,  $x_\eps\to \bar x$, where $V(\bar x)=0$, i.e. $\bar x$ is a minimum of $V$.

If, in addition, $V$ has the form
\begin{equation}\label{Vpoly}
V(x) = h(x) \prod_{j = 1}^{n} |x-x_j|^{b_j}, \qquad C_V^{-1} \le h(x) \le C_V \text{ on $\R^N$,}
\end{equation}
for some $x_j \in \R^N$, and $b_j>0$ (with $\sum_{j=1}^n b_j=b$), then $x_\eps\to x_i$,
with $i\in \{j=1,\dots, n \ | \ b_j=\max_{k} b_k\}$. 
 \end{theorem} 

Secondly, we describe the asymptotic profile of $(\bar u_\eps, \bar m_\eps)$ as $\eps \to 0$. Note that as a byproduct we obtain the existence of solutions to MFG systems without potential.  
 \begin{theorem}\label{fullconv}
Up to subsequences, $(\bar u_\eps, \bar m_\eps)$ converges in $C^1_{\rm loc}(\R^N) \times C_{\rm loc}(\R^N) \cap L^p(\R^N)$, for all $p \ge 1$, to a  solution $(\bar u, \bar m)$ of
\begin{equation}\label{mfglimit2}\begin{cases}
- \Delta u+ C_H |\nabla u|^\gamma+ \lambda = - C_f  m^\alpha\\
- \Delta  m -C_H\gamma \di( m  |\nabla  u|^{\gamma-2}\nabla  u)=0\\\int_{\R^N}  m=M. \end{cases}
\end{equation}
The function $\bar u$ is globally Lipschitz continuous on $\R^N$, and there exists $c_1, c_2>0$ such that $0<\bar m(x)\leq c_1 e^{-c_2|x|}$.

Finally, if $\bar w=-C_H\gamma \bar m  |\nabla \bar u|^{\gamma-2}\nabla  \bar u$, then
\begin{equation}\label{min2} 
 \mathcal{E}_0(\bar m,\bar w)=\min \left\{ \mathcal{E}_0(m,w) \ |\ (m,w)\in \mathcal{K}_{1, M},\ m(1+|y|^b)\in L^1(\R^N)\right\},
\end{equation}  where
\begin{equation} \label{energia0} \mathcal{E}_0(m,w)=\int_{\R^N}  C_L  \frac{|w|^{\gamma'}}{m^{\gamma'-1}} - \frac{1}{\alpha+1} m^{\alpha+1}dy.
\end{equation} 
\end{theorem}
We finally observe that by analogous methods, one can prove existence of solutions to more general potential-free MFG systems, see Remark \ref{remgs}.

\subsection*{Notation} 
We will intend a classical solution to the system \eqref{mfg} to be a
triple~$(u,   m, \lambda)\in C^{2}(\R^N) \times W^{1,p}(\R^N)\times\R$, for all $p>1$.\\
 For any given $p>1$, we will denote by $p'=\frac{p}{p-1}$ the conjugate exponent of $p$,  $p^*= \frac{Np}{N-p}$ if $p<N$ and $p^*=+\infty$ if $p\geq N$. \\ For all $R > 0$, $x \in \R^N$, $B_R(x) := \{y \in \R^N : |x-y| < R\}$. We will denote by $\omega_N := |B_1(0)|$.
 \\ Finally, $C,C_1,K,K_1,\ldots$ denote (positive) constants we need not to specify.

\medskip

{\bf Acknowledgements.} The authors are partially supported by the Fondazione CaRiPaRo Project ``Nonlinear Partial Differential Equations: Asymptotic Problems and Mean-Field Games'' and PRAT CPDA157835 of University of Padova ``Mean-Field Games and Nonlinear PDEs". % A.C. is partially supported by the INdAM-GNAMPA project ``Tecniche EDP, dinamiche e probabilistiche per lo studio di problemi asintotici".
 \section{Some preliminary regularity results}\label{s:preli}
Let $L$ be  the Legendre transform of $H$, i.e.
\begin{equation}\label{Leg}
L(q) = H^*(q) = \sup_{p \in \R^N} [p \cdot q- H(p)], \quad q \in \R^N.
\end{equation}
The assumptions on $H$ guarantee the following (see, e.g., \cite[Proposition 2.1]{c14}). 

\begin{proposition}\label{Lproperties} There exist $C_L, C_1, C_2>0$ depending on $C_H$ and on $\gamma$ such that  $\forall \ p, q \in \R^N$,
\begin{itemize}
\item[\textit{i)}] $L \in C^2(\R^N \setminus \{0\})$ and it is strictly convex,
\item[\textit{ii)}] $0\leq C_L|q|^{\gamma'} \le L(q) \le C_L(|q|^{\gamma'} + 1)$,
\item[\textit{iii)}] $\nabla L(q) \cdot q - L(q) \ge C_1|q|^{\gamma'} -C_1^{-1}$,
\item[\textit{iv)}] $ C_1q|^{\gamma'-1} -  C_1^{-1} \le |\nabla L(q)| \le  C_1^{-1}(|q|^{\gamma'-1} + 1)$.
\item[\textit{v)}]  $ C_2|p|^{\gamma-1} -  C_2^{-1} \le |\nabla H(p)| \le  C_2^{-1}(|p|^{\gamma-1} + 1)$.

\end{itemize}
\end{proposition}

We will use the following (standard) result on H\"older functions vanishing at infinity.

\begin{lemma}\label{van_lemma} Suppose that $m \ge 0$, $\|m\|_{C^{0,\theta}(\R^N)} \le c_h $, for some $\theta, c_h > 0$, and $\int_{\R^N} m \, dx < \infty$. Then, $m(x) \to 0$ as $|x| \to \infty$. Moreover, if
\[
\int_{|x| \ge R} m \, dx < \eta
\]
for some $\eta, R > 0$, then
\begin{equation}\label{maxcontrol}
\max_{|x| \ge R} m(x) \le C \eta^{\frac{\theta}{\theta+N}},
\end{equation}
where $C > 0$ depends only on $c_h, N$.
\end{lemma}

\begin{proof} By contradiction, suppose that there exists $\delta > 0$ and a sequence $|x_n| \to \infty$ such that $m(x_n) > \delta$ for all $n$. We may also assume that $|x_{n+1}| \ge |x_n| + 1$ for all $n$. By the H\"older regularity assumption,
\[
m(x) \ge m(x_n) - c_h |x-x_n|^{\theta} \ge \frac{\delta}{2},
\]
provided that $x \in B_r(x_n)$, and $r^\theta \le \frac{\delta}{2c_h}$. Choose $r = \min\{1, \left(\frac{\delta}{2c_h}\right)^{\frac{1}{\theta}}\}$, so that $B_r(x_n) \cap B_r(x_m) = \emptyset$ for all $n \neq m$. Then,
\[
\int_{\R^N} m \, dx \ge \sum_{n \in \mathbb{N}} \int_{B_r(x_n)} m \, dx \ge \sum_{n \in \mathbb{N}} \frac{\delta}{2}|B_r(0)| = +\infty
\]
that is impossible.

As for the second part, let $M := \max_{|x| \ge R} m(x) = m(\bar x)$, $|\bar x| \ge R$ (note that such a maximum is achieved as a consequence of the first part of the lemma). As before,
\[
m(x) \ge m(\bar x) - c_h |x-\bar x|^{\theta} \ge \frac{M}{2}
\]
for all $x \in B_r(\bar x)$, where $r = \left(\frac{M}{2c_h}\right)^{1/\theta}$. Therefore,
\[
\eta > \int_{|x| \ge R} m \, dx \ge \frac{M}{4} |B_r(\bar x)| = \frac{M}{4}|B_1(0)| \left(\frac{M}{2c_h}\right)^{N/\theta},
\]
and \eqref{maxcontrol} follows.
\end{proof}
We recall the following well known result, proved in \cite[Theorem 1]{bl}.
\begin{theorem}\label{teobrezis} 
Let $f_n\to f$ a.e. in $\R^N$ and assume that $\|f_n\|_{L^p(\R^N)} \leq C$ for all $n$ and for some $p\in [1, +\infty)$. 
Then
\[\lim_n [\|f_n\|_{L^p(\R^N)}^p- \|f_n-f\|_{L^p(\R^N)}^p]=\|f\|_{L^p(\R^N)}^p.\]
\end{theorem} 
From classical elliptic regularity, we have the following result.

\begin{proposition}\label{ell_regularity} Let $p > 1$ and $m \in L^p(\R^N)$ be such that
\[
\left| \int_{\R^N} m \, \Delta \varphi \, dx \right| \le K\| \nabla \varphi\|_{L^{p'}(\R^N)} \quad \text{for all $\varphi \in C^\infty_0(\R^N)$}
\]
for some $K > 0$. Then,  $m \in W^{1,p}(\R^N)$  and  there exists $C > 0$ depending only on $p$, such that
\[
\|\nabla m\|_{L^{p}(\R^N)} \le C \, K.
\]
\end{proposition}

\begin{proof} Fix any $R > 1$. Let $\psi \in C^\infty_0(B_2(0))$, $\varphi(R x) := \psi(x)$ (so, $\varphi \in C^\infty_0(B_{2R}(0))$) and $v(x) := m(Rx)$ on $\R^N$. Then,
\begin{multline*}
\left| \int_{B_2(0)} v \, \Delta \psi \, dx \right| = R^{2-N} \left| \int_{B_{2R}(0)} m \, \Delta \varphi \, dy \right| \le K R^{2-N} \left( \int_{B_{2R}(0)} |\nabla \varphi|^{p'} dy \right)^{1/{p'} }\\
= K R^{1-N+N/p'} \left( \int_{B_{2}(0)} |\nabla \psi|^{p'} dx \right)^{1/{p'} } \le K R^{1-N/p}\| \psi \|_{W^{1,p'}(B_2(0))}.
\end{multline*}

Hence, by \cite[Theorem 6.1]{agmon},  $v\in W^{1,p}(B_1(0))$ and there exists a constant $C$, depending on $p$ (but not on $R$), such that
\[
\|\nabla v \|_{L^{p}(B_1(0))} \le \| v \|_{W^{1,p}(B_1(0))} \le C(K R^{1-N/p} + \| v \|_{L^{p}(B_2(0))}).
\]
Therefore,
\begin{multline*}
\left( \int_{B_R(0)} |\nabla m|^p \, dy \right)^{1/p} = R^{N/p - 1}\left( \int_{B_1(0)} |\nabla v|^p \, dx \right)^{1/p} \le C\left[K +  R^{N/p - 1} \left( \int_{B_2(0)} |v|^p \, dx \right)^{1/p} \right] \\ = 
C (K + R^{-1} \| m \|_{L^{p}(B_{2R}(0))} ).
\end{multline*}
Letting $R \to \infty$, we get that $|\nabla m|\in L^p(\R^n)$ and the desired estimate.
\end{proof}

\subsection{The Hamilton-Jacobi-Bellman equation on the whole space}

In this section we provide some a priori regularity estimates and existence results for Hamilton-Jacobi-Bellman equations in the whole spaces of ergodic type. 
In particular we will consider families of Hamilton-Jacobi-Bellman equations
\begin{equation}\label{hjb2}
-\Delta u_n + H_n(\nabla u_n) + \lambda_n = F_n(x) -f_n(x)\qquad \text{on $\R^N$}
\end{equation}
where $F_n-f_n$ is  locally H\"older continuous,
$\lambda_n\in \R$ are equibounded in $n$, that is  $|\lambda_n|\leq \lambda$ and $f_n \in L^\infty(\R^N)$,  with $\|f_n\|_\infty\leq c_f$ for some $c_f>0$ independent of $n$. 
Moreover $ H_n$ is for every $n$ an Hamiltonian which satisfies  \eqref{Hass}, with constants $\gamma$ and $C_H$ independent of $n$; finally,  
there exists $C_F\geq 0$ and $b\geq 0$ independent of $n$ such that 
\begin{equation}\label{vass2}
C_F^{-1} (\max\{ |x|-C_F, 0\})^{b} \le F_n(x) \le C_F (1 + |x|)^{b} \qquad \forall n \text{ and } \forall x \in \R^N.
\end{equation}
Note that, differently from assumption \eqref{vass} for the potential $V$,  the function $F_n$ can also be bounded, if $b=0$. 

  \begin{theorem}\label{bernstein2} 
Let $u_n \in C^2(\R^N)$ be a sequence of classical  solutions of the HJB equations \eqref{hjb2}.
Then there exists a constant $K > 0$ depending on $C_H, C_F, c_f, \gamma, N, \lambda$  such that
\begin{equation}\label{berne}
|\nabla u_n(x)| \le K(1+|x|)^{\frac{b}{\gamma}},
\end{equation}
where $b\geq 0$ is the growth of $F_n$  appearing in \eqref{vass2} and $\gamma$ is the growth of $H_n$ appearing in \eqref{Hass}. 
 \end{theorem}

\begin{proof} Without loss of generality we may consider $H_n(p) = C_H|p|^\gamma$ for all $n$ and $p$. Indeed, every $v_n$ solves
\[
-\Delta u_n + C_H|\nabla u_n|^\gamma + \lambda_n = F_n(x) -f_n(x) + C_H|\nabla u_n|^\gamma - H_n(\nabla u_n)\qquad \text{on $\R^N$},
\]
and since $|C_H|\nabla u_n|^\gamma - H_n(\nabla u_n)| \le C_H$ by \eqref{Hass}, we can redefine $f_n$ to include $C_H|\nabla u_n|^\gamma - H_n(\nabla u_n)$, which then satisfies the bound $\|f_n\|_\infty\leq c_f + C_H$.  

We first claim that if $v \in C^2(B_2(0))$ satisfies
\[
|-\Delta v + C_H|\nabla v|^\gamma \,| \le k \quad \text{on $B_2(0)$}
\]
for some $k > 0$, then we have for any $r \in [1, \infty]$,
\begin{equation}\label{Lrest}
\|\nabla v\|_{L^r(B_1(0))} \le \widetilde C,
\end{equation}
where $\widetilde C$ depends only on $k, C_H, \gamma, N, r$. If $r \in [1, \infty)$, this is proven in \cite[Theorem A.1]{LL89}. The case $r = \infty$ follows by classical elliptic regularity, since if $r$ in \eqref{Lrest} is large enough, then $-\Delta v$ is bounded in $L^q(B_{3/2}(0))$ for some $q > N$, and the statement follows by Sobolev embeddings.

In view of these considerations, the gradient bound \eqref{berne} easily follows if $b = 0$. For the case $b > 0$, fix $x_0 \in \R^N$, and let $\delta = (1 + |x_0|)^{-b/\gamma'}$. Let \[v_n(y) := \delta^{\frac{2-\gamma}{\gamma-1}} u_n (x_0 + \delta y) \quad \text{on $\R^N$}.\]
Then, $v_n$ solves 
\[
-\Delta v_n + C_H |\nabla v_n|^\gamma  = \delta^{\gamma'}(F_n(x_0 + \delta y) -f_n(x_0 + \delta y)-\lambda_n).
\]
Since $\delta \le 1$,
\[
 \delta^{\gamma'} |F_n(x_0 + \delta y) -f_n(x_0 + \delta y)-\lambda_n| \le \frac{C_F (3 + |x_0|)^{b} + c_f + \lambda}{(1 + |x_0|)^b} \le C_1
\]
for all $y \in B_2(0)$ by \eqref{vass2} and the bound on $f_n$.

Therefore, by the first claim,
\[
\|\nabla v_n\|_{L^\infty(B_1(0))} \le \widetilde C,
\]
for all $n$. In particular, choosing $y = 0$,
\[
|\nabla u_n(x_0)| = \delta^{-\frac{1}{\gamma-1}} |\nabla v_n(0)| \le \widetilde C(1 + |x_0|)^{b/\gamma},
\]
and the desired estimate follows.
\end{proof}

Moreover, we prove the following a priori estimates on bounded from below solutions to \eqref{hjb2}. 
\begin{theorem}\label{boundedfrombelowsol} 
Let $u_n\in C^2(\R^N)$ be a family of uniformly  bounded from below classical solutions to \eqref{hjb2}, that is for which there exists $C>0$ such that $u_n\geq -C$ for every $n$. 

If $b=0$ in \eqref{vass2}, we moreover assume   that there exists $\delta>0$ and  $R>0$ independent of $n$
such that \begin{equation}\label{pos} F_n(x)-f_n(x)-\lambda_n>\delta>0, \qquad \text{for all $|x|>R$}.\end{equation}

Then there exists $C>0$ such that 
\begin{equation}\label{improvedgrowth}
u_n(x)\geq C|x|^{1+\frac{b}{\gamma}}-C^{-1}, \qquad \forall n \in \mathbb N, x\in \R^N, 
\end{equation}
where $b\geq 0$ is the growth power appearing in \eqref{vass2} and $\gamma$ is the growth power appearing in \eqref{Hass}. 
\end{theorem}
\begin{proof} The proof is based on the same argument as in \cite[Proposition 3.4]{bm}, we sketch it briefly for completeness.
Since $u_n$ is bounded from below we can assume $u_n\geq 0$, up to addition of constant $C$ (without changing the equation). 

We assume by contradiction that \eqref{improvedgrowth} does not hold. Then there exist  sequences $x_l$ and   $u_{n_l}$, such that $|x_l|>2R$, $|x_l|\to +\infty$, and 
$\frac{u_{n_l}(x_l)}{|x_l|^{1+\frac{b}{\gamma}}}\to 0$. Let $a_l= \frac{|x_l|}{2}$ and we define the function
\[v^l(x)=\frac{1}{a_l^{1+\frac{b}{\gamma}} }u_{n_l}(x_l+a_l x).\] By Theorem \ref{bernstein2}, we get that $|\nabla u_{n_l}(x)| \le K(1+|x|)^{\frac{b}{\gamma}}$. Therefore, $v^l, |\nabla v^l|$ are uniformly bounded.

Moreover, $v^l$ is a solution to 
\[- a_l^{ \frac{b}{\gamma}-1}\Delta v^l + H_{n_l}( a_l^{ \frac{b}{\gamma}} \nabla v^l)+\lambda_{n_l}= F_{n_l}(x_l+a_l x)- f_{n_l}(x_l+a_l x).\]
In particular, recalling \eqref{Hass}, we get that $v^l$ is a supersolution to 
\[- a_l^{ \frac{b}{\gamma}-1-b}\Delta v^l + C_H  |\nabla v^l|^\gamma \geq  a_l^{-b}\left( -\lambda_{n_l}+F_{n_l}(x_l+a_l x)- f_{n_l}(x_l+a_l x)\right).\]

Note that, for every $l$ sufficiently large,  by \eqref{vass2} and by \eqref{pos} (in the case $b=0$)
 the right hand side of the equation \[ a_l ^{-b}\left(  -\lambda_{n_l}+F_{n_l}(x_l+a_l x)- f_{n_l}(x_l+a_l x)\right)>0\] for  $x$ such that $|x|\leq 1$. 
 
 Moreover, passing eventually to a subsequence, we get that $v^l\to v$ locally uniformly in $n$ and  $a_l^{ \frac{b}{\gamma}-1-b}\to 0$. So 
  $v$ is a supersolution to  $C_H |\nabla v|^\gamma \geq \delta>0$ in $B(0, 1)$ with homogeneous boundary conditions (since $v\geq 0$). By comparison, recalling the explicit formula of the solution to the eikonal equation $|\nabla f|^\gamma= C$ in $B(0,1)$ with homogeneous boundary conditions, we conclude that $v(x)\geq C^{\frac{1}{\gamma}} (1-|x|)$ for all $x$ such that $|x|\leq 1$. Moreover, by uniform convergence, we get that, eventually enlarging $C$ and taking $l$ sufficiently large, $v^l(x)\geq C^{\frac{1}{\gamma}} (1-|x|)$ for all $x$ with $|x|\leq 1$, in particular $v^l(0)\geq  C^{\frac{1}{\gamma}} $. 
 Recalling the definition of $v^l$, we get that $v^l(0)\to 0$,  which yields a contradiction. 
\end{proof}

Define  \[
\bar{\lambda}_n := \sup\{\lambda \in \R : \text{\eqref{hjb2} has a  solution $u_n \in C^2(\R^N)$}\}.
\]
\begin{theorem}\label{principal_eig}
Assume that for every $n$ the function $F_n-f_n$ is bounded from below uniformly in $n$. 
\begin{enumerate} \item    
 $\bar{\lambda} _n< \infty$, for every $n$,  and there exists, for every $n$, a solution $u_n \in C^2(\R^N)$ to 
\eqref{hjb2} with $\lambda_n = \bar{\lambda}_n$. Moreover \[\bar{\lambda}_n := \sup\{\lambda \in \R : \text{\eqref{hjb2} has a  subsolution $u_n \in C^2(\R^N)$}\}.\]

\item  If $F_n$ satisfies \eqref{vass2}, with $b>0$, then, for every $n$, the solution $u_n$ to \eqref{hjb2} with $\lambda_n = \bar{\lambda}_n$ is unique up to addition of constants and satisfies 
\eqref{improvedgrowth}. 
\item If $F_n\equiv 0$, and there exists $\delta>0$ independent of $n$ such that \begin{equation}\label{positiva}\limsup_{ |x|\to +\infty} f_n(x)+\bar \lambda_n <-\delta<0,\end{equation}  then for every $n$ there exists a solution to \eqref{hjb2} 
with $\lambda_n = \bar{\lambda}_n$ which satisfies \eqref{improvedgrowth} with $b=0$.  
\end{enumerate} 
\end{theorem}

\begin{proof} 

(i). The proof of this result can be obtained by a straightforward adaptation of the proof of Theorem 2.1 in \cite{bm}, using the a priori estimates on the gradient given in 
Theorem \ref{bernstein2}.  Observe that actually in \cite{bm} it is required a stronger assumption on the regularity of $F_n-f_n$, in particular local Lipschitz continuity. 
This assumption is used to derive a priori estimates on the gradient of solutions by using the so called Bernstein method (see Appendix A in \cite{bm}), which depends also on the 
$L^\infty$ norm of $\nabla (F_n-f_n)$.  In our case we can weaken this assumption to just H\"older continuity (so still ensuring classical elliptic regularity)
 since we are using a priori estimates on the gradient given in Theorem \ref{bernstein2}, which depends only on the $L^\infty$ norm of $F_n-f_n$, and are obtained in \cite{LL89} by
  the so called integral Bernstein method. 
 
(ii). For the proof we refer to   \cite{ichi11} (see also \cite{bm} and \cite{c14}).  In particular in \cite{ichi11}, it is proved that $u_n$ is bounded from below.
By looking at the proof, it is easy to check that, due to the uniformity in $n$ of the norms of coefficients, the bound can be taken independent of $n$,  and by Theorem
\ref{boundedfrombelowsol} we get the estimate on the growth. 

(iii).  By adapting the argument in \cite[Theorem 2.6]{bm}, we get that there exists a bounded from below solution to \eqref{hjb2} with $\lambda_n= \bar{\lambda}_n$, with bound uniform in $n$. Then using Theorem
\ref{boundedfrombelowsol}, we get the estimate on the growth. 
We give a brief sketch of the proof of the existence of a bounded from below solution. 
For every $R>0$, we consider the ergodic problem
 \begin{equation}\label{sc} \begin{cases} -\Delta u_n^R+H_n(\nabla u_n^R) +\lambda_n^R= -f  & |x|<R \\ 
 u_n^R(x)\to +\infty & |x|\to R.\end{cases}
 \end{equation} 
Using the result  in \cite{bpt}, we get that for every $R>0$ there exists a unique $\lambda_n^R$ and a unique up to addition of constant solution $u_n^R\in C^2(B_R)$.

First of all we claim that $\lim_R\lambda_n^R=\bar\lambda_n$. 
 It is easy to check that if $R'>R$, then $\lambda_n^{R'}\leq \lambda_n^R$, and moreover that 
 $\lambda_n^R\geq \bar\lambda_n$. So, the sequence $\lambda_n^R$ is converging as $R\to +\infty$ to some
 $\lambda^\star_n\geq \bar\lambda_n$. Moreover, by the same argument as in Theorem \ref{bernstein2}, we get that for every compact $K\subset\R^N$,
 there exists a constant $C>0$ such that $|\nabla u_n^R|\leq C$ in $K$ for every $R$ sufficiently large and for all $n$.
  Without loss of generality we can assume that $u_n^R(0)=0$ for every $R$. So, using the gradient bound, 
  and elliptic regularity, we conclude that $u_n^R$ is bounded in $C^2(K)$ by some constant independent of $R$. 
  Hence, by Ascoli-Arzel\`a Theorem, and via a diagonalization procedure, we get that $u_n^R$ converges locally in $\R^N$, with $u_n\in C^2(\R^N)$. Moreover, $u_n$ is a solution to \eqref{hjb2}, with $\lambda=\lambda_n^\star$. Recalling the characterization of $\bar\lambda_n$ 
 and the fact that $\lambda^\star_n\geq \bar\lambda_n$, we conclude that $\lambda^\star_n=\bar\lambda_n$. 
 
Then, we consider $x_n^R\in B_R$ such that $u_n^R(x_n^R)=\min_{|x|\leq R} u_n^R$. Recalling that $u_n^R$ is a solution to \eqref{sc}, we get by computing the equation at $x_n^R$ and by recalling that $H_n(0)\leq 0$, that 
\[\lambda_n^R+f(x_n^R)\geq H_n(0)+\lambda_n^R+f(x_n^R)\geq 0.\] Using condition \eqref{positiva}, and recalling that $\lambda_n^R\to\bar \lambda_n$, 
we get that there exists a compact set $K$ (independent of $R$ and of $n$) 
and $R_0>0$ such that for all $R>R_0$, $x_n^R\in K$. 

Recalling that $u_n^R(0)=0$ and  $|\nabla u_n^R|\leq C$ in $K$ with $C$ independent of $n, R$, we conclude that $u_n^R(x_R)\geq -C$ for some constant $C$ independent of $n, R$. 
 But, this implies, since $u_n^R(x)\geq u_n^R(x_n^R)$ 
for every $R$, that passing to the limit $u_n(x)\geq -C$, with $C$ independent of $n$. 
 \end{proof}

\subsection{A priori estimates for the Kolmogorov equation}\label{s:kolmo_reg}
In this section we provide general a priori estimates for couples $(m,w)\in  (L^1(\R^N)\cap  W^{1,q}(\R^N)) \times L^1(\R^N)$ such that 
$\int_{\R^N} m(x)=M$ and $-\eps\Delta m+\di w=0$ where
\begin{equation}\label{definizioneq} 
q=\begin{cases} \gamma' & \gamma'\geq N\\ 
\frac{N}{N-\gamma'+1} & \gamma'<N.
\end{cases} 
\end{equation} 

\begin{lemma}\label{lemma1}  Let $\beta\leq \frac{Nq}{N-q}$, for $q<N$, and $\beta<+\infty$ for $q\geq N$. 
We define $1\leq r\leq \beta $ as follows
\begin{equation}\label{r} \frac{1}{r}= \frac{1}{\gamma'}+ \left(1-\frac{1}{\gamma'}\right)\frac{1}{\beta}.
\end{equation} 
Then,  there exists a constant $C$, depending only on $N$ and $\beta$, such that
\begin{eqnarray}\label{stimam}
\|m\|_{W^{1,r}(\R^N)}&\leq &C \left(\frac{1}{\eps^{\gamma'}}\int_{\R^N} m  \left|\frac{w}{m}\right|^{\gamma'} \, dx  + M\right)^{\frac{1}{\gamma'}}\|m\|_{L^\beta(\R^N)}^{\frac{1}{\gamma}}
\\ \nonumber &\leq & C  \left( \frac{C_L}{ \eps^{\gamma'}}\int_{\R^N} m L\left(-\frac{w}{m}\right) \, dx+ M \right)^{\frac{1}{\gamma'}}\|m\|_{L^\beta(\R^N)}^{\frac{1}{\gamma}},
\end{eqnarray} 
where $C_L = C_L(C_H, \gamma)$ is the constant appearing in  Proposition \ref{Lproperties}. 

We now assume that
 \begin{equation} \label{beta}
1<\beta< 1+\frac{\gamma'}{N}.
\end{equation} 
Then, there exists $\delta>0$   such that  \begin{equation}\label{stimabeta}
\|m\|_{L^\beta(\R^N)}^{(1+\delta)\beta}\leq C \frac{1}{\eps^{\gamma'}}  M^{ (1+\delta)\beta-1}     \left(\int_{\R^N} m 
\left|\frac{w}{m}\right|^{\gamma'}dx\right)\leq  C C_L \frac{1}{\eps^{\gamma'} }  M^{ (1+\delta)\beta-1}    \int_{\R^N}m L\left(-\frac{w}{m}\right) \, dx, 
\end{equation}
where the constant $C$ depends only on $\gamma$, $N$, and  $\beta$. 
 \end{lemma} 
\begin{proof} Since $m\in W^{1,q}(\R^N)$, by Sobolev embedding and interpolation, we get that  $m\in L^\beta(\R^N) $.   
Using  $-\eps \Delta m+\di w=0$, we get for all $\varphi \in C_0^\infty(\R^N)$, 
\[\eps \int_{\R^N} \nabla m \cdot \nabla \varphi \, dx = \int_{\R^N} w \cdot \nabla \varphi \, dx .\]
Using Holder inequality, recalling \eqref{r},  we obtain
\begin{multline*}\left|\frac{1}{\eps}\int_{\R^N} w\cdot \nabla \varphi \, dx\right|\leq \int_{\R^N} \frac{1}{\eps }\left|\frac{w}{m}\right| m^{\frac{1}{\gamma'}} m^{1-\frac{1}{\gamma'}} |\nabla\varphi| dx
\\ \leq \left(\frac{1}{\eps^{\gamma'}}\int_{\R^N} m \left|\frac{w}{m}\right|^{\gamma'}dx\right)^{\frac{1}{\gamma'}} \|m\|_{L^\beta(\R^N)}^{\frac{1}{\gamma}} \|\nabla\varphi\|_{L^{r'}(\R^N)}.\end {multline*}
Therefore, we get that for all $\varphi \in C_0^\infty(\R^N)$, 
\[\left| \int_{\R^N} \nabla m \cdot \nabla \varphi \, dx\right| \leq  \left(\frac{1}{\eps^{\gamma'}}\int_{\R^N} m \left|\frac{w}{m}\right|^{\gamma'}dx\right)^{\frac{1}{\gamma'}} \|m\|_{L^\beta(\R^N)}^{\frac{1}{\gamma}} \|\nabla\varphi\|_{r'} .\]
We apply then Proposition \ref{ell_regularity} and we obtain that $m\in W^{1,r}(\R^N)$ and that there exists a constant $C$, depending only on $r$, such that
\begin{equation}\label{c1} \|\nabla m\|_{L^{r}(\R^N)} \le C    \left(\frac{1}{\eps^{\gamma'}}\int_{\R^N} m \left|\frac{w}{m}\right|^{\gamma'}dx\right)^{\frac{1}{\gamma'}} \|m\|_{L^\beta(\R^N)}^{\frac{1}{\gamma}}.
\end{equation} 
From this inequality, using Proposition \ref{Lproperties}  and recalling that by interpolation, since $\|m\|_{L^1(\R^N)}=M$,
 $\|m\|_{L^r(\R^N)}\leq \|m\|_{L^\beta(\R^N)}^{\frac{1}{\gamma}}M^{\frac{1}{\gamma'}}$, we conclude the desired inequality \eqref{stimam}. 

Now we fix $\eta$ such that 
\[\frac{1}{\eta}=\left(\frac{1}{r}-\frac{1}{N}\right)\frac{N}{N+1}+1-\frac{N}{N+1}= \frac{N}{N+1}\frac{1}{r}.\]
Note that, by a simple computation using \eqref{r}, we get  $\frac{1}{\eta}-\frac{1}{\beta}= \frac{N}{N+1} \frac{1}{\beta\gamma'}\left(\beta-1-\frac{\gamma'}{N}\right)$, therefore, by
 \eqref{beta}, we conclude that
that $\eta>\beta$.
By Gagliardo Nirenberg inequality, and recalling that $\|m\|_1=M$, we get
\begin{equation}\label{c2}
\|m\|_{L^\eta(\R^N)}\leq C \|\nabla m\|_{L^r(\R^N)}^{\frac{N}{N+1}}M^{\frac{1}{N+1}}.
\end{equation} 
Since $\eta>\beta$, by interpolation we get that there exists $\theta>1$ such that $\|m\|_{L^\beta(\R^N)}^\theta\leq \|m\|_{L^\eta(\R^N)} M^{\theta-1}$. 
Actually \[\frac{1}{\theta}= \left(1-\frac{1}{\beta}\right) (N+1) \frac{1}{1+N \left(1-\frac{1}{\beta}\right) \left(1-\frac{1}{\gamma'}\right)}.\]
So, we substitute in \eqref{c2} and \eqref{c1} and we get, elevating both terms to $\gamma'\frac{N+1}{N}$, 
\begin{equation}\label{c3}
\|m\|_{L^\beta(\R^N)}^{\theta \gamma' \frac{N+1}{N}}\leq C \frac{1}{\eps^{\gamma'} }  M^{\gamma'(\theta\frac{N+1}{N}-1)}   \left(\int_{\R^N} m 
\left|\frac{w}{m}\right|^{\gamma'}dx\right) \|m\|_{L^\beta(\R^N)}^{\frac{\gamma'}{\gamma}}.
\end{equation}
Now, since $\theta>1$, by \eqref{beta}, we get \[\theta \gamma' \frac{N+1}{N}-\frac{\gamma'}{\gamma}= \frac{\beta\gamma'}{N(\beta-1)}=\beta+\frac{\beta}{\beta-1} \left[\frac{\gamma'}{N}+1-\beta \right]>0.\] 
Therefore we deduce \eqref{stimabeta} from \eqref{c3} with \begin{equation}\label{delta} \delta=\frac{1}{\beta-1} \left[\frac{\gamma'}{N}+1-\beta \right].\end{equation}  
 \end{proof} 
 
\begin{corollary} \label{calpha} For every $r<q$, there exists $C>0$ depending on $N$, $\gamma'$  and $r$ such that     \begin{equation}\label{stimaw1q}  \|m\|_{W^{1,r}(\R^N)} \leq \frac{C}{\eps^{\gamma'}}     \left(C_L\int_{\R^N} m L\left(-\frac{w}{m}\right) \, dx+\eps^{\gamma'} M\right).\end{equation}

Moreover, if $\gamma'>N$ (so $q>N$), then   $m\in C^{0,\theta}(\R^N)$ and
\begin{equation}\label{stimaholder}  \|m\|_{C^{0,\theta}(\R^N)} \leq \frac{C}{\eps^{\gamma'}}     \left(C_L\int_{\R^N} m L\left(-\frac{w}{m}\right) \, dx+\eps^{\gamma'} M\right).\end{equation}
\end{corollary}  
\begin{proof} 
For $q\geq  N$ (equivalently $\gamma'\geq N$), we fix $r<q$ and we choose $\beta$ which satisfies  \eqref{r} for such $r$. 
By Sobolev embedding theorem, $W^{1,r}(\R^N)$ is continuously embedded in $L^\beta(\R^N)$.
So, there exists $C$ depending on $N$ and $r$ such that $\|m\|_{L^\beta(\R^N)}\leq C \|m\|_{W^{1,r}(\R^N)}$.
Using inequality \eqref{stimam}, we get 
\[\|m\|_{L^\beta(\R^N)} 
\leq \frac{C}{\eps^{\gamma'}}   \left(\int_{\R^N} m  \left|\frac{w}{m}\right|^{\gamma'} \, dx   +\eps^{\gamma'}  M\right).\]
If we substitute again in \eqref{stimam} we get 
\[\|m\|_{W^{1,r}(\R^N)}\leq \frac{C}{\eps^{\gamma'}} 
\left(\int_{\R^N} m  \left|\frac{w}{m}\right|^{\gamma'} \, dx  +\eps^{\gamma'} M\right). \]
In particular for $q>N$, we can choose $r>N$ and by Sobolev embedding theorem  we get that there exists $\theta=1-\frac{N}{r}$ and a constant $C>0$ depending on $N$ and $r$ such that 
\begin{eqnarray*} \|m\|_{C^{0,\theta}(\R^N)}&\leq & \frac{C}{\eps^{\gamma'}} 
\left(\int_{\R^N} m  \left|\frac{w}{m}\right|^{\gamma'} \, dx  +\eps^{\gamma'} M\right)
\\ &\leq &  \frac{C}{\eps^{\gamma'}}     \left(C_L\int_{\R^N} m L\left(-\frac{w}{m}\right) \, dx  +\eps^{\gamma'} M\right).
\end{eqnarray*}

For $q<N$, we fix $r<q$, and choose the corresponding $\beta$ in \eqref{r}, that satisfies $\beta<\frac{N}{N-\gamma'}$. Hence
we conclude again from inequality \eqref{stimam}. 

\end{proof}

%\begin{corollary}\label{solpde} 
%Assume that $(u,m)\in W^{1,\gamma}(\R^N) \times (W^{1,\gamma' }(\R^N)\cap L^1(\R^N))$ solves the Kolmogorov equation in \eqref{mfg}. 
%For any $\beta$ satisfying \eqref{beta}, there exists $\delta>0$ such that 
%\[ \|m\|_{L^\beta(\R^N)}^{(1+\delta)\beta}\leq C \frac{1}{\eps^{\gamma'}}   \left(\int_{\R^N} m 
%|\nabla u|^\gamma dx\right) \]
%where $C$ is a constant which depends only on $N, \beta, \gamma$. 
%\end{corollary} 
%\begin{proof} We apply  Lemma \ref{lemma1} with $w= m \nabla H(\nabla u)$. Note that $|w|\leq m |\nabla H(\nabla u)|\leq C_H m (|\nabla u|^{\gamma-1}+ 1)$.
%So, since $|\nabla u|\in L^\gamma(\R^N)$ and $m\in L^{\gamma'}(\R^N)$ we obtain $|w|\in L^{\gamma'}(\R^N)$. 
%\end{proof} 

\section{Regularization procedure and existence of approximate solutions for $\eps>0$} \label{sectionreg}
\medskip
\subsection{The regularized problem} 
We consider the following approximation of the system \eqref{mfg},

\begin{equation}\label{mfgeps}\begin{cases}
- \eps\Delta u + H(\nabla u )+\lambda  =f_k[m](x)+V(x),\\
-\eps \Delta  m -\di(m  \nabla H(\nabla u ) )=0, \\   \int_{\R^N} m \, dx=M, \end{cases}
\end{equation}
where
\begin{equation}\label{regf} f_k[m](x)= f(m \star\chi_k)\star \chi_k (x)=\int_{\R^N} \chi_k(x-y) f\left (\int_{\R^N} m (z)\chi_k(y-z)dz\right)dy\end{equation} 
and $\chi_k$, for $k>0$,  is a sequence of standard symmetric mollifiers approximating the unit as $k \to \infty$.

We observe that $f_k[m](x)$ is the $L^2$-gradient of a $C^1$  potential $F_k: L^1(\R^N)\to \R$, defined as follows  
\begin{equation}\label{datamol}  F_k[m]:= \int_{\R^N} F(m\star \chi_k(x)) dx,
\end{equation}  
where $F(m)=\int_0^m f(n)dn$ for $m\geq 0$ and $F(m)=0$ for $m\leq 0$. 
Note that using Jensen inequality and \eqref{assFlocal}, we get that  for all $m\in L^1(\R^N)$ such that $m\geq 0$, and $\int_{\R^N}m(x)dx=M$, 
\begin{equation}\label{cond1} -\frac{C_f}{\alpha+1} \int_{\R^N} m^{\alpha+1}(x)dx-KM\leq F_k[m]\leq  -\frac{C_f}{\alpha+1} \int_{\R^N} \left(m\star\chi_k(x)\right)^{\alpha+1}\, dx+KM.\end{equation} 

In order to construct solutions to the system, we follow a variational approach and we associate to \eqref{mfgeps} a energy, as already described in the introduction. We define the energy \begin{equation}\label{energia}
\mathcal{E}_k(m, w) := \begin{cases}
\displaystyle \int_{\R^N} m L\left(-\frac{w}{m}\right) + V(x)m\, dx + F_k[m]  & \text{ if $(m,w) \in \mathcal{K}_{\eps, M}$}, \\
+\infty & \text{otherwise},
\end{cases}
\end{equation}
where $\mathcal{K}_{\eps,M}$ is defined in \eqref{kcalconstraint} and  $L$ is defined in \eqref{dati}.  We recall that the exponent $q$ appearing in the definition of   $\mathcal{K}_{\eps,M}$ is \[ q=\begin{cases} \frac{N}{N-\gamma'+1} & \gamma'\leq N \\ \gamma' & \gamma'>N.\end{cases}\] Therefore, $q\leq \gamma'$.  Observe that, if $q<N$,  $q^*=\frac{qN}{N-q}= \frac{N}{N-\gamma'}$, and that $q^*>1+\frac{\gamma'}{N}>1+\alpha$ by \eqref{alpha}. 
If $q=\gamma'\geq N$, then we let $q^*=+\infty$.

\subsection{A priori estimates and energy bounds} 
In this section, we provide bounds from  below for the energy $\cE_k$, assuring in particular that the minimum problem is well defined.

\begin{lemma}\label{boundfrombelow} Let $(m,w)\in \cK_{\eps, M}$.  
Then 
\begin{equation}\label{energybound1}\cE_k(m, w) \geq  -K-C  \eps^{-\frac{\gamma'\alpha N}{ \gamma'-\alpha N}}  \end{equation}
where $C, K >0 $ are constants depending only on $N, M, C_L, \gamma,\alpha, M$.

In particular there exists finite
\[e_{k,\eps}(M)=\inf_{(m,w) \in\cK_{\eps, M}}\cE_k(m, w). \]
\end{lemma}
\begin{proof}
Recalling that $V\geq 0$,  estimate \eqref{cond1} and applying   \eqref{stimabeta} with $\alpha=\beta-1$, we get  
\begin{eqnarray*}
\cE_k(m, w) &\geq& \int_{\R^N} m L\left(-\frac{w}{m}\right) \, dx   - \frac{C_f}{\alpha+1} \int_{\R^N} m^{\alpha + 1} \, dx-KM \\
& \geq & C \eps^{\gamma'}  M^{1-(1+\delta)(1+\alpha)}\|m\|_{L^{\alpha+1}}^{(1+\alpha)(1+\delta)}  -\frac{1}{\alpha+1} \|m\|_{L^{\alpha+1}}^{(1+\alpha)}-KM \nonumber \\
&\geq &   -C\delta \eps^{-\frac{\gamma'}{\delta}} \left(\frac{1}{(\delta+1)(\alpha+1)}\right)^{1+\frac{1}{\delta}} -KM 
\end{eqnarray*}
where $C$ is a constant depending only on $N, M, C_L, \gamma,\alpha$ and 
\begin{equation}\label{contodelta} \delta=\frac{1}{\alpha} \left[\frac{\gamma'}{N}-\alpha \right]. \end{equation}
Therefore, substituting in the energy, we get 
\[\cE_k(m, w) \geq -C\frac{(\gamma'-\alpha N)}{\alpha N}  \eps^{-\frac{\gamma'\alpha N}{ \gamma'-\alpha N}} \left(\frac{\alpha N}{\gamma'(\alpha+1)}\right)^{\frac{\gamma'}{\gamma'-\alpha N}}-KM, \]
which gives the desired inequality.
\end{proof}

We get also a priori bounds on minimizers and minimizing sequences. 
\begin{proposition}\label{aprioriest} 
Let $(m,w)\in  \cK_{\eps, M}$ such that $e_{k,\eps}(M)\geq \cE_k(m,w)-\eta$, for some positive $\eta$. 
Then 
\begin{eqnarray}\label{risc1} \int_{\R^N} m \left| \frac{w}{m}\right|^{\gamma'} dx \leq C\eps^{-\frac{\gamma'N\alpha}{\gamma'-N\alpha}} +K,\\
\label{risc2} \|m\|_{L^{\alpha+1}(\R^N)}^{\alpha+1}\leq C\eps^{-\frac{\gamma'N\alpha}{\gamma'-N\alpha} }+K,\end{eqnarray}

for some   $C, K$ positive constants which depends only on $\alpha, N, V, C_L$. 
\end{proposition} 
\begin{proof} First of all we observe that there exists $C\geq 0$ depending   on $M, C_L, C_V$ such that 
\begin{equation}\label{stimaalto1} e_{k,\eps}(M)\leq C.\end{equation} 
Let $m=c e^{-|x|}$, where $c$ is chosen to have $\int_{\R^n} mdx=M$, 
and $w=\eps \nabla m$, so that $(m,w)\in \mathcal{K}_{\eps, M}$. By assumption \eqref{vass}, we get that $\int_{\R^n} mV(x) dx\leq C$ for some constant $C>0$, by \eqref{cond1} that $F_k[m]\leq KM$ and   by the properties of $L$ in Proposition \ref{Lproperties}, we have that $\int_{\R^n} mL(-w/m)dx\leq( \frac{\eps^{\gamma'}}{c^{\gamma'}}+C_L)M $. So, in conclusion $e_{k,\eps}(M)\leq \mathcal{E}_k(m,w)\leq C$ as required.

Note that if $(m,w)\in  \cK_{\eps, M}$, and $e_\eps(M)\geq \cE(m,w)-\eta$, for some positive $\eta$, then, by \eqref{cond1},  by the fact that  $V\geq 0$, and by the properties of $L$ in Proposition \ref{Lproperties}, we get 
\begin{equation} \label{energia1} 
C+\eta\geq e_\eps(M)+\eta\geq  \cE_k(m,w) \geq \int_{\R^N}  m\left|\frac{w}{m}\right|^{\gamma'}-\frac{C_f}{\alpha+1} m^{\alpha+1} \, dx\ -KM. \end{equation} 
We apply  \eqref{stimabeta} with $\alpha=\beta-1$, and we obtain   
\begin{multline*}
C+\eta+KM\geq \int_{\R^N}  m\left|\frac{w}{m}\right|^{\gamma'}-\frac{C_f}{\alpha+1} m^{\alpha+1} \, dx\\
 \geq  C \eps^{\gamma'}  M^{1-(1+\delta)(1+\alpha)}\|m\|_{L^{\alpha+1}}^{(1+\alpha)(1+\delta)}  -\frac{C_f}{\alpha+1} \|m\|_{L^{\alpha+1}}^{(1+\alpha)}.  \end{multline*}
 Recall that  $\delta+1=\frac{\gamma'}{\alpha N}$ (can be computed using \eqref{delta}), so $\frac{\gamma'}{\delta}=  \frac{\gamma'N\alpha}{\gamma'-N\alpha}$. 
Note that if we choose $A$ sufficiently large (depending on $\delta, M, C_f, C_L)$, we get that 
\[C \eps^{\gamma'}  M^{1-(1+\delta)(1+\alpha)}(\eps^{-\frac{\gamma'}{\delta}}A)^{1+\delta}  -\frac{C_f}{\alpha+1}( \eps^{-\frac{\gamma'}{\delta}}A)\geq C+\eta+KM,\]
from which we conclude that $\|m\|_{L^{\alpha+1}}^{(1+\alpha)}\leq \eps^{-\frac{\gamma'}{\delta}}A$, and so estimate \eqref{risc2} holds. 
Estimate \eqref{risc1} comes from \eqref{risc2} and \eqref{energia1}. 
\end{proof}
 
\subsection{Existence of a solution} 
We are now in the position to show existence of minimizers of the energy $\cE_k$  in the class  $\cK_{\eps, M}$ for every $\eps, M > 0$.

\begin{proposition}\label{existence} For every $\eps>0$ and $M > 0$, there exists a minimizer $(m_k, w_k)\in \cK_{\eps, M}$ of  $\cE_k$, that is
\[\cE_k(m_k,w_k)=\inf_{(m,w)\in \cK_{\eps, M}}\cE_k(m,w).\] 
Moreover,  for every minimizer $(m_k,w_k)\in \cK_{\eps, M}$ of  $\cE_k$, there holds 
\begin{equation}\label{moreintegr}
m_k (1 + |x|)^{b} \in L^1(\R^N), \quad w_k(1 + |x|)^{b/\gamma} \in L^{1}(\R^N),
\end{equation}
and there exist constants $C>0$ and $K$, independent of $\eps$ and $k$, such that 
 \begin{equation}\label{intV} 
\int_{\R^N} m_k \left| \frac{w_k}{m_k}\right|^{\gamma'} dx +\int_{R^N} m_k V(x)\, dx+ \|m_k\|_{L^{\alpha+1}(\R^N)}^{\alpha+1} \leq C \eps^{-\frac{\gamma'\alpha N}{\gamma'-N\alpha}}+K.
\end{equation} 

\end{proposition} 

\begin{proof} Let $(m_n, w_n)\in \cK_{\eps, M}$ be a minimizing sequence, that is  $\cE_k(m_n, w_n)\to e_{k,\eps}(M)$. 
This implies that, choosing $n$ sufficiently large, $\cE_k(m_n, w_n)\leq e_\eps(M)+1$. 
From this and \eqref{cond1} we get
\begin{multline}\label{eq20} \int_{\R^N} m_n L\left(-\frac{w_n}{m_n}\right) \, dx + \int_{\R^N} V(x) m_n \, dx \leq \cE_k(m_n, w_n)+  \frac{C_f}{\alpha+1} \int_{\R^N} m_n^{\alpha + 1} \, dx +KM\\ \leq e_{k,\eps}(M)+1+ \frac{C_f}{\alpha+1} \int_{\R^N} m_n^{\alpha + 1}+KM.\end{multline}
%Using \eqref{stimabeta}\label{mLbound} with $\beta=\alpha+1$ and $\delta$ as in \eqref{contodelta},   we get $\frac{1}{1+\delta}= \frac{\alpha N}{\gamma'}<1$ and 
%\begin{equation} \int_{\R^N} m_n L\left(-\frac{w_n}{m_n}\right) \, dx  \leq e_\eps(M)+1+ \frac{1}{\alpha+1} \left[ CL_0^{-1} \frac{1}{\eps }  M^{\delta'}   \left(\int_{\R^N}m_n L\left(-
%\frac{w_n}{m_n}\right) \, dx+M\right)\right]^{\frac{\alpha N}{\gamma'}}. \end{equation}
%Recalling that $V\geq 0$, this implies that $\int_{\R^N} m_n L\left(-\frac{w_n}{m_n}\right)\, dx $ is bounded by a constant $C_\eps$ depending on $\eps, e_\eps(M), M, \gamma', N$.

By Proposition \ref{aprioriest}, we get  that 
\[\|m_n\|_{L^{\alpha+1}} +\int_{\R^N} m_n^{1-\gamma'} |w_n|^{\gamma'}dx\leq C\eps^{-\frac{\gamma'\alpha N} {\gamma'-\alpha N}}+K.\]
We conclude also that 
\[\int_{\R^N} V(x)m_n(x)dx\leq C\eps^{-\frac{\gamma'\alpha N} {\gamma'-\alpha N}}+K,\] for some $C,K>0$. 
These estimates will imply \eqref{intV}, after passing to the limit,   using Fatou lemma.
 
Moreover, by Corollary \ref{calpha}, we have that there exists $C_\eps>0$ depending on $\eps$ such that  for all $r<q$, 
\[\|m_n\|_{W^{1,r}(\R^N)}\leq C_\eps.\] Moreover, due to Sobolev embeddings, we get that for all $s<q^*$, then $\|m_n\|_{L^s(\R^N)}\leq C_\eps$. 
In addition, by applying Holder inequality, we get that there exists $C>0$
\[ \int_{\R^N} |w_n|^{\frac{\gamma'\alpha+\gamma'}{\gamma'+\alpha}}dx \leq C \left( \int_{\R^N} m_n^{1-\gamma'} |w_n|^{\gamma'}dx\right)^{\frac{\alpha+1}{\gamma'+\alpha}} 
\|m_n\|_{L^{\alpha+1}(\R^N)}^{\frac{\gamma'-1}{(\alpha+1)(\gamma'+\alpha)}}. \]

By  these estimates and Sobolev compact embeddings, we get that eventually extracting a subsequence via a diagonalization procedure, 
$m_n\to m_k$ weakly in $W^{1,r}(\R^N)$  for all $r<q$ and   strongly in $L^s(K)$ for all  $1 \le s<q^*$ and for every compact $K\subset \R^N$, and $w_n\to w_k$ weakly in $L^{\frac{\gamma'\alpha+\gamma'}{\gamma'+\alpha}}(\R^N)$.  
By using  the fact that $\int_{\R^N} V(x)m_n(x)dx\leq C_\eps$ and \eqref{vass}, we get that 
we get that for all $R>1$, 
\[C_\eps\geq \int_{\R^N}m_n(x) V(x)dx\geq \int_{|x|>R} m_n(x) V(x)dx\geq CR^b  \int_{|x|>R} m_n(x)dx.\]
So for every $\eps>0$ fixed and all $\eta>0$, there exists $R>0$ for which $ \int_{|x|>R} m_n(x)dx\leq \eta$:
up to extracting a subsequence we get that $m_n\to m_k$ in $L^1(\R^N)$,  and so $\int_{\R^N}m_k(x) dx=M$. 
By boundedness of $m_n$ in $L^s(\R^N)$ for all  $1 \le s<q^*$, we then have $m_n\to m_k$ strongly in $L^{\alpha+1}(\R^N)$.
Finally, observe that from \eqref{intV}, using \eqref{vass}, we conclude that $m_k(1+|x|^b)\in L^1(\R^N)$. Moreover, we get that 
\[
\int_{\R^N} |w_k| \, dx \le \int_{\R^N} |w_k| (1 + |x|)^{b/\gamma} \, dx \le \left( \int_{\R^N} \frac{|w_k|^{\gamma'}}{m_k^{\gamma'-1}} \, dx \right)^{1/\gamma'} \left( \int_{\R^N} m_k (1 + |x|)^{b} \, dx \right)^{1/\gamma},
\]
and so $w_k(1 + |x|)^{b/\gamma} \in L^{1}(\R^N)$.

Therefore the convergence is sufficiently strong to assure that $(m_k,w_k)\in \mathcal{K}_{\eps,M}$. We conclude that $(m_k,w_k)$ is a minimum of the energy,   by the lower semicontinuity  with respect to weak convergence of the functional $ \int_{\R^N} m L\left(-\frac{w}{m}\right) + V(x)m dx $ and  by using the fact that $F_k[m_n]\to F_k[m_k]$, since $m_n\to m_k$ strongly  in $L^{\alpha+1}(\R^N)$.
\end{proof} 
Using the minimizers we constructed in Proposition \ref{existence}, we prove existence of a classical solution to \eqref{mfgeps}. 
\begin{proposition}\label{ex}  There exists a classical solution $(u_k, m_k,\lambda_k)$ to \eqref{mfgeps} that satisfies for some constant $C_{k,\eps} > 0$ the following inequalities \begin{equation}\label{ukesti} |\nabla u_k(x)|\leq C_{k,\eps}(1+|x|^{\frac{b}{\gamma}})\qquad u_k(x)\geq  C_{k,\eps}^{-1}(1+|x|^{1+\frac{b}{\gamma}})-C_{k,\eps}.\end{equation}
Finally there exist $C, K>0$ not depending on $\eps, k$ such that 
 \begin{equation}\label{lambda}-K -C \eps^{-\frac{\gamma'\alpha N} {\gamma'-\alpha N}}\leq \lambda_k \leq C \eps^{-\frac{\gamma'\alpha N} {\gamma'-\alpha N}}+ K.\end{equation}  
\end{proposition} 
\begin{proof} 
Let $(m_k, w_k)$ be a minimizer of $\cE_k$. Define the space of test functions
\begin{equation}\label{Adef}
\cA = \cA_{b,\gamma} := \left\{ \psi \in C^2(\R^N) : \limsup_{|x| \to \infty} \frac{|\nabla \psi(x)|}{|x|^{b/\gamma}} < \infty, \,  \limsup_{|x| \to \infty} \frac{|\Delta \psi(x)|}{|x|^{b}} < \infty \right\}.
\end{equation}
Note that we also have,  for all $\psi\in \cA$,
\[
\limsup_{|x| \to \infty} \frac{|\psi(x)|}{|x|^{b/\gamma+1}} < \infty.
\]

We claim that
\begin{equation}\label{eq42}
- \eps \int_{\R^N} m_k \Delta \psi \, dx = \int_{\R^N} w_k \nabla \psi \, dx \qquad \forall \psi \in \cA.
\end{equation}
Indeed, consider a radial smooth cutoff function $\chi(x)$ which is identically equal to one in $B_1(0)$ and identically zero in $\R^N \setminus B_2(0)$. Set $\chi_R(x) := \chi(x/R)$; we have $|\nabla \chi_R| \le C \, R^{-1}$ and $|\Delta \chi_R| \le C \, R^{-2}$ on $\R^N$ for some positive constant $C$. 

Since the equality $\eps \Delta m_k  = {\rm div} w_k$ holds in the weak sense on $\R^N$, we may multiply it by $\chi_R \psi$ with $\psi\in \cA$ and integrate by parts to obtain
\begin{equation}\label{eq43}
- \eps \int_{B_{2R}} m_k ( \chi_R \Delta \psi + 2 \nabla \psi \cdot \nabla \chi_R + \psi \Delta \chi_R) \, dx = \int_{B_{2R}} w_k \cdot (\chi_R \nabla \psi + \psi \nabla\chi_R ). \, dx
\end{equation}
Note that for some positive $C$,
\[
\int_{\R^N} |w_k \nabla \psi| \, dx \le C \int_{\R^N} |w_k| (1 + |x|)^{b/\gamma} \, dx < \infty, \quad \int_{\R^N} m_k |\Delta \psi| \, dx \le C \int_{\R^N} m_k (1 + |x|)^{b} \, dx < \infty
\]
by the integrability properties \eqref{moreintegr}. Moreover,
\begin{multline*}
\int_{R \le |x| \le 2R} m_k |\psi| |\Delta \chi_R| \, dx \le C \int_{R \le |x| \le 2R} m_k \frac{(1 + |x|)^{b/\gamma + 1}}{R^2} \, dx \\ \le C_1 \int_{R \le |x| \le 2R} m_k (1 + |x|)^{b/\gamma -1} \, dx \to 0 \qquad \text{as $R \to \infty$,}
\end{multline*}
because $b/\gamma -1 \le b$. Reasoning in a similar way, we also have that $\int_{R \le |x| \le 2R} m_k \nabla \psi \cdot \nabla \chi_R$ and $\int_{R \le |x| \le 2R} w_k \cdot \psi \nabla \chi_R$ converge to zero as $R \to \infty$. Equality \eqref{eq42} then follows by passing to the limit in \eqref{eq43}. 

Therefore, recalling the integrability properties of $m_k,w_k$ obtained in Proposition \ref{existence}, 
the problem of minimizing $\cE_k$ on $\cK_{\eps, M}$ is equivalent to minimize $\cE_k$ on $\cK$, where
\[
\cK :=\{(w,m) \in (L^1 \cap W^{1,r})(\R^N) \times L^{\frac{\gamma'(\alpha+1)}{\gamma'+\alpha}}(\R^N) : \text{$(w,m)$ satisfies \eqref{moreintegr}, \eqref{eq42}, $m \ge 0$, $\int_{\R^N}m = M$}\}
\] for some $r<q$. 
As in \cite[Proposition 3.1]{BC16}, convexity of $L$ implies that $(m_k, w_k)$ is also a minimizer of the following convex functional on $\cK$:
\[
\widetilde{J}(m, w) = \int_{\R^N} m L\left(-\frac{w}{m}\right)+ (V(x) +f_k[m_k]) m  \, dx.
\]

We now aim to prove that
\begin{equation}\label{dual_pb} 
\sup \{ \lambda M : \text{$- \eps \Delta \psi + H(\nabla \psi) + \lambda \le V(x) + f_k[m_k]$ on $\R^N$ for some $\psi \in \cA$} \} = \min_{(w,m) \in \cK}\widetilde{J}(m, w).
\end{equation}
We proceed as in \cite[Theorem 3.5]{CG}: setting \[\mathcal{L}(m,w,\lambda, \psi) := \widetilde{J}(m, w) +  \int_{\R^N} \eps m \Delta \psi + w \nabla \psi - \lambda m\, dx + \lambda M,\] we have
\[
\min_{(m,w) \in \cK}\widetilde{J}(m, w) = \min_{(m,w)} \sup_{(\lambda, \psi) \in \R \times \cA } \mathcal{L}(m,w,\lambda, \psi),
\]
where the minimum in the right hand side has to be intended among couples $(m,w) \in(L^1 \cap W^{1,r})(\R^N) \times L^{\frac{\gamma'(\alpha+1)}{\gamma'+\alpha}}(\R^N)$ for some $r<q$, satisfying \eqref{moreintegr}. Note that $\mathcal{L}(\cdot,\cdot,\lambda, \psi)$ is convex, and $\mathcal{L}(m,w,\cdot,\cdot)$ is linear. Moreover, since $\mathcal{L}(\cdot,\cdot,\lambda, \psi)$ is weak-* lower semi-continuous, we can use the min-max theorem (see \cite[Theorem 2.3.7]{BW}), to get
\begin{multline*}
\min_{(m,w)} \sup_{(\lambda, \psi) \in \R \times \cA } \mathcal{L}(m,w,\lambda, \psi) = \sup_{(\lambda, \psi) \in \R \times \cA } \min_{(m,w)} \mathcal{L}(m,w,\lambda, \psi) = \\
\sup_{(\lambda, \psi) \in \R \times \cA } \min_{(m,w)} \int_{\R^N} m L\left(-\frac{w}{m}\right)+(V(x) +f_k[m_k]) m +  \eps m \Delta \psi + w \nabla \psi - \lambda m\, dx + \lambda M =\\
\sup_{(\lambda, \psi) \in \R \times \cA } \int_{\R^N} \min_{(m,w) \in \R \times \R^N}  m L\left(-\frac{w}{m}\right)+(V(x)+f_k[m_k]) m +  \eps m \Delta \psi + w \nabla \psi - \lambda m\, dx + \lambda M,
\end{multline*}
where the interchange of the $\min$ and the integration is possible by standard results in convex optimisation. By computation, $\min_{(m,w) \in \R \times \R^N}  m L\left(-\frac{w}{m}\right)+(V(x) +f_k[m_k]) m +  \eps m \Delta \psi + w \nabla \psi - \lambda m$ is zero whenever $\eps \Delta \psi - H(\nabla \psi) - \lambda + (V(x) +f_k[m_k])$ is positive, and it is $-\infty$ otherwise. Therefore, we have proven \eqref{dual_pb}.

By Theorem \ref{principal_eig}, {\it i), ii)}, there exists $u_k \in C^2(\R^N)$ such that
\begin{equation}\label{hjbeps}
- \eps \Delta u_k + H(\nabla u_k) + \lambda_k= V(x) +f_k[m_k] \qquad \text{on $\R^N$},
\end{equation} and which satisfies  
\[ |\nabla u_k(x)| \le C_{k,\eps}(1+|x|)^{\frac{b}{\gamma}}\qquad u_k(x)\geq  C_{k,\eps}|x|^{\frac{b}{\gamma}+1}- C_{k,\eps}^{-1}\] for some $ C_{k,\eps}>0$.

Moreover,  \[
\eps |\Delta u_k(x)| \le |H(\nabla u_k(x))| + |\lambda_k| + V(x) -f_k[m_k]\le C_{k,\eps}(1 + |x|)^{b} \qquad \text{on $\R^N$}
\]
so $u_k \in \cA$. Thus, the supremum in the left hand side of \eqref{dual_pb} is achieved by $\lambda_k$, and it holds true that
\begin{equation}\label{lambdaJ}
\lambda_k M = \widetilde{J}(m_k, w_k)=\mathcal{E}_k(m_k, w_k)+\int_{\R^N} f_k[m_k]m_k \, dx- F[m_k].
\end{equation}
This gives in particular \eqref{lambda}, using  Lemma \ref{boundfrombelow}, estimates \eqref{stimaalto1} and   recalling Proposition \ref{aprioriest} and assumptions \eqref{assFlocal}, \eqref{regf} and \eqref{cond1}.

We now use \eqref{lambdaJ}, \eqref{hjbeps} and  \eqref{eq42} with $\psi = u_k$ to get
\begin{multline*}
0 = \int_{\R^N} \left( L \left(-\frac{w_k}{m_k} \right) + V(x) - m_k^\alpha - \lambda_k \right)m_k \, dx = \int_{\R^N} \left( L \left(-\frac{w_k}{m_k} \right) - \eps \Delta u_k + H(\nabla u_k) \right)m_k \, dx \\
= \int_{\R^N} \left( L \left(-\frac{w_k}{m_k} \right) + H(\nabla u_k) + \nabla u_k \cdot \frac{w_k}{m_k} \right)m_k \, dx,
\end{multline*}
that implies
\[
\frac{w_k}{m_k} = - \nabla H(\nabla u_k) \qquad \text{on the set $\{m_k > 0\}$}.
\]
Hence, the Kolmogorov equation $\eps \Delta m_k + {\rm div} (m_k \nabla H(\nabla u_k)) = 0$ holds in the weak sense, and by elliptic regularity we conclude that $(u_k, m_k, \lambda_k)$ is a classical solution to \eqref{mfg}.
\end{proof} 

\begin{remark}\label{remarkene}\upshape 
Note that if we assume that the local term $f$ satisfies \eqref{assFlocal2} instead of \eqref{assFlocal},
then the same argument as above applies. In particular there exists a classical solution $(u_k, m_k,\lambda_k)$ to \eqref{mfgeps}
such that 
\begin{gather*}|\nabla u_k(x)|\leq  C_{k,\eps}(1+|x|^{\frac{b}{\gamma}})\qquad u_k(x)\geq   C_{k,\eps}^{-1}(1+|x|^{1+\frac{b}{\gamma}})- C_{k,\eps},\\
\int_{\R^N} m_k^{\alpha+1}dx, \int_{\R^N} m_k(x) V(x)dx\leq C\eps^{-\frac{\gamma'\alpha N} {\gamma'-\alpha N}}+K.\end{gather*} 

We finally prove that every $m_k$ is bounded from above in $\R^N$ (this is not obvious from Proposition \ref{ex} unless $\gamma' > N$). Note that the following result does not provide uniform bounds with respect to $k$. These will be produced in Theorem \ref{stimalinfinito} using a much more involved argument.

\begin{proposition}\label{mbond} Let $(u_k, m_k,\lambda_k)$ be as in Proposition \ref{ex}. Then, $m_k$ is bounded in $L^\infty(\R^N)$.
\end{proposition}

\begin{proof} Let $\phi(x)= u_k(x)^p$, for $p>1$ to be chosen later. Using the fact that $u_k$ is a classical solution to the HJB equation, we get
\begin{multline}\label{eq123}- \eps \Delta \phi + \nabla H(\nabla u_k)\cdot \nabla \phi = p u_k^{p-1} \left(-\Delta u_k-(p-1)\frac{|\nabla  u_k|^2}{u_k}+ \nabla H(\nabla u_k)\cdot \nabla u_k\right) \\
= p u_k^{p-1} \left(-\Delta u_k+H(\nabla u_k)-(p-1)\frac{|\nabla  u_k|^2}{u_k}- H(\nabla u_k) + \nabla H(\nabla u_k)\cdot \nabla u_k\right) \\
=  p u_k^{p-1} \left(-(p-1)\frac{|\nabla  u_k|^2}{u_k}- H(\nabla u_k) + \nabla H(\nabla u_k)\cdot \nabla u_k-\lambda+f_k[m_k]+V\right).
\end{multline} 
Observe that by \eqref{Hass},  \eqref{vass}, \eqref{ukesti} and the fact that $f_k[m_k]$ is bounded on $\R^N$, there exist large $R$ and $C$ such that
\begin{multline*}
G(x)\geq K^{-1} |\nabla u_k|^{\gamma}-(p-1)\frac{|\nabla  u_k|^2}{u_k} -K-\lambda+f_k[m_k]+V(x)
 \\
\geq (p-1)|\nabla u_k|^\gamma 
\left(\frac{1}{K(p-1)}-
\frac{|\nabla  u_k|^{2-\gamma}}{u_k}\right)-C +C_V^{-1} |x|^b \ge 1 \qquad \text{for all $|x| > R$}.
\end{multline*}
Hence, again by \eqref{ukesti}, for all $|x| > R$
\[
- \eps \Delta \phi + \nabla H(\nabla u_k)\cdot \nabla \phi \ge c|x|^{(1+b/\gamma)(p-1)}.
\]
In view of \cite[Proposition 2.6]{MP05}, we have $|x|^{(1+b/\gamma)(p-1)} m_k \in L^1(\R^N)$. Recall now that $|\nabla H(\nabla u_k)| \le C(1+|x|)^{\frac{b}{\gamma'}}$ by \eqref{ukesti}. Therefore, by choosing $p$ large enough, $|\nabla H(\nabla u_k)|^s m_k\in L^1(\R^N)$ for some $s > N$. We conclude boundedness of $m_k$ in $L^\infty$ by \cite[Theorem 3.5]{MP05}.
\end{proof} 

\end{remark}

\section{Existence of a solution to the MFG system for $\eps>0$}\label{secex}
Our aim is to pass to the limit $k \to \infty$ for solutions to \eqref{mfgeps}.  

\subsection{A priori $L^\infty$ bounds} 
We need first a priori $L^\infty$ bounds on  $m_k$ that are independent w.r.t. $k$. These will be achieved by a blow-up argument, as proposed in \cite{c16} for systems set on the flat torus $\mathbb{T}^N$. Here, the unbounded space $\R^N$ and the presence of the unbounded term $V$ make the argument much more involved than the one in \cite{c16}. To control the points $x_k \in \R^N$ where $m_k(x_k)$ possibly explodes, some delicate estimates on the decay (in $L^1$) of its renormalization will be produced.

We provide a more general result, that will be used also in the rescaled framework (Section \ref{secco}). Let $ r_k, s_k, t_k$ be bounded sequences of positive real numbers.

\begin{theorem}\label{stimalinfinito}Let  $( u_k, \lambda_k,  m_k)$
be a classical solution to the mean field game system
\[
\begin{cases}
- \Delta u + r^\gamma_k H(r_k^{-1}\nabla u )+\lambda_k  =g_k[m]+ s_kV(t_k x),\\
-\Delta  m -\di(m \, r_k^{\gamma-1} \nabla H(r_k^{-1} \nabla u ) )=0, \\   \int_{\R^N} m \, dx=M, \end{cases}
\] where $g_k : L^1(\R^N) \to L^1(\R^N)$ are so that for all $m\in L^\infty(\R^N)\cap L^1(\R^N)$ and for all $k$, 
\begin{equation}\label{glocal} \|g_k[m]\|_{L^\infty(\R^N)} \leq K(\|m\|_{L^\infty(\R^N)}^\alpha + 1) \end{equation} 
for some $K > 0$. Suppose also that  for all $k$, $u_k$ is bounded from below and $m_k$ is bounded from above on $\R^N$. Then, there exists a constant $C$ independent of $k$ such that 
 \[\|m_k\|_{L^\infty}\leq C.\]
\end{theorem}

\begin{proof}
We argue by contradiction, so  we assume that \[\sup_{\R^N} m_k =L_k\to +\infty.\] 
We divide the proof in several steps.

\noindent {\bf Step 1: rescaling of the solutions}.\\
Let \[\mu_k:= L_k^{-\beta}\qquad \beta=\alpha\frac{\gamma-1}{\gamma}>0.\] So, observe that $\mu_k\to 0$ as $k\to 0$. 
Since $u_k$ is bounded by below, up to adding a suitable constant we can assume that $\min_{\R^N} u_k=0$. 
We define the following rescaling
\[\begin{cases} v_k(x)= \mu_k^{\frac{2-\gamma}{\gamma-1}} u_k(\mu_k x )+1\\
n_k(x)= L_k^{-1} m_k(\mu_k x).
\end{cases} \]

Note that $0\leq n_k(x)\leq 1$. Moreover, due to \eqref{alpha},
\begin{equation}\label{l1} \int_{\R^N} n_k(x)dx = M L_k^{\frac{\alpha N(\gamma-1)}{\gamma}-1}\to 0, \end{equation} 
and $\min v_k=1 $.  We define 
\[ H_k(q)=\mu_k^{\frac{ \gamma}{\gamma-1}} r_k^{\gamma}H(r_k^{-1}\mu_k^{\frac{1}{1-\gamma}} q), \qquad \text{ so } \quad\nabla H_k(q)= \mu_k  r_k^{\gamma-1 }\nabla H(r_k^{-1} \mu_k^{\frac{1 }{1-\gamma}} q).\] 
Recalling \eqref{Hass} we have that for all $q \in \R^N$, 
\begin{equation}\label{bh}
\begin{split}
&C_H|q|^\gamma- K\leq H_k(q)\leq C_H( |q|^\gamma +1), \\
 &|\nabla H_k(q))|\leq C_H|q|^{\gamma-1}, \\
& \nabla H_k(q)\cdot q-H_k(q)\geq K^{-1} |q|^\gamma -K.
\end{split}
\end{equation} 
Moreover, we define 
\[\tilde g_k(x)=\mu_k^{ \frac{\gamma}{\gamma-1}} g_k[m_k](\mu_k x).\]
Recalling that $0\leq m_k\leq L_k$,  by \eqref{glocal}  we get that for all $x$ and for all $k$,
\begin{equation}\label{bf} |\tilde g_k(x)|\leq\mu_k^{\frac{\gamma}{\gamma-1}} K ( L_k^{\alpha} + 1) \leq 2 K\end{equation}
where we used the fact that $\mu_k= L_k^{-\beta}$ with $\beta=\alpha \frac{\gamma-1}{\gamma}$. 
Finally, we let
\[\tilde \lambda_k= \mu_k^{ \frac{\gamma}{\gamma-1}} \lambda_k=\frac{1}{L_k^\alpha} \lambda_k \] and we observe that%, by \eqref{lambda},
 \begin{equation}\label{bl} |\tilde \lambda_k|\leq C.\end{equation}
Finally, let
\[
V_k(x) = \mu_k^{\frac{\gamma}{\gamma-1}} s_kV(\mu_k t_k x).
\]
By assumption \eqref{vass}, we get 
 \begin{equation}\label{bv} s_k \mu_k^{\frac{\gamma}{\gamma-1}} C_V^{-1} (\max\{ |t_k \mu_k x|-C_V, 0\})^{b} \leq V_k(x) \leq C_V(1+ \sigma_k |x|^b),  \end{equation}
where
\[
\sigma_k := \mu_k^{\frac{\gamma}{\gamma-1}+b} s_k t_k^b \to 0 \qquad \text{as $k \to \infty$.}
\]
 In particular we also have the following bound from below for $V_k$, 
 \begin{equation}\label{bv2}
V_k(x) \ge \frac{C_V^{-1}}{2^b} \sigma_k |x|^b \quad \text{for all $|x| \ge 2C_V (t_k \mu_k)^{-1}$.}
 \end{equation}

An easy computation shows that  by rescaling 
we have that $(v_k, n_k,\tilde \lambda_k)$ is a solution to
\begin{equation}\label{resc} \begin{cases}-\Delta v_k + H_k(\nabla v_k)+\tilde \lambda_k =\tilde g_k(x)+V_k(x),\\
-\Delta  n_k -\di(n_k  \nabla H_k(\nabla v_k ) )=0.
 \end{cases} \end{equation}

\noindent {\bf Step 2: a priori bounds on the rescaled solution to the  Hamilton-Jacobi equation}.\\ 
We observe that by Theorem \ref{bernstein2} and \eqref{bv}, there exists $C>0$, independent of $k$, such that 
 \begin{equation}\label{gb}  |\nabla v_k(x)|\leq C (1+\sigma_k^{\frac{1}{\gamma}}|x|^{\frac{b}{\gamma}}) \quad \text{on $\R^N$.}\end{equation}
 We recall that we assumed $v_k(\hat x_k) = \min v_k=1$. %, thanks to the property of function $u_k$ stated in Proposition \ref{ex}. Let $\hat x_k$ be a minimum point of $v_k$.
Since $v_k$ is a classical solution to \eqref{resc}, at a minimum point $\hat x_k$ we have, by \eqref{bh}, \eqref{bf}, \eqref{bl} and \eqref{bv2}, \[\sigma_k|\hat x_k|^b\leq C.\]
 Therefore, by using this estimate and \eqref{gb}, since $|v_k(0)| \le |v_k(\hat x_k)| + |\hat x_k| \sup_{|y| \le |\hat x_k|} |\nabla u_k(y)|$ we obtain \[|v_k(0)|\leq 1+ C (1+\sigma_k^{\frac{1}{\gamma}}|\hat x_k|^{1+\frac{b}{\gamma}})\leq C_1
 (1+\sigma_k^{-\frac{1}{b}})\] and then again by \eqref{gb}, 
\begin{equation} \label{stimasu} |v_k(x)|\leq C(1+ \sigma_k^{-\frac{1}{b}} + \sigma_k^{\frac{1}{\gamma}} |x|^{\frac{b}{\gamma}+1}) \quad \text{on $\R^N$}.\end{equation}

Let $\chi$ be a smooth function $\chi:[0,+\infty)\to [0, +\infty)$ such that $\chi\equiv 0$ in $(0, 1/2)\cup (3/2, +\infty)$, $\chi(1) > 0$ and that $|\chi'|, |\chi''|\leq 1$.  
We fix $\tilde x\in \R^N$ such that $|\tilde x|> 4 C_V (t_k \mu_k)^{-1}$,
and we denote by \[w(x)=K\sigma_k^{\frac{1}{\gamma}}|\tilde x|^{1+\frac{b}{\gamma}}\chi\left(\frac{|x|}{|\tilde x|}\right)\]where $K\geq 0$ has to be chosen.
We have that $w(x)\leq v_k(x)$ for all $x$ such that $|x|\geq \frac{3}{2}|\tilde x|$ or $|x|\leq \frac{1}{2} |\tilde x|$. 
Moreover,  for $x$ such that $\frac{1}{2}|\tilde x|\leq |x|\leq \frac{3}{2}|\tilde x|$
we have $|x|> 2 C_V (\mu_k t_k)^{-1}$, so using the estimates \eqref{bh}, \eqref{bf}, \eqref{bl} and \eqref{bv2}, 
\[- \Delta w + H_k(\nabla w)+\tilde \lambda_k -\tilde g_k(x)-V_k(x)\\\leq 
KN \sigma_k^{\frac{1}{\gamma}}|\tilde x|^{\frac{b}{\gamma}-1} +C_HK^\gamma \sigma_k |\tilde x|^{b}+C- \frac{C_V^{-1}}{2^b} \sigma_k |\tilde x|^b. 
\]
Note that there exist $K > 0$ small and $C_2 >0 $ large, depending only $C_V$ and $C_H$ and not on $|\tilde x|$, $k$, such that the right-hand side of the last expression is negative if 
\[
\sigma_k |\tilde x|^b \ge C_2
\]
(this also implies that $t_k \mu_k |\tilde x|> 4 C_V$, as required). The test function $w$ is then a subsolution of the HJB equation in \eqref{resc}, therefore by comparison we get that,
\[v_k(\tilde x) \geq K \chi(1) \sigma_k^{\frac{1}{\gamma}}|\tilde x|^{1+\frac{b}{\gamma}}. \] By arbitrariness of $\tilde x$ we conclude that, for some $C>0$, 
\begin{equation} \label{stimagiu} v_k(x)\geq  C\sigma_k^{\frac{1}{\gamma}}|x|^{\frac{b}{\gamma}+1} \qquad \text{for all $\sigma_k |x|^b \ge C_2$}.\end{equation}

\noindent {\bf Step 3: estimates on the (approximate) maxima of $n_k$}. \\ We now fix $0<\delta<<1$ and $x_k$ such that $n_k(x_k)=1-\delta$. 
Two possibilities may arise: either $\lim_k \sigma_k|x_k|^b=+\infty$ up to some subsequence, or there exists $C>0$ such that 
$\sigma_k|x_k|^b\leq C$. We rule out the second possibility by contradiction. Suppose indeed that   there exists $C>0$ such that   $\sigma_k|x_k|^b\leq C$.
By \eqref{gb},  $|\nabla v_k|\leq C$ on $B_2(x_k)$ for some $C>0$. Therefore, using the fact that $n_k$ solves the second equation in \eqref{resc}, 
the elliptic estimates in  Proposition \ref{ell_regularity}, \eqref{bh}, the interpolation inequality $\|n\|_q \le \|n\|^{1/q}_1 \|n\|^{1-1/q}_\infty$ and the fact that $0\leq n_k\leq 1$, 
we get for all $q>1$, 
\begin{equation}\label{Wb}
\|n_k\|_{W^{1,q}(B_1(x_k))}\leq C (1+\|\nabla H_k(\nabla v_k)\|_{L^\infty(B_2(x_k))})\|n_k\|^{1/q}_{L^1(B_2(x_k))} \leq C_q
\end{equation}
for some $C_q>0$ depending on $q$. This implies, choosing $q>N$, that  for all $\theta\in (0,1)$ there exists $C_\theta$ depending on $\theta$ (but not on $k$) such that 
$\|n_k\|_{C^{0,\theta}(B_1(x_k))}\leq C_\theta$. Recalling that $n_k(x_k)=1-\delta$, we can fix $r<1$ such that $n_k(x)\geq \frac{1}{2}$ for all $x\in B_r(x_k)$. It is sufficient to choose $r=C_\theta^{-1/\theta}(1/2-\delta)^{1/\theta}$.
Therefore we have, by \eqref{l1}, 
\[0< \frac{1}{2} \omega_N r^N \leq \int_{B_r(x_k) } n_k(x)dx\leq \int_{\R^N}n_k(x)dx= M L_k^{\frac{\alpha N(\gamma-1)}{\gamma}-1}\to 0. \]
This gives a contradiction. Then we deduce that, up to a subsequence, 
  \begin{equation}\label{limite} \lim_k \sigma_k|x_k|^b=+\infty.\end{equation}

\noindent {\bf Step 4: construction of a Lyapunov function}.\\ 
Let $\phi(x)= v_k(x)^p$, for $p>1$ to be chosen later. Using the fact that $v_k$ is a classical solution to \eqref{resc} (arguing as in \eqref{eq123}) we get
\begin{multline*}- \Delta \phi + \nabla H_k(\nabla v_k)\cdot \nabla \phi = p v_k^{p-1} \left(-\Delta v_k-(p-1)\frac{|\nabla  v_k|^2}{v_k}+ \nabla H_k(\nabla v_k)\cdot \nabla v_k\right) \\
% = p v_k^{p-1} \left(-\Delta v_k+H_k(\nabla v_k)-(p-1)\frac{|\nabla  v_k|^2}{v_k}- H_k(\nabla v_k) + \nabla H_k(\nabla v_k)\cdot \nabla v_k\right) \\
=  p v_k^{p-1} \left(-(p-1)\frac{|\nabla  v_k|^2}{v_k}- H_k(\nabla v_k) + \nabla H_k(\nabla v_k)\cdot \nabla v_k-\tilde\lambda_k+\tilde g_k(x)+V_k(x)\right).
\end{multline*} 
We denote by 
\begin{equation}\label{g} G_k(x)=-(p-1)\frac{|\nabla  v_k|^2}{v_k}- H_k(\nabla v_k) + \nabla H_k(\nabla v_k)\cdot \nabla v_k-\tilde\lambda_k+\tilde g_k(x)+V_k(x).\end{equation} 
Using the previous computation and the fact that $n_k$ is a solution to \eqref{resc}, we get, by integrating by parts,  that
\[0=\int_{\R^N} n_k(x)\left(- \Delta \phi(x) + \nabla H_k(\nabla v_k(x))\cdot \nabla \phi(x)\right)dx =p\int_{\R^N} n_k(x) G_k(x)  \phi^{\frac{p-1}{p}}(x)dx.\]
Therefore from this, for every $\Lambda>0$ we get 
\begin{equation}\label{stimaint} 
\int_{\{\phi(x)\geq \Lambda^p\}}  n_k(x) G_k(x) \phi^{\frac{p-1}{p}}(x)dx = - \int_{\{\phi(x)\leq \Lambda^p\}}  n_k(x) G_k(x)  \phi^{\frac{p-1}{p}}(x)dx.
\end{equation} 
Observe that by \eqref{bh}, \eqref{bf}, \eqref{bl} and \eqref{bv2} we get that for all $t_k \mu_k |x| \ge 2C_V$, 
\begin{multline}\label{stimag} 
G_k(x)\geq K^{-1} |\nabla v_k|^{\gamma}-(p-1)\frac{|\nabla  v_k|^2}{v_k} -K-\tilde\lambda_k+\tilde g_k(x)+V_k(x)
 \\
\geq (p-1)|\nabla v_k|^\gamma 
\left(\frac{1}{K(p-1)}-
\frac{|\nabla  v_k|^{2-\gamma}}{v_k}\right)-C +C_V\sigma_k|x|^b.
\end{multline}
We first claim that by \eqref{gb} and \eqref{stimagiu}, $\frac{1}{K(p-1)}-
\frac{|\nabla  v_k|^{2-\gamma}}{v_k}$ is positive if $\sigma_k|x|^b \ge C_2$, eventually enlarging $C_2$ in \eqref{stimagiu}. Indeed,
\begin{equation}\label{ineq9}
\frac{|\nabla  v_k(x)|^{2-\gamma}}{v_k(x)} \le C\frac{\left[1 + \sigma_k^{\frac{1}{\gamma}}|x|^{\frac{b}{\gamma}}\right]^{2-\gamma}}
{\left[\sigma_k^{\frac{1}{\gamma}}|x|^{\frac{b}{\gamma}}\right]|x|} \le \frac{C_H}{p-1}
\end{equation}
whenever $\sigma_k|x|^b$ is large enough. This implies that for all $\sigma_k|x|^b \ge C_2$, by \eqref{stimag} we have $G_k(x) \ge -C$. On the other hand, again by the gradient bounds in \eqref{gb} we have that $|\nabla v_k(x)|\leq C (1+C_2)$ on the set $\sigma_k|x|^b \le C_2$, so \eqref{stimag} and $\min v_k=1$ again guarantee that $G_k(x) \ge -C_3$. In conclusion, there exists $C>0$ such that \[G_k(x)\geq -C\qquad \forall x\in\R^N.\] 

Therefore, going back to \eqref{stimaint}, recalling \eqref{l1}, we obtain that
\begin{multline}\label{stimaint2} 
\int_{\{\phi(x)\geq \Lambda^p\}}  n_k(x) G_k(x) \left(\frac{\phi(x)}{\Lambda^p}\right)^{\frac{p-1}{p}} dx \leq C \int_{\{\phi(x)\leq \Lambda^p\}}  n_k(x) dx\leq C\int_{\R^N} n_k(x)dx\\=
CM\mu_k^{-N+\frac{\gamma}{\alpha (\gamma-1)}} \to 0
\end{multline} 
as $k \to \infty$.

Note that by \eqref{stimag} and \eqref{ineq9}, if $x$ is such that $G_k(x)\leq 0$, then necessarily  $\sigma_k|x|^b \leq C$ for some $C>0$. Hence, by \eqref{stimasu}, we get that $v_k(x)\leq C_3(1+\sigma_k^{-\frac{1}{b}})$. 
Therefore if we choose $\Lambda=\Lambda_k= K \sigma_k^{-\frac{1}{b}}$ for a sufficiently large $K>0$, we get that $G_k(x)>0$ in the set $\{x| \phi(x)\geq \Lambda^p\}$. 

 \smallskip 
\noindent {\bf Step 5: integral estimates on $n_k$}.\\ 
Arguing as in the end of Step 4, we may choose $K$ big enough so that $G_k(x)\geq 1$ in the set $\{x| \phi(x)\geq \Lambda^p\}$, where $\Lambda_k= K \sigma_k^{-\frac{1}{b}}$. If $k$ is sufficiently large, by \eqref{stimagiu}  and \eqref{limite} it follows that for some $C > 0$,
\[
\begin{split}
& v_k(x)\geq C\sigma_k^{\frac{1}{\gamma}}|x_k|^{1+\frac{b}{\gamma}} \quad \text{in $B_1(x_k)$, and} \\
& B_1(x_k)\subseteq  \{x| \phi(x)\geq \Lambda^p\}.
\end{split}
\]

Therefore, we may conclude that 
\begin{multline}\label{stimaint3} \int_{\{\phi(x)\geq \Lambda^p\}}  n_k(x) G_k(x) \left(\frac{\phi(x)}{\Lambda^p}\right)^{\frac{p-1}{p}} dx
\geq C \left( \frac{ \sigma_k^{\frac{1}{\gamma}}|x_k|^{1+\frac{b}{\gamma}} }{ \sigma_k^{-\frac{1}{b}}}\right)^{p-1} \int_{B_1(x_k)} n_k(x)dx \\ \ge
C \left( \sigma_k^{\frac{1}{\gamma}}|x_k|^{\frac{b}{\gamma}} \right)^{p-1} \int_{B_1(x_k)} n_k(x)dx,
\end{multline}
that together with \eqref{stimaint2} gives
\begin{equation}\label{stimaint4}
 \int_{B_1(x_k)} n_k(x)dx \le \left(\sigma_k^{\frac{1}{\gamma}} |x_k|^{\frac{b}{\gamma}} \right)^{1-p}
\end{equation}
for all $k$ large.

Reasoning as in Step 3 (see in particular \eqref{Wb}), by Proposition \ref{ell_regularity}, \eqref{bh}, \eqref{gb} and \eqref{stimaint4}, we get that for all $q>1$, 
\begin{multline*}
\|n_k\|_{W^{1,q}(B_{1/2}(x_k))}\leq C (1+\|\nabla H_k(\nabla v_k)\|_{L^\infty(B_1(x_k))})\|n_k\|^{1/q}_{L^1(B_1(x_k))}\\ \leq C_4 \left[1+ \left(\sigma_k^{\frac{1}{\gamma}} |x_k|^{\frac{b}{\gamma}}\right)^{\gamma-1}\right] \left( \sigma_k^{\frac{1}{\gamma}} |x_k|^{\frac{b}{\gamma}} \right)^{(1-p)/q} \le 1,
\end{multline*}
whenever $p$ is such that $\gamma-1+(1-p)/q < 0$ and $k$ is large (recall that we are supposing $\sigma_k^{\frac{1}{\gamma}} |x_k|^{\frac{b}{\gamma}}\to +\infty$).

Therefore, we may conclude as in Step 3: choosing $q>N$, for some $\theta\in (0,1)$ there exists $C_\theta$ such that 
$\|n_k\|_{C^{0,\theta}(B_{1/2}(x_k))}\leq C_\theta$. Since $n_k(x_k)=1-\delta$, we can fix $r<1$ such that $n_k(x)\geq \frac{1}{2}$ for all $x\in B_r(x_k)$. Finally, by \eqref{l1}
\[0< \frac{1}{2} \omega_N r^N \leq \int_{B_r(x_k) } n_k(x)dx\leq \int_{\R^N}n_k(x)dx= M L_k^{\frac{\alpha N(\gamma-1)}{\gamma-1}}\to 0. \]
That gives a contradiction and rules out the possibility that $ \sigma_k |x_k|^{b}\to +\infty$. Therefore, $L_k\to +\infty$ is impossible.
\end{proof} 
\subsection{Existence of a solution to the MFG system} 
Using the a priori bounds we obtained, we can pass to the limit in $k$ in the MFG system  \eqref{mfgeps} to get a solution to \eqref{mfg} for every $\eps>0$. 
\begin{proof}[Proof of Theorem \ref{exthm}]
First, by Proposition \ref{ex}, the existence for all $k$ of a classical solution  $(u_k,m_k,\lambda_k)$ to \eqref{mfgeps} follows. 
By \eqref{lambda},  up to passing to a subsequence we have that $\lambda_k\to \lambda_\eps$. 

Note that by Propositions \ref{ex} and \ref{mbond}, $u_k$ and $m_k$ are bounded by below and above respectively, so due to Theorem \ref{stimalinfinito} (with $g[m]=f_k[m]$ and $r_k = s_k = t_k = 1$ for all $k$),
we get that there exists $C_\eps>0$ independent of $k$ (but eventually on $\eps>0$) such that $\|m_k\|_{L^\infty(\R^N)}\leq C_\eps$. Using Theorem \ref{bernstein2}, this implies that 
$|\nabla u_k(x)|\leq C_\eps(1+|x|^{\frac{b}{\gamma}})$, for some $C_\eps$ independent of $k$. We can normalize $u_k(0)=0$ and using Ascoli-Arzel\'a theorem we can extract by a diagonalization procedure 
a sequence $u_k$ such that $u_k\to u_\eps$ locally uniformly in $\R^N$. Moreover, by using the estimates and the equation we have that actually $u_k\to u_\eps$ 
locally uniformly in  $C^1$. Note that, denoting by $x_k$ a minimum point of $u_k$ on $\R^N$, we have by the HJB equation that
\[
H(0)+\lambda_k  -f_k[m_k](x_k) \ge V(x_k).
\]
Coercivity \eqref{vass} of $V$ and uniform  boundedness of $\lambda_k$ and $f_k[m_k]$ guarantee that $x_k$ remains bounded, in particular that $u_k \ge -C$ on $\R^N$ by gradient bounds. Theorem \ref{boundedfrombelowsol} then applies, in particular $
u_k(x)\geq C|x|^{1+\frac{b}{\gamma}}-C^{-1}$ for all $k$. This implies, passing to the limit, that
\begin{equation}\label{uhasmin}
u_\eps(x)\geq C|x|^{1+\frac{b}{\gamma}}-C^{-1} \qquad \text{on $\R^N$}.
\end{equation}

By the elliptic estimates in Proposition \ref{ell_regularity}, we get that $m_k\to m_\eps $ locally uniformly in $C^{0,\alpha}$ for all $\alpha\in (0,1)$  and weakly in $W^{1,p}(B_R)$ for every $p>1$ and  $R>0$. Therefore we get that $u_\eps$ is a solution  in the viscosity sense of the Hamilton-Jacobi equation, by stability with respect to uniform convergence, and $m_\eps$ is a weak solution to the Fokker-Planck equation, by strong convergence of $\nabla u_k\to \nabla u_\eps$. 
Finally this implies, again by using the regularity of the HJB equation, that $u_k\to u_\eps$ locally uniformly  in $C^2$. 
Therefore, $u_\eps, m_\eps$ solve in classical  sense the system \eqref{mfg}. 

Now we show that $\int_{\R^N} m_\eps(x)dx=M$. We have that $m_k\to m_\eps$ locally uniformly in $C^{0,\alpha}$ for every $\alpha\in (0,1)$. Moreover, due to \eqref{intV} and to \eqref{vass}, we get that for all $R>1$, 
\[C_\eps\geq \int_{\R^N}m_k(x) V(x)dx\geq \int_{|x|>R} m_k(x) V(x)dx\geq CR^b  \int_{|x|>R} m_k(x)dx.\]
This implies that $\int_{|x|\leq R} m_k(x)dx\geq  M-C_\eps R^{-b}$ and then by uniform convergence we get that for every $\eps>0$, and $\eta>0$, there exists $R>0$ 
such that 
\[\int_{|x|\leq R} m_\eps(x)dx \geq M-\eta.\] From this we can conclude that $m_k\to m_\eps$ in $L^1(\R^N)$, that is $\int_{\R^N}m_\eps(x)dx=M$. By boundedness of $m_k$ in $L^\infty$, it also follows that $m_k\to m_\eps$ in $L^{\alpha+1}(\R^N)$.

Finally, we get that if $w_\eps= -m_\eps \nabla H (\nabla u_\eps)$, then $(m_\eps, w_\eps)\in \mathcal{K}_{\eps, M}$, due to the second equation in \eqref{mfg}. Moreover, 
we have that if $m_k\to m$ strongly in $L^{\alpha+1}(\R^N)$, then, due to the Lebesgue dominated convergence theorem and \eqref{cond1}, 
$ F(m_k\star \chi_k)\to F(m)$ strongly in $L^1(\R^N)$. 
This implies that the energy $\mathcal{E}_k$ $\Gamma$-converges to the energy $\mathcal{E}$, from which we conclude that $(m_\eps, w_\eps)$ is a minimizer of $\mathcal{E}$ in the set $\mathcal{K}_{\eps,M}$. 
\end{proof} 

\begin{remark}\label{rem2} \upshape Note that by the very same arguments,  recalling Remark \ref{remarkene}, we have the existence of solutions also in the more general case that condition 
\eqref{assFlocal2} is satisfied. 
\end{remark} 

We conclude proving some estimates on the solution $(u_\eps, m_\eps, \lambda_\eps)$ given in Theorem \ref{exthm} that will be useful in the following. 
\begin{corollary} \label{pnuova}  Let $(u_\eps, m_\eps,\lambda_\eps)$ be as in Theorem \ref{exthm}. 
There exist constants $C, C_1,C_2, K, K_1, K_2> 0$ independent of  $\eps$  such that
\begin{gather}
\int_{\R^N} m_\eps |\nabla u_\eps|^{\gamma} dx+ \int_{\R^N} m_\eps^{\alpha+1}dx+ \int_{\R^N} m_\eps(x) V(x)dx\leq C\eps^{-\frac{\gamma'\alpha N} {\gamma'-\alpha N}}+K
\label{intV1}\\
\label{lambda1} -K_1-C_1\eps^{-\frac{\gamma'\alpha N} {\gamma'-\alpha N}}\leq \lambda_\eps \leq K_2-C_2\eps^{-\frac{\gamma'\alpha N} {\gamma'-\alpha N}}.
\end{gather} 
 \end{corollary}

\begin{proof}
We observe that, by the arguments in the proof of Theorem \ref{exthm},  $m_k\to m_\eps$ and  $|\nabla u_k|\to |\nabla u_\eps|$ almost everywhere, and using the fact that $V(x)\geq 0$, we have that by Fatou lemma 
$ \int_{\R^N} m_\eps(x) |\nabla u_\eps|^\gamma dx\leq \liminf_k \int_{\R^N} m_k(x) |\nabla u_k|^\gamma dx $,  $ \int_{\R^N} m_\eps(x) V(x)dx\leq \liminf_k \int_{\R^N} m_k(x) V(x)dx$  and $\int_{\R^N} m_\eps^{\alpha+1}dx\leq \liminf_k \int_{\R^N} m_k^{\alpha+1}dx$. So inequality \eqref{intV} gives immediately \eqref{intV1}. 

Now we prove \eqref{lambda1}. Note that the estimate from below is a direct consequence of \eqref{lambda}. So, it remains to show that 
$ \lambda_\eps \leq C_2-C_2\eps^{-\frac{\gamma'\alpha N} {\gamma'-\alpha N}}$. 
Recalling that formula \eqref{lambdaJ} holds and $\int f(m)m - F(m) \le 2KM$ by \eqref{assFlocal}, it is sufficient to show that \begin{equation}\label{energybound2}\inf_{(m,w)\in \cK_{\eps, M}}\cE(m,w) \leq  -C_2 \eps^{-\frac{\gamma'\alpha N}{ \gamma'-\alpha N}} + C_2 \end{equation}
where $C_2$ is a constant depending only on $N, M, C_L, \gamma,\alpha, V$.
We construct a couple $(m,w)\in \cK_{\eps, M}$  as follows. 
First of all we consider a smooth function $\phi:[0, +\infty)\to \R$ which solves the following ordinary differential equation
\[\begin{cases} \phi'(r)=-\phi(r)(1+\phi(r)^\alpha)^{\frac{1}{\gamma'}}\\ \phi(0)=\frac{1}{2}.\end{cases}\]
Then, it is easy to check that $0<\phi(r)\leq \frac{1}{2}e^{-r}$. 
We define $m(x)=A \phi(\tau |x|)$, where $A, \tau$ are constants to be fixed, and $w(x)=\eps \nabla m(x)$. 

First of all we impose 
\[M=\int_{\R^N} m(x)dx= \frac{A}{\tau^N} \int_{\R^N} \phi(|y|)dy= \frac{A}{\tau^N} C^{-1},\]
recalling that $\phi$ is exponentially decreasing. 
So $A= M  \tau^N C $, where $C^{-1}=\int_{\R^N} \phi(|y|)dy$. 

Observe also that 
\begin{equation}\label{malpha}  \int_{\R^N} m^{\alpha+1}(x)dx=  M^{\alpha+1}\tau^{\alpha N} C^{\alpha+1} \int_{\R^N} \phi^{\alpha+1}(|y|)dy = M^{\alpha+1}\tau^{\alpha N} C^{\alpha+1} C_\alpha \end{equation} 
where $C_\alpha = \int_{\R^N} \phi^{\alpha+1}(|y|)dy.$

We check, recalling that the growth condition \eqref{vass}, that the following holds 
\begin{equation}\label{vest} \int_{\R^N} m(x)V(x)dx=MC \int_{\R^N} V\left(\frac{y}{\tau}\right) \phi(|y|)dy= C_1 \frac{1}{\tau^b},
\end{equation}  where $K$ is a constant depending on $N$, $\phi$, $C_0$. 

Moreover, we compute, recalling that $\phi$ solves the ODE
\begin{equation}\label{we} 
|w|^{\gamma'} =\left|\eps  \tau m \left(1+ \frac{1}{M^\alpha C^\alpha \tau^{N\alpha}} m^\alpha \right)^{\frac{1}{\gamma'}}\right|^{\gamma'}= 
\eps^{\gamma'} \tau^{\gamma'}m^{\gamma'} \left(1+ \frac{1}{M^\alpha C^\alpha \tau^{N\alpha}} m^\alpha\right).
\end{equation}

We consider the energy at $(m,w)$
\[\cE(m, w) = \int_{\R^N} m L\left(-\frac{w}{m}\right) +F(m) + mV(x)\,dx  
 \]
 Observe that by \eqref{assFlocal}, $F(m)\leq -\frac{C_f}{\alpha+1}m^{\alpha+1} +Km$. 
 Using   Proposition \ref{Lproperties}, and computation \eqref{we} and \eqref{malpha}, we get 
\begin{multline*} \int_{\R^N} m L\left(-\frac{w}{m}\right) +F(m) \,dx \leq \int_{\R^N} m L\left(-\frac{w}{m}\right) \, dx   - \frac{C_f}{\alpha+1} \int_{\R^N} m^{\alpha + 1} \, dx+KM   \\
\leq C_L \int_{\R^N} m \frac{|w|^{\gamma'}}{m^{\gamma'}}\, dx +(C_L+K) M - \frac{C_f}{\alpha+1} \int_{\R^N} m^{\alpha + 1} \, dx \\ =
C_L\eps^{\gamma'} \tau^{\gamma'} \left(M+\int_{\R^N}  \frac{1}{M^\alpha C^\alpha \tau^{N\alpha}} m^{\alpha+1}dx\right)  +(C_L+K)M - 
\frac{C_f}{\alpha+1} \int_{\R^N} m^{\alpha + 1}  
\\ = C_L \eps^{\gamma'} \tau^{\gamma'} M+ (C_L+K) M - \left(\frac{C_f}{\alpha+1}-  \frac{\eps^{\gamma'} \tau^{\gamma'- N\alpha}}{M^\alpha C^\alpha }\right) \int_{\R^N} m^{\alpha + 1} \, dx\\ =
 (M C_L +M CC_\alpha) \eps^{\gamma'} \tau^{\gamma'} - \frac{C_f}{\alpha+1}M^{\alpha+1}C^{\alpha+1} C_\alpha\tau^{\alpha N}+(C_L+K)M.
\end{multline*} 
We choose now $\tau$ such that $ \tau=\frac{1}{A} \eps^{-\frac{\gamma'}{\gamma'- N\alpha}} $, where $A$ is sufficiently large, in such a way that 
\[ \int_{\R^N} m L\left(-\frac{w}{m}\right) \, dx   +F(m) \, dx\leq -C\eps^{-\frac{\gamma' N\alpha}{\gamma'-N\alpha}} +C \]
where $C$ is a constant depending on $\alpha, C_L, M$. 
Substituting this in the energy and recalling \eqref{vest}, we get the desired inequality. 
\end{proof}

\section{Concentration phenomena}\label{secco}
In the second part of this work, we are interested in the asymptotic analysis of solutions to \eqref{mfg} when $\eps\to 0$. 

\subsection{The rescaled problem} \label{rescaledsection}
We consider the following rescaling
\begin{equation}
\label{rescaling} 
\begin{cases} 
\tilde m(y ) := \eps^{\frac{N \gamma'}{\gamma'-\alpha N}}  m(\eps^{\frac{\gamma'}{\gamma'-\alpha N}} y),\\
%\tilde w(y) :=  \eps^{\frac{N(\gamma'+\alpha)}{\gamma'-\alpha N}}  w(\eps^{\frac{\gamma'}{\gamma'-\alpha N}} y+x_\eps)\\
\tilde{u}(y) := \eps^{\frac{N \alpha(\gamma'-1)-\gamma'}{\gamma'-\alpha N}} {u} (\eps^{\frac{\gamma'}{\gamma'-\alpha N}}  y) \\
\tilde \lambda := \eps^{\frac{N \alpha \gamma'}{\gamma'-\alpha N}} \lambda.
\end{cases} 
\end{equation} 
We introduce  the rescaled potential  \begin{equation}\label{riscalata} 
V_\eps(y) = \eps^{\frac{N \alpha \gamma'}{\gamma'-\alpha N}}  V(\eps^{\frac{\gamma'}{\gamma'-\alpha N}} y ) .% h(\eps^{\frac{\gamma'}{\gamma'-\alpha N}} y) \prod_{j = 1}^{n} |\eps^{\frac{\gamma'}{\gamma'-\alpha N}} y-x_j|^{\hat b}
\end{equation}
Note that by \eqref{vass}, we get    
\begin{equation}\label{vassresc}
C_V^{-1} \eps^{\frac{N \alpha \gamma'}{\gamma'-\alpha N}} (\max\{ |\eps^{\frac{\gamma'}{\gamma'-\alpha N}}y|-C_V, 0\})^{b} \le V_\eps(y) \le C_V \eps^{\frac{N \alpha \gamma'}{\gamma'-\alpha N}} (1 + \eps^{\frac{\gamma'}{\gamma'-\alpha N}}|y|)^{b} .
\end{equation}
 The rescaled coupling term is given by
  \begin{equation}\label{friscalata} 
f_\eps(\tilde m(y)) = \eps^{\frac{N \alpha \gamma'}{\gamma'-\alpha N}}  f\left(\eps^{-\frac{N \gamma'}{\gamma'-\alpha N}}  m(\eps^{\frac{\gamma'}{\gamma'-\alpha N}} y)\right).\end{equation}
Note that, using \eqref{assFlocal}, we obtain that 
\begin{equation}\label{assFlocalr}  -C_fm^{\alpha}-K\eps^{\frac{N \alpha \gamma'}{\gamma'-\alpha N}} \leq  f_\eps( m)
\leq -C_fm^{\alpha}+K\eps^{\frac{N \alpha \gamma'}{\gamma'-\alpha N}} ,\end{equation}
Then we get that 
\begin{equation}\label{flimite} \lim_{\eps\to 0} f_\eps(m)= -C_f m^{\alpha}\qquad\text{uniformly in $[0, +\infty)$}.\end{equation}
Moreover, we define $F_\eps(m)=\int_0^m f_\eps(n)dn $ if $m\geq 0$ and $0$ otherwise, and we get 
\begin{equation}\label{Fresc}  -\frac{C_f}{\alpha+1}  m^{\alpha+1}-K\eps^{\frac{N \alpha \gamma'}{\gamma'-\alpha N}}m \leq  F_\eps( m)
\leq -\frac{C_f}{\alpha+1} m^{\alpha+1}+K\eps^{\frac{N \alpha \gamma'}{\gamma'-\alpha N}} m.\end{equation}

We define also the rescaled  Hamiltonian
\begin{equation}\label{Lrescaled}
 H_\eps(p) = \eps^{\frac{N \alpha \gamma'}{\gamma'-\alpha N}} H\left(\eps^{-\frac{N \alpha (\gamma'-1)}{\gamma'-\alpha N}} p \right).
\end{equation}
By \eqref{Hass}, 
\begin{equation} \label{Hassrescaled} 
\begin{gathered}
C_H|p|^\gamma-  \eps^{\frac{N \alpha \gamma'}{\gamma'-\alpha N}}K \leq  H_\eps(p)  \leq C_H  |p|^\gamma,\\
 |\nabla H_\eps(p) | \leq K |p|^{\gamma-1}.
 \end{gathered}
\end{equation}
So, we get that 
\begin{equation}\label{Hlimite}   \lim_{\eps\to 0} H_\eps(p) = H_0(p):=C_H |p|^\gamma\qquad\text{uniformly in $\R^N$}.\end{equation} 
Moreover, if we assume  that $\nabla H_\eps$ is locally bounded in 
$C^{0,\gamma-1}(\R^N)$, then \[\nabla H_\eps(p)\to\nabla H_0(p)=\frac{C_H}{\gamma}|p|^{\gamma-2} p\qquad \text{ locally uniformly.}\]

We can define $L_\eps$ as in \eqref{dati}, with $H_\eps$ in place of $H$ and we obtain that condition \eqref{Hassrescaled} gives that there exists $C_L>0$ such that 
\begin{equation} \label{Lassrescaled} 
C_L|q|^{\gamma'}\leq  L_\eps(q)  \leq C_L  |q|^{\gamma'} +\eps^{\frac{N \alpha \gamma'}{\gamma'-\alpha N}}C_L 
\end{equation}  which in turns gives that 
\begin{equation}\label{Llimite} L_\eps(q)\to L_0(q)= C_L|q|^{\gamma'}\qquad\text{uniformly in $\R^N$}.\end{equation} %\begin{remark}\upshape  If we additionally assume that $C_L|q|^{\gamma'} \le L(q) \le C_L(|q|^{\gamma'} + 1)$, 
%then $L_0(q)=C_L |q|^{\gamma'}$, the convergence is uniform in $\R^N$ and holds along any sequence $\eps \to 0$. Similarly, $H_0(p)=C_H |p|^\gamma$ and the  convergences of $H_\eps\to H_0$ and 
%of $\nabla H_\eps\to \nabla H_0$  are uniform in $\R^N$. On the other hand, if the lower and upper bounds on $L$ (or $H$) have different constants, one has to expect different behaviour along different sequences $\eps \to 0$: consider, for example, $L(q) = (3 + \sin(\log(|q|))|q|^2$ and the corresponding $L_\eps(q) = (3 + \sin(\log( \eps^{-\frac{N \alpha}{2-\alpha N}} |q|))|q|^2$.
%\end{remark}
%

The rescalings \eqref{mfgresc} lead to the following rescaled system
\begin{equation}\label{mfgresc}\begin{cases}
- \Delta \tilde u_\eps+ H_\eps(\nabla\tilde u_\eps)+\tilde \lambda_\eps=f_\eps(\tilde m_\eps) +V_\eps(y)\\
- \Delta \tilde m_\eps-\di(\tilde m_\eps \nabla H_\eps(\nabla \tilde u_\eps) )=0 \\   \int_{\R^N} \tilde m_\eps=M.  \end{cases}
\end{equation}
Existence of a triple $(\tilde u_\eps, \tilde m_\eps, \tilde \lambda_\eps)$ solving the previous system is an immediate consequence of Theorem \ref{exthm}.
 We first start by stating some a priori estimates.

\begin{lemma}\label{existencerescaled}
There exist   $ C, C_1, C_2>0$ independent of $\eps$ such that  the following holds 
 \begin{gather}\label{lambdatilda} 
-C_1\leq\tilde\lambda_\eps\leq-C_2,  \\ \label{intVtilda}  \int_{\R^N} \tilde{m}_\eps |\nabla \tilde u_\eps|^{\gamma}dy+ \int_{\R^N} \tilde{m}_\eps(y)V_\eps(y )dy + \|\tilde m_\eps\|_{L^{\alpha+1}(\R^N)}^{\alpha+1} \leq C,\\
\label{minfinitytilda} 
\|\tilde m_\eps\|_{L^{\infty}(\R^N)}\leq C.\end{gather}
\end{lemma}

\begin{proof}  
Estimates   \eqref{lambda1}, \eqref{intV1} give \eqref{lambdatilda}, \eqref{intVtilda} by rescaling. 
 
We apply Theorem \ref{stimalinfinito}  with $g[m](x)=f_\eps(m(x))$, $r_k=\eps^{\frac{N \alpha (\gamma'-1)}{\gamma'-\alpha N}} $, $s_k= \eps^{\frac{N \alpha \gamma'}{\gamma'-\alpha N}} $ and $t_k= \eps^{\frac{\gamma'}{\gamma'-\alpha N}} $, which are all bounded sequences, and we obtain  \eqref{minfinitytilda}.
 \end{proof}
  
Using the  a priori bounds on the solutions to \eqref{mfgresc}, we want to pass to the limit $\eps \to 0$. The problem is that these estimates are not sufficient to assure that there is no loss of mass, namely that the limit of $\tilde m_\eps$ has still $L^1$-norm equal to $M$. Therefore, we need  to translate the reference system at a point around which the mass of $\tilde m_\eps$ remains positive. 
This will be done as follows.

Let $y_\eps\in\R^N$ be such that \begin{equation}\label{y} \tilde u_\eps(y_\eps)=\min_{\R^N} \tilde u_\eps(y),\end{equation} note that this point exists due to \eqref{uhasmin}. 

We will denote by \begin{equation}\label{resc2} \begin{cases}
\bar u_\eps(y)=\tilde u_\eps(y+y_\eps)-\tilde u_\eps(y_\eps) \\ 
\bar m_\eps(y)=\tilde m_\eps (y+y_\eps).
\end{cases}\end{equation} 
Note that $(\bar u_\eps, \bar m_\eps, \tilde \lambda_\eps)$ is a classical solution to 
\begin{equation} 
\label{mfgresc1}\begin{cases}
- \Delta \bar u_{\eps}+ H_\eps(\nabla\bar u_{\eps})+\tilde \lambda_{\eps}=f_{\eps}(\bar m_{\eps}) +V_\eps(y+y_\eps)\\
- \Delta \bar m_{\eps}-\di(\bar m_{\eps} \nabla H_\eps(\nabla \bar u_{\eps}) )=0 \\   \int_{\R^N} \bar m_{\eps}=M,  \end{cases}
\end{equation}  
and in addition $\bar u_\eps(0)=0=\min_{\R^N} \bar u_\eps$.

\subsection{A preliminary convergence result} 
In this section, we provide some preliminary convergence results, where we are not preventing possible loss of mass in the limit. First of all we need some a priori estimates on the solutions to \eqref{mfgresc1}. 
\begin{proposition}\label{yepsilon}
Let $(\bar u_\eps, \bar m_\eps, \tilde \lambda_\eps)$ be as in \eqref{resc2}. 
Then there exists a constant $C>0$ independent of $\eps$ such that the following hold
\begin{gather} \label{Vtilda}   \eps^{\frac{(N\alpha +b)\gamma'}{\gamma'-N\alpha}} |y_\eps|^b\leq C \qquad \text{and}\qquad
0\leq V_\eps(y+y_\eps)\leq  C( \eps^{\frac{(N\alpha +b)\gamma'}{\gamma'-N\alpha}} |y|^b+1), \\ \label{gradienttilda}  |\nabla \bar u_\eps(y)|\leq C(1+|y|)^{\frac{b}{\gamma}} \qquad \text{ and } \qquad \bar u_\eps(y)\geq C|y|^{1+\frac{b}{\gamma}}-C^{-1},\\  \int_{B_R(0)} \bar m_\eps(y)dy\geq C\qquad \forall R\geq 1\label{intm} .
\end{gather}

Finally, if $\bar w_\eps=-\bar m_\eps\nabla H_\eps(\nabla \bar u_\eps)$, then  $(\bar m_\eps, \bar w_\eps)$ 
 is a minimizer  in the set $\mathcal{K}_{1,M}$ of the energy
\begin{equation}\label{energiarisc} \mathcal{E}_\eps(m, w)=\int_{\R^N} m L_\eps \left(-\frac{ w}{ m}\right) + V_\eps(y+y_\eps)m +F_\eps(m)\, dy,\end{equation}
where $L_\eps$ and $F_\eps$ are defined in Section \ref{rescaledsection}.  
\end{proposition} 
\begin{proof} Since $\bar u_\eps$ is a classical solution, we can compute the equation in $y = 0$, obtaining 
\[H_\eps(0)+\tilde \lambda_\eps\geq f_\eps(\bar m_\eps(0))+V(y_\eps).\] 
Using the a priori estimates \eqref{lambdatilda}, \eqref{minfinitytilda}, \eqref{Hassrescaled}  and the assumption \eqref{assFlocalr}, \eqref{vassresc}, this implies 
that $\eps^{\frac{(N\alpha +b)\gamma'}{\gamma'-N\alpha}} |y_\eps|^b\leq C$, and then, again by assumption \eqref{vassresc}, that \eqref{Vtilda} holds. 

Using estimates \eqref{lambdatilda}, \eqref{minfinitytilda},  and  \eqref{Vtilda}, we conclude by Theorem \ref{bernstein2} that estimate \eqref{gradienttilda} holds.  

Again by the equation computed at $y = 0$, recalling that $H_\eps(0)\to 0$ and $V_\eps\geq 0$ and estimate \eqref{lambdatilda}, we deduce  
that $-f_\eps(\bar m_\eps(0))\geq -C_2>0$. So,  by assumption \eqref{assFlocalr}, we get that there exists $C>0$ indipendent of $\eps$, such that 
 $\tilde m_\eps(0)>C>0$.  Using the estimates \eqref{gradienttilda} and \eqref{minfinitytilda}, by Proposition \ref{ell_regularity}, we get that there exists a positive constant depending on $p$ such that  $\|\bar m_\eps\|_{W^{1,p}(B_2(0))}\leq C_p$ for all $p>1$. This, by Sobolev embeddings, gives that $\|\bar m_\eps\|_{C^{0,\alpha}(B_2(0))}\leq C_\alpha$ for every $\alpha\in (0,1)$ and for some positive constant depending on $\alpha$.  We choose now $R_0\in (0,1]$ such that $\bar m_\eps\geq C/2$ in $B_{R_0}(0)$, using the $C^\alpha$ estimate and the fact that $\bar m_\eps(0)>C>0$. This implies immediately that  $\int_{B_{R_0}(0)} \bar m_\eps(y)dy\geq C/2 |B_{R_0}|>0$. This gives the estimate \eqref{intm}, for all radii bigger than $R_0$.

Finally the fact that $(\bar m_\eps, \bar w_\eps)$ is a minimizer of \eqref{energiarisc} in $\mathcal{K}_{1, M}$ follows from Theorem \ref{exthm}, by rescaling. 
\end{proof}

We get the first convergence result. 
\begin{proposition}\label{propconv}
Let   $(\bar u_\eps, \bar m_\eps, \tilde \lambda_\eps)$ be the  classical solution to \eqref{mfgresc1} constructed above.  
Up to subsequences, we get that $ \tilde \lambda_\eps\to \bar \lambda $,   and   \begin{equation}\label{convergenza1}  
\bar u_\eps \to \bar u,\qquad  \bar m_\eps \to \bar m, \qquad 
\nabla \bar u_\eps\to \nabla \bar u,\qquad \nabla H_\eps(\nabla \bar u_\eps )\to \nabla H_0(\nabla \bar u) \end{equation} 
 locally uniformly, where $\bar u\geq 0=\bar u(0)$, and $(\bar u, \bar m, \bar \lambda)$ is a classical solution to 
\begin{equation}\label{mfglimit}\begin{cases}
- \Delta \bar u+ H_0(\nabla\bar u)+\bar \lambda = - C_f \bar m^\alpha+g(x)\\
- \Delta \bar m -\di(\bar m  \nabla H_0(\nabla \bar u) )=0 \end{cases}
\end{equation}
for a continuous function $g$ such that $0\leq g(x)\leq C$ on $\R^N$ for some $C>0$.  

Moreover,    there exist  $a\in (0, M]$, $C, K, \kappa >0$ such that $ \int_{\R^N}\bar m dx =a$,  and  
\begin{equation}\label{improvedintbarra} 
\bar u (x)\geq C |x| - C, \qquad |\nabla \bar u|\leq K \text{ on $\R^N$}, \qquad \int_{\R^N} e^{\kappa |x|} \bar m(x)dx<+\infty.
\end{equation} 
\end{proposition} 

\begin{proof} 
First of all observe that,   since $V$ is a locally H\"older continuous function, then \eqref{Vtilda} implies that, up to subsequence, 
 $V_\eps(x+y_\eps)\to g(x)$, locally uniformly as $\eps\to 0$, where $g$ is a continuous function such that $0\leq g(x)\leq C$, for some $C>0$.

Using the a priori estimate \eqref{gradienttilda},  and recalling that $\bar u_\eps$ is a classical solution to \eqref{mfgresc1}, by classical elliptic regularity theory we obtain that $\bar u_\eps$ is locally bounded in $C^{1, \alpha}$ in every compact set, uniformly with respect to $\eps$. 
So, up to extracting a subsequence via a diagonalization procedure, we get that \[\bar u_\eps\to \bar u, \quad \nabla \bar u_\eps\to \nabla \bar u, \quad \nabla H_\eps(\nabla\bar u_\eps)\to \nabla H_0(\nabla \bar u)\] locally uniformly, and $\tilde \lambda_\eps\to \bar \lambda$.   Using the estimates \eqref{gradienttilda} and \eqref{minfinitytilda}, by Proposition \ref{ell_regularity}, and by Sobolev embeddings, for every compact set $K\subset \R^N$, we have that $\|\bar m_\eps\|_{C^{0,\alpha}(K)}\leq C_{K,\alpha}$ for every $\alpha\in (0,1)$ and  for some positive constant depending on $\alpha$ and $K$. So, up to extracting a subsequence via a diagonalization procedure, we get that $\bar m_\eps\to \bar m$ locally uniformly.  

So, we can pass to the limit in   \eqref{mfgresc1} and obtain that $(\bar u, \bar m, \bar \lambda)$ is a  solution to \eqref{mfglimit}, which is classical  by elliptic regularity theory.

Using \eqref{intm} and locally uniform convergence, we get that there exists $a\in (0, M]$ such that $\int_{\R^N} \bar m dy=a$. 

Observe that $\bar u$ is a solution to 
\[- \Delta \bar u+ H_0(\nabla\bar u)+\bar \lambda =-C_f\bar m^\alpha+g(x).\]
By Theorem \ref{bernstein2}, we get that there exists a constant $K$ depending on   $\sup g$ and $-\bar \lambda$ such that $|\nabla \bar u|\leq K$. Moreover, by construction $\bar u\geq 0$.

Since $\bar m$ is H\"older continuous, and such that $\int_{\R^N}\bar m dx=a \in (0, M]$, by Lemma \ref{van_lemma}, we get that $\bar m\to 0$ as $|x|\to +\infty$. Therefore, we get that
$\liminf_{|x|\to +\infty} (-\bar m^\alpha(x)+g(x)-\bar \lambda -H_0(0))\geq -\lambda>0$. 
So, by Theorem \ref{boundedfrombelowsol}, recalling that by construction $\bar u(0)=0\leq \bar u(y)$, we get that 
$\bar u $ satisfies \begin{equation}\label{stimeeps3}\bar u(x) \geq C  |x| -C  \end{equation} for some $C>0$. 

To conclude, consider the function $\Phi(x)=e^{\kappa\bar u (x)}$. We claim that we can choose $\kappa>0$ such that there exist $R>0$ and $\delta>0$  with 
\begin{equation}\label{Phidef}  -\Delta \Phi+\nabla H_0 (\nabla \bar u)\cdot \nabla \Phi>\delta \Phi \qquad |x|>R.\end{equation} 
Indeed, since $\bar u$ solves the first equation in  \eqref{mfglimit} , we get \[-\Delta \Phi+\nabla H_0 (\nabla \bar u)\cdot \nabla \Phi\geq \kappa (-\bar\lambda-\kappa |\nabla\bar u|^2- \bar m^\alpha) \Phi.\] 
Using \eqref{stimeeps3} and  $\bar m\to 0$ as $|x|\to +\infty$, we obtain the claim. 
Reasoning as in \cite[Proposition 4.3]{ichi15}, or \cite[Proposition 2.6]{MP05}, we get that $\int_{\R^N} e^{\kappa \bar u}\bar m dx<+\infty$, which concludes the estimate \eqref{improvedintbarra}.

\end{proof}

\begin{remark}\upshape\label{remarkcrescita}  With estimates \eqref{improvedintbarra} in force, the pointwise bounds stated in \cite[Theorem 6.1]{MP05} hold, namely there exist positive constants $c_1, c_2$,  such that
\[
\bar m (x) \le c_1 e^{-c_2|x|} \quad \text{on $\R^N$}.
\]
\end{remark}

\subsection{Concentration-compactness} 
In this section we show that actually  there is no loss of mass when passing to the limit as in Proposition \ref{propconv}.
In order to do so, we  apply a kind of concentration-compactness argument. 

First of all we show that the functional $ \mathcal{E}_\eps (m, w)$  enjoys the following subadditivity property.
Let us denote \[\tilde e_\eps (M)=\min_{(m,w)\in \cK_M} \cE_\eps(m,w).\] 
Recalling \eqref{energybound1} and \eqref{energybound2}, and the rescaling \eqref{rescaling}, we get that for every $M>0$ there exist $C_1(M), C_2(M), K_1, K_2>0$ depending on $M$  (and on the other constants of the problem) but not on $\eps$ 
such that there holds \begin{equation}\label{er}-C_1(M) -K_1\eps^{\frac{N\alpha \gamma'}{\gamma'-N\alpha}}\leq \tilde e_\eps (M)\leq -C_2(M) -K_2\eps^{\frac{N\alpha \gamma'}{\gamma'-N\alpha}}.
\end{equation} 
\begin{lemma}\label{subadd} For all $a \in (0,M)$, there exist $\eps_0=\eps_0(a)$ and a constant $C=C(a,M)\geq 0$ depending only  on $a, M$ and the data (not on $\eps$), 
such that  $C(M, M)=0=C(0,M)$, $C(a, M)>0$ for $0<a<M$ and   
\begin{equation}
\tilde{e}_{\eps}(M) \leq \tilde{e}_{\eps}(a) +  \tilde{e}_{\eps} (M-a)-C(a,M)\qquad \forall \eps\leq \eps_0.\end{equation}

\end{lemma}
\begin{proof}  We assume   that $a \geq M/2$ (otherwise it suffices to replace $a$ with $M-a$).

Let $c >1$ and $B>0$. For all admissible couples $(m,w)\in \cK_B$ we have, recalling \eqref{Fresc},
\begin{multline}\label{subadd2}
\tilde e_\eps(cB)\leq \cE_\eps(c m, c w) = \int_{\R^N} c m L_\eps\left(- \frac{w}{m}\right)+F_\eps(cm) + c V_\eps(x+y_\eps)m\, dx \\
 = c \cE_\eps( m,  w)+ \int_{\R^N} F_\eps(cm)-cF_\eps (m)\, dx\\ \leq c \cE_\eps(m, w) - \frac{c (c^{\alpha}-1)C_f}{\alpha+1}\int_{\R^N}m^{\alpha+1}\, dx+2KcB \eps^{\frac{N\alpha \gamma'}{\gamma'-N\alpha}}.
\end{multline}

Let now $(m_n , w_n )$ be a  minimizing sequence of $\cE_\eps$ in $\cK_{B}$, such that $\cE_\eps(m_n, w_n)\leq \tilde e_\eps(B)+\frac{C_2(B)}{4}$ where $C_2(B)$ is the constant appearing in \eqref{er}, which depends on $B$ and on the data of the problem.  Recalling that $V_\eps\geq 0$ and $L_\eps\geq 0$, and estimate \eqref{Fresc}, 
we get that
\[\tilde e_\eps(M)+\frac{C_2(B)}{4}\geq \cE_\eps(m_n, w_n)\geq \int_{\R^N} F_\eps(m_n)\, dx\geq - \frac{C_f}{\alpha+1} \int_{\R^n} m^{\alpha+1}\, dx-KB \eps^{\frac{N\alpha \gamma'}{\gamma'-N\alpha}}. \]
Using \eqref{er}, we get, for all $\eps$ sufficiently small,
\[\frac{C_f}{\alpha+1}\int_{\R^N}m_n^{\alpha+1}\, dx\geq  \frac{3 C_2(B)}{4} -K\eps^{\frac{N \alpha \gamma'}{\gamma'-\alpha N}}> \frac{C_2(B)}{2}>0.
\]
So, this estimate in particular holds for a minimizer of $\cE_\eps$.
Therefore in \eqref{subadd2} we get, taking $(m,w)$ to be a minimizer of $\cE_\eps$ (which exists by Proposition \ref{yepsilon})
\begin{equation}\label{suba1}
\tilde e_\eps(cB) < c \tilde e_\eps(B)-c (c^\alpha-1)\frac{C_2(B)}{2}+2KcB\eps^{\frac{N\alpha \gamma'}{\gamma'-N\alpha}}.
\end{equation}
Using \eqref{suba1}  with $B=a$ and $c = M/a$  we get 
\[ \tilde e_\eps(M) < \frac{M}{a}\tilde e_\eps(a) -
\frac{M}{a}\left[\left(\frac{M}{a}\right)^{\alpha}-1\right]\frac{C_2(a)}{2}+2KM \eps^{\frac{N\alpha \gamma'}{\gamma'-N\alpha}}.\] 
If $a=M/2$, this permits to conclude, choosing $\eps$ sufficiently small (depending on $a$). If  $a > M/2$, we use \eqref{suba1} with $B=M-a$ and  $c = a/(M-a)$ to get (multiplying everything by $\frac{M-a}{a}$) 
\begin{multline*}\frac{M-a}{a}\tilde e_\eps(a) < \tilde e_\eps(M-a) -
\left[\left( \frac{a}{M-a}\right)^{\alpha}-1\right]\frac{C_2(M-a)}{2}+2K(M-a) \eps^{\frac{N\alpha \gamma'}{\gamma'-N\alpha}}\\
<\tilde e_\eps(M-a) -
\left[\left( \frac{a}{M-a}\right)^{\alpha}-1\right]\frac{C_2(M-a)}{2}+2KM \eps^{\frac{N\alpha \gamma'}{\gamma'-N\alpha}}\leq \tilde e_\eps(M-a) +2KM \eps^{\frac{N\alpha \gamma'}{\gamma'-N\alpha}}.\end{multline*}
 So putting to together the last two inequalities we get 
\begin{multline*} 
\tilde e_\eps(M) < \frac{M}{a}\tilde e_\eps(a) -
\frac{M}{a}\left[\left(\frac{M}{a}\right)^{\alpha}-1\right]\frac{C_2(a)}{2}+2KM \eps^{\frac{N\alpha \gamma'}{\gamma'-N\alpha}} \\
= \tilde e_\eps(a) + \frac{M-a}{a}\tilde e_\eps(a) - 
\frac{M}{a}\left[\left(\frac{M}{a}\right)^{\alpha}-1\right]\frac{C_2(a)}{2}+2KM \eps^{\frac{N\alpha \gamma'}{\gamma'-N\alpha}}
\\
<\tilde e_\eps(a) + \tilde e_\eps(M-a) - 
\frac{M}{a}\left[\left(\frac{M}{a}\right)^{\alpha}-1\right]\frac{C_2(a)}{2}+4KM \eps^{\frac{N\alpha \gamma'}{\gamma'-N\alpha}}
\\
 \leq \tilde e_\eps(a) + \tilde e_\eps(M-a) - 
\frac{M}{a}\left[\left(\frac{M}{a}\right)^{\alpha}-1\right]\frac{C_2(a)}{4}
\end{multline*}
for $\eps$ sufficiently small (depending on $a$). 
\end{proof}

\begin{theorem}\label{nodicotomy} 
Let $(\bar m_\eps, \bar w_\eps)$  be the minimizer of $\mathcal{E}_\eps$ as in  Proposition \ref{yepsilon}.  Let $ \bar u,\bar m$ as in Proposition \ref{propconv}, so that
$\bar m_\eps\to \bar m$, $\bar w_\eps\to \bar w=-\bar m\nabla H_0(\nabla \bar u)$ locally uniformly, and $\bar m$ satisfies the exponential decay \eqref{improvedintbarra}. 
Then,   \begin{equation}\label{stimelimite} \int_{\R^N}\bar m dx =M.
\end{equation}  \end{theorem} 
\begin{proof}   Assume by contradiction that $\int_{\R^N} \bar m\,dx=a$, with $0<a<M$. We fix $\eps_0(a)$ as in Lemma \ref{subadd}, and we consider from now on $\eps\leq \eps_0(a)$. 
Let $\bar c>0$ be such that $\bar m\leq \bar c e^{-|x|}$ (such $\bar c$ exists by Remark \ref{remarkcrescita}). \\
For $R$ sufficiently large (to be chosen later), we define 
\begin{equation}\label{nu}\nu_R(x)=\begin{cases} \bar c e^{-R} & |x|\leq R\\ \bar c  e^{-|x|} & |x|>R.\end{cases} \end{equation}
So in particular $\bar m (x)\leq \nu_R(x)$ for $|x|>R$. 
 
We observe that as $R\to +\infty$ 
\begin{eqnarray}\label{nu1} 
\int_{\R^n}  \nu_R(x) dx  = \bar c\omega_N e^{-R} R^N+\int_{\R^N\setminus B_R} \bar c e^{-|x|} \, dx\leq C  e^{-R} R^N\to  0.
\end{eqnarray}

Since $\bar m_\eps\to \bar m$ and $\nabla H_\eps(\nabla\bar u_\eps)\to \nabla H_0(\nabla \bar u)$ locally uniformly, there exists $\eps_0=\eps_0(R)$ such that 
for all $\eps\leq \eps_0$, 
\begin{equation}\label{in}
|\bar m_\eps-\bar m|+|\nabla H_\eps(\nabla\bar u_\eps)- \nabla H_0(\nabla \bar u)|\leq \bar c e^{-R} \qquad |x|\leq R.
\end{equation} 
We observe  that  for all $\eps\leq \eps_0$, 
\begin{equation}\label{positivo}  \bar m_\eps-\bar m+2\nu_R\geq \nu_R(x) \qquad \forall x \in \R^N. \end{equation} 
Indeed, if $|x|>R$, then $ \bar m_\eps-\bar m+2\nu_R\geq \bar m_\eps+\nu_R\geq \nu_R$, since $\bar m\leq \nu_R$. On the other hand, if $|x|\leq R$, then by \eqref{in}
$\bar m_\eps-\bar m+2\nu_R\geq -\bar c e^{-R}+2 \bar c e^{-R}= \bar c e^{-R}=\nu_R$.
From \eqref{positivo} we deduce that  
\begin{equation}\label{positivo2} |\bar m_\eps-\bar m|\leq \bar m_\eps-\bar m+2\nu_R. \end{equation} 

Moreover, since $\bar m_\eps\to \bar m$ a.e. by Theorem \ref{teobrezis},  
recalling that $\int_{\R^N} \bar m_\eps dx=M$, $\int_{\R^n}\bar m=a$ and using \eqref{nu1} and  \eqref{positivo2}, we have that 
 \begin{eqnarray} \label{bre1} 
 \int_{\R^N} (\bar m_\eps-\bar m+2\nu_R)dx&=& M-a+ 2\int_{\R^N} \nu_Rdx\to M-a \qquad  \text{as } R\to +\infty,\\
 \label{bre3} 
  \lim_{\eps\to  0} \int_{\R^N} \bar m_\eps^{\alpha+1}dx &= & \int_{\R^N}\bar m^{\alpha+1}  dx+\lim_{\eps\to 0}   \int_{\R^N} |\bar m_\eps-\bar m|^{\alpha+1}dx  
\\ &\leq &
 \int_{\R^N}\bar m^{\alpha+1}  dx+\lim_{\eps\to 0}   \int_{\R^N}(\bar m_\eps-\bar m+2\nu_R)^{\alpha+1}dx .\nonumber 
\end{eqnarray}

We claim that
\begin{equation}\label{claim}
\cE_\eps(\bar m_\eps, \bar w_\eps)\geq \cE_\eps(\bar m,\bar w)+\cE_\eps(\bar m_\eps-\bar m+2 \nu_R, \bar w_\eps-\bar w+2\nabla \nu_R)+ o_\eps(1)+o_R(1),
\end{equation} 
where $o_\eps(1)$ is an error such that $\lim_{\eps\to 0} o_\eps(1)=0$.

%Recalling \eqref{convergenza1},  we choose $C$ sufficiently large  such that for all $|x|\leq R$, 
%\begin{equation}\label{stima1} |\bar m_\eps-\bar m|\leq Ce^{-R}\qquad |\nabla H_\eps (\nabla \tilde  u_\eps)-\nabla H_0(\nabla \bar u)|\leq Ce^{-R}.
% \end{equation} 
We consider the function $(m,w)\mapsto mL_\eps\left(-\frac{w}{m}\right)$. This is a convex function in $(m,w)$. We compute  $\nabla_w\left( mL_\eps\left(-\frac{w}{m}\right)\right)
=- \nabla L_\eps\left(-\frac{w}{m}\right)$, so in particular by \eqref{Lassrescaled}  we get 
\begin{equation}\label{l1primo}C_L \left|\frac{w}{m}\right|^{\gamma'-1} - C_L^{-1} \eps^{\frac{N \alpha (\gamma'-1)}{\gamma'-\alpha N}}\leq  |\nabla_w\left( mL_\eps\left(-\frac{w}{m}\right)\right)| \leq C_L^{-1} \left|\frac{w}{m}\right|^{\gamma'-1}  + C_L^{-1} \eps^{\frac{N \alpha (\gamma'-1)}{\gamma'-\alpha N}}.
\end{equation}  
 Moreover, $\partial_m\left( mL_\eps\left(-\frac{w}{m}\right)\right)
=L_\eps\left(-\frac{w}{m}\right) + \frac{w}{m}\cdot \nabla L_\eps\left(-\frac{w}{m}\right)$, therefore, again by  \eqref{Lassrescaled}  we get 
\begin{equation}\label{l2}C_L \left|\frac{w}{m}\right|^{\gamma' } - C_L^{-1} \eps^{\frac{N \alpha (\gamma'-1)}{\gamma'-\alpha N}}\leq |\partial_m\left( mL_\eps\left(-\frac{w}{m}\right)\right)|\leq C_L^{-1} \left|\frac{w}{m}\right|^{\gamma'}  + C_L^{-1} \eps^{\frac{N \alpha (\gamma'-1)}{\gamma'-\alpha N}}.
\end{equation}

\smallskip

Note that 
\[  \int_{\R^N} V_\eps(y+y_\eps)\bar m_\eps\, dx = \int_{\R^N} V_\eps(y+y_\eps) \bar m \,dx + \int_{\R^N} V_\eps(y+y_\eps)(  \bar m_\eps-\bar m +2\nu_R)\,dx -
2 \int_{\R^N} V_\eps(y+y_\eps)\nu_R  \,dx. 
\]
Recalling the estimate \eqref{Vtilda}  and the definition of $\nu_R$, we have 
  \[ 2 \int_{\R^N} V_\eps(y+y_\eps)\nu_R  \,dx \leq C  R^{b+N} e^{-R}.  \]
Hence we obtain  \begin{equation}\label{stimav1} \int_{\R^N} V_\eps(y+y_\eps)\bar m_\eps\, dx \ge \int_{\R^N} V_\eps(y+y_\eps) \bar m \,dx + \int_{\R^N} V_\eps(y+y_\eps)(  \bar m_\eps-\bar m +2\nu_R)\,dx -
C  R^{ b+N} e^{-R}. 
\end{equation} 

\smallskip

By \eqref{bre3} and \eqref{Fresc} we get 
 \begin{multline}\label{stimam1}  \int_{\R^N} F_\eps(\bar m_\eps)\, dx\geq  -\frac{C_f}{\alpha+1}  \int_{\R^N}  \bar m_\eps^{\alpha+1}\, dx -KM\eps^{\frac{N\alpha \gamma'}{\gamma'-\alpha N}}
 \\ \geq -\frac{C_f}{\alpha+1}  \int_{\R^N}  \bar m^{\alpha+1} \,dx -\frac{C_f}{\alpha+1} 
\int_{\R^N}  (  \bar m_\eps-\bar m +2\nu_R)^{\alpha+1}\,dx + o_{\eps} (1)  \\
\geq  \int_{\R^N} F_\eps(\bar m)\, dx + \int_{\R^N} F_\eps(\bar m_\eps-\bar m +2\nu_R)\, dx+ o_{\eps} (1).
 \end{multline} 

\smallskip

Finally, we estimate the kinetic terms in the energy. Splitting 
\[\int_{\R^N}  \bar m_\eps L_\eps\left(-\frac{\bar w_\eps}{\bar m_\eps}\right)dx= \int_{B_R}  \bar m_\eps L_\eps\left(-\frac{\bar w_\eps}{\bar m_\eps}\right)dx+\int_{\R^N\setminus B_R}  \bar m_\eps L\left(-\frac{\bar w_\eps}{\bar m_\eps}\right)dx, \]
we proceed by estimating separately the two terms. 

\smallskip 

{\bf Estimates in $\R^N\setminus B_R$}. 

First of all, note that by \eqref{improvedintbarra},  \eqref{Hassrescaled}  and    \eqref{Lassrescaled}, we get that $L_\eps\left(-\frac{\bar w}{\bar m}\right) = L_\eps(\nabla H_0(\nabla \bar u)) \leq C$ 
for come constant $C>0$, just depending on the data. Moreover, recalling that $\bar m\leq \bar c e^{-|x|}$, we get that, eventually enlarging $C$, 
\begin{equation}\label{Lfuorimbarra}
\int_{\R^N\setminus B_R}  \bar m L_\eps\left(-\frac{\bar w}{\bar m}\right)\, dx \leq C\int_{|x|>R} e^{-|x|} dx \leq CR^Ne^{-R}.
\end{equation}

By convexity of the function $(m,w)\mapsto mL\left(-\frac{w}{m}\right)$, we get that 
\begin{eqnarray}\nonumber \int_{\R^N\setminus B_R}  \bar m_\eps L\left(-\frac{\bar w_\eps}{\bar m_\eps}\right) dx
\geq \int_{\R^N\setminus B_R} ( \bar m_\eps-\bar m +2\nu_R) L_\eps\left(-\frac{\bar w_\eps-\bar w+2\nabla \nu_R}{\bar m_\eps-\bar m +2\nu_R}\right) dx \\  
\label{conve1} + \int_{\R^N\setminus B_R} \partial_m \left(( \bar m_\eps-\bar m +2\nu_R) L_\eps\left(-\frac{\bar w_\eps-\bar w+2\nabla \nu_R}{\bar m_\eps-\bar m +2\nu_R}\right) \right) (\bar m-2\nu_R) \,dx \\ 
\label{conve2} + \int_{\R^N\setminus B_R} \nabla_w \left[ ( \bar m_\eps-\bar m +2\nu_R) L_\eps 
\left(-\frac{\bar w_\eps-\bar w+2\nabla \nu_R}{\bar m_\eps-\bar m +2\nu_R}\right) \right] \cdot (\bar w-2\nabla \nu_R) \, dx.
\end{eqnarray}

We recall  that   $|\bar w|=\bar m |\nabla H_0(\nabla\bar u)|\leq C\bar m$
by \eqref{improvedintbarra}  and $|\nabla \nu_R|\leq C\nu_R$ by definition. Moreover, by \eqref{gradienttilda} and \eqref{Hassrescaled},
\[|\bar w_\eps|=\bar m_\eps |\nabla H_\eps(\nabla \bar u_\eps)|\leq C\bar m_\eps [(1+|x|)^{\frac{b}{\gamma}}]^{\gamma-1}\leq C_1\bar m_\eps (1+|x|)^{\frac{b}{\gamma'}}.\]
Using the triangular inequality 
we get the following, where the constant $C$ can change from line to line, 
\begin{eqnarray}\label{rapporto} \left|\frac{\bar w_\eps-\bar w+2\nabla \nu_R}{\bar m_\eps-\bar m +2\nu_R}\right| 
\leq \frac{\bar m_\eps|\nabla H_\eps(\nabla \bar u_\eps)|}{\bar m_\eps-\bar m +2\nu_R}+\frac{\bar m  |\nabla H_0(\nabla\bar u)|}{\bar m_\eps-\bar m +2\nu_R}+
\frac{C\nu_R}{\bar m_\eps-\bar m +2\nu_R}\\ 
\leq \frac{C \bar m_\eps (1+|x|)^{\frac{b}{\gamma'}}}{\bar m_\eps-\bar m +2\nu_R}+ \frac{C\bar m}{\bar m_\eps-\bar m +2\nu_R}+\frac{C\nu_R}{\bar m_\eps-\bar m +2\nu_R} \leq C (1+|x|)^{\frac{b}{\gamma'}} \nonumber 
\end{eqnarray} 
on $\R^N\setminus B_R(0)$, where we used  respectively the fact that $\bar m_\eps-\bar m +2\nu_R\geq \bar m_\eps$, $\bar m\leq \nu_R$ and that $\bar m_\eps-\bar m +2\nu_R\geq \nu_R$.

Now, using   \eqref{l2}  and \eqref{rapporto}, we can estimate \eqref{conve1}, and by \eqref{l1primo} and \eqref{rapporto} we can estimate \eqref{conve2}.
Indeed, we get 
\[  \int_{\R^N\setminus B_R} \left|\partial_m \left(( \bar m_\eps-\bar m +2\nu_R) L_\eps\left(-\frac{\bar w_\eps-\bar w+2\nabla \nu_R}{\bar m_\eps-\bar m +2\nu_R}\right) \right)\right| |\bar m-2\nu_R| \,dx \leq C \int_{\R^N\setminus B_R}  (1+|x|)^{b} \nu_R(x)dx  \]
and 
\[\int_{\R^N\setminus B_R} \left|\nabla_w \left[ ( \bar m_\eps-\bar m +2\nu_R) L_\eps 
\left(-\frac{\bar w_\eps-\bar w+2\nabla \nu_R}{\bar m_\eps-\bar m +2\nu_R}\right) \right] \right| ( |\bar w| +2|\nabla \nu_R|)dx\leq  C \int_{\R^N\setminus B_R}  (1+|x|)^{\frac{b }{\gamma}}\nu_R(x)dx,  \]
because $\bar w \le C \bar m$ on $\R^N$. Therefore, we may conclude, possibly enlarging $C$, that 
\begin{multline}\label{mepsfuori}
\int_{\R^N\setminus B_R}  \bar m_\eps L\left(-\frac{\bar w_\eps}{\bar m_\eps}\right) dx  \\
\geq \int_{\R^N\setminus B_R} ( \bar m_\eps-\bar m +2\nu_R) L_\eps\left(-\frac{\bar w_\eps-\bar w+2\nabla \nu_R}{\bar m_\eps-\bar m +2\nu_R}\right) dx -C \int_{\R^N\setminus B_R}(1+|x|)^{b} \nu_R(x)dx\\
\geq \int_{\R^N\setminus B_R} ( \bar m_\eps-\bar m +2\nu_R) L_\eps\left(-\frac{\bar w_\eps-\bar w+2\nabla \nu_R}{\bar m_\eps-\bar m +2\nu_R}\right) dx -C R^{N+b}e^{-R}.
\end{multline}
Finally,  putting together \eqref{Lfuorimbarra} and \eqref{mepsfuori}, we have, choosing $C$ suffficiently large 
\begin{multline}\label{fuori} 
\int_{\R^N\setminus B_R}  \bar m_\eps L_\eps\left(-\frac{\bar w_\eps}{\bar m_\eps}\right)dx\geq \int_{\R^N\setminus B_R}  \bar m L_\eps\left(-\frac{\bar w}{\bar m}\right)\, dx\\+\int_{\R^N\setminus B_R} ( \bar m_\eps-\bar m +2\nu_R) L_\eps\left(-\frac{\bar w_\eps-\bar w+2\nabla \nu_R}{\bar m_\eps-\bar m +2\nu_R}\right) dx -C R^{N+b}e^{-R}.
\end{multline} 

\smallskip 

{\bf Estimates in $B_R$}. 
Again by convexity of the function $(m,w)\mapsto mL\left(-\frac{w}{m}\right)$, we get that 
\begin{multline} \label{conve1r}
\int_{ B_R}  \bar m_\eps L\left(-\frac{\bar w_\eps}{\bar m_\eps}\right) dx
\geq \int_{ B_R} \bar m  L_\eps\left(-\frac{ \bar w }{ \bar m }\right) dx \\  
 + \int_{  B_R} \partial_m \left(\bar m L_\eps\left(-\frac{\bar w}{\bar m}\right) \right) (\bar m_\eps- \bar m) \,dx  + \int_{ B_R} \nabla_w \left[ \bar m L_\eps 
\left(-\frac{\bar w}{\bar m}\right) \right] \cdot (\bar w_\eps- \bar w ) \, dx.
\end{multline}

We now estimate \eqref{conve1r}. We recall  that $\left|\frac{\bar w}{\bar m}\right|\leq |\nabla H_0(\nabla \bar u)|\leq K$  and also $|\nabla H_\eps(\nabla \bar u_\eps)|\leq K$ for all $\eps\leq \eps_0(R)$. Then, using this fact and  \eqref{l1primo} and \eqref{l2} and  recalling \eqref{in},  we get 
\[  \int_{B_R}\left| \partial_m \left(\bar m  L_\eps\left(-\frac{\bar w}{\bar m}\right) \right) \right| \left|\bar m_\eps- \bar m\right|  \,dx =   \int_{B_R} \left|\partial_m \left(\bar m L_\eps\left(\nabla H_0(\nabla \bar u)\right) \right) \right| |\bar m_\eps- \bar m| \,dx\leq C e^{-R} R^N\]
and 
\[ \int_{ B_R} \left| \nabla_w \left[ \bar m L_\eps 
\left(\nabla H_0(\nabla \bar u)\right) \right] \right| ( |\nabla H_\eps(\nabla u_\eps)| |\bar m_\eps-\bar m| +|\nabla H_\eps(\nabla \bar u_\eps)- \nabla H_0(\nabla \bar u)| \bar m) \, dx\leq C e^{-R} R^N.
\]
This implies that for all $\eps\leq \eps_0(R)$ 
\begin{equation} \label{mbarrain} \int_{ B_R}  \bar m_\eps L\left(-\frac{\bar w_\eps}{\bar m_\eps}\right) dx
\geq \int_{ B_R} \bar m  L_\eps\left(-\frac{ \bar w }{ \bar m }\right) dx -C e^{-R} R^N.
\end{equation}

Now we observe that by \eqref{Lassrescaled}, 
\[ \int_{B_R} ( \bar m_\eps-\bar m +2\nu_R) L_\eps\left(-\frac{\bar w_\eps-\bar w+2\nabla \nu_R}{\bar m_\eps-\bar m +2\nu_R}\right) dx\leq 
C \int_{B_R} \left[ \left|\frac{\bar w_\eps-\bar w+2\nabla \nu_R}{\bar m_\eps-\bar m +2\nu_R}\right|^{\gamma'} +1\right] (\bar m_\eps-\bar m +2\nu_R)dx .   \]
By \eqref{positivo2} we get that $\bar m_\eps-\bar m +2\nu_R\leq |\bar m_\eps-\bar m |+2\nu_R\leq C e^{-R}$, eventually enalarging $C$. 
Moreover, reasoning as in \eqref{rapporto}, we get 
\begin{align*} \left|\frac{\bar w_\eps-\bar w+2\nabla \nu_R}{\bar m_\eps-\bar m +2\nu_R}\right| \leq 
|\nabla H_\eps(\nabla \bar u_\eps)| \frac{|\bar m_\eps-\bar m|}{\bar m_\eps-\bar m +2\nu_R}+  \frac{|\nabla H_\eps(\nabla \bar u_\eps)-\nabla H_0(\nabla \bar u)| }{\bar m_\eps-\bar m +2\nu_R}\bar m 
 \leq C\end{align*} 
where we used that $\nabla \nu_R=0$ for $|x|<R$,  that $|\nabla H_\eps(\nabla \bar u_\eps)|\leq K$, that  by \eqref{positivo2} $\frac{|\bar m_\eps-\bar m|}{\bar m_\eps-\bar m +2\nu_R}\leq 1$, $\frac{  |\nabla H_\eps(\nabla \bar u_\eps)-\nabla H_0(\nabla \bar u)| }
 {\bar m_\eps-\bar m +2\nu_R} \leq C$ by \eqref{positivo} and \eqref{in}. So, we conclude that 
 \begin{equation} \label{mepsbarrain}\int_{B_R} ( \bar m_\eps-\bar m +2\nu_R) L_\eps\left(-\frac{\bar w_\eps-\bar w+2\nabla \nu_R}{\bar m_\eps-\bar m +2\nu_R}\right) dx\leq 
  C e^{-R} R^N. 
\end{equation}

Putting together \eqref{mbarrain} and \eqref{mepsbarrain} 
we get, choosing $C$ suffficiently large and for all $\eps\leq \eps_0(R)$, 
\begin{multline}\label{dentro} 
\int_{B_R}  \bar m_\eps L_\eps\left(-\frac{\bar w_\eps}{\bar m_\eps}\right)dx\geq \int_{B_R}  \bar m L_\eps\left(-\frac{\bar w}{\bar m}\right)\, dx\\+\int_{B_R} ( \bar m_\eps-\bar m +2\nu_R) L_\eps\left(-\frac{\bar w_\eps-\bar w+2\nabla \nu_R}{\bar m_\eps-\bar m +2\nu_R}\right) dx -C R^N e^{-R}.    
\end{multline}

%%%%%%%%%%%%%%%%%%%%%%%%}
Therefore, summing up \eqref{dentro}, \eqref{fuori}, \eqref{stimav1} and \eqref{stimam1}, 
we conclude for all $\eps\leq \eps_0(R)$, 
\begin{align} \label{claimf} 
\cE_\eps(\bar m_\eps, \bar w_\eps)\geq  \cE_\eps(\bar m,\bar w)+\cE_\eps(\bar m_\eps-\bar m+2 \nu_R, \bar w_\eps-\bar w+2\nabla \nu_R)
+o_\eps(1)-C R^{ b+N} e^{-R}.\end{align}

Let now $c_R= \frac{M-a}{M-a+2\int_{\R^N}\nu_R dx}$.  We have that $c_R\to 1$ as $R\to +\infty$ and $c_R<1$.
In particular, $(c_R(\bar m_\eps-\bar m+2 \nu_R), c_R (\bar w_\eps-\bar w+2\nabla \nu_R))\in \cK_{M-a}$. 
By the same  computation as in  \eqref{subadd2}, we get 
\begin{multline}\label{riga1} c_R \cE_\eps(\bar m_\eps-\bar m+2 \nu_R, \bar w_\eps-\bar w+2\nabla \nu_R)\\
= \cE_\eps(c_R(\bar m_\eps-\bar m+2 \nu_R),c_R( \bar w_\eps-\bar w+2\nabla \nu_R))+\int_{\R^N} c_R F_\eps(\bar m_\eps-\bar m+2 \nu_R)-F_\eps(c_R(\bar m_\eps-\bar m+2 \nu_R))\, dx\\\geq  \cE_\eps(c_R(\bar m_\eps-\bar m+2 \nu_R),c_R( \bar w_\eps-\bar w+2\nabla \nu_R))\\
+c_R \frac{c_R^{\alpha}-1}{\alpha+1}C_f \int_{\R^N} (\bar m_\eps-\bar m+2 \nu_R)^{\alpha+1} dx-2K\left(M-a+2\int_{\R^N}\nu_R dx\right) \eps^{\frac{N\alpha \gamma'}{\gamma'-N\alpha}}.  \end{multline}
Observe that by \eqref{intVtilda} there exists $C$ independent of $\eps$ such that 
\[0\leq \int_{\R^N} (\bar m_\eps-\bar m+2 \nu_R)^{\alpha+1} dx \leq  ( \|\bar m_\eps\|_{\alpha+1} + \|\bar m\|_{\alpha+1} + \|2\nu_R \|_{\alpha+1})^{\alpha+1}\leq C. \]
Therefore, \eqref{riga1} reads (recalling that $c_R<1$ and enlarging the constants $C,K$), 
\begin{multline*} c_R \cE_\eps(\bar m_\eps-\bar m+2 \nu_R, \bar w_\eps-\bar w+2\nabla \nu_R)\\ \geq  \cE_\eps(c_R(\bar m_\eps-\bar m+2 \nu_R),c_R( \bar w_\eps-\bar w+2\nabla \nu_R))+
c_R \frac{c_R^{\alpha}-1}{\alpha+1}C - K M  \eps^{\frac{N\alpha \gamma'}{\gamma'-N\alpha}}\\
\geq \tilde e_\eps\left (M-a\right)+
c_R \frac{c_R^{\alpha}-1}{\alpha+1}C - K M  \eps^{\frac{N\alpha \gamma'}{\gamma'-N\alpha}}.
 \end{multline*}

Using this inequality, and using the fact that  $\cE_\eps(\bar m_\eps, \bar w_\eps)=\tilde e_\eps(M)$ and that  $\cE_\eps(\bar m,\bar w)\geq \tilde e_\eps(a)$,   
 we obtain from \eqref{claimf}
\begin{align*} \tilde e_\eps(M)\geq \tilde e_\eps(a) +  \tilde e_\eps\left (M-a\right)+(1-c_R) \cE_\eps(\bar m_\eps-\bar m+2 \nu_R, \bar w_\eps-\bar w+2\nabla \nu_R) \\
+Cc_R\frac{c_R^\alpha-1}{\alpha+1} - K M  \eps^{\frac{N\alpha \gamma'}{\gamma'-N\alpha}}+o_\eps(1)-C R^{ b+N} e^{-R} \end{align*}
 Moreover by \eqref{er} we get that there exist $K=K(M-a)>0$ such that $\cE_\eps(\bar m_\eps-\bar m+2 \nu_R, \bar w_\eps-\bar w+2\nabla \nu_R) \geq -K$,
therefore the previous inequality gives 
\begin{align} \label{energiasuper}  \tilde e_\eps(M)\geq \tilde e_\eps(a) +  \tilde e_\eps\left (M-a\right)-(1-c_R)K+
Cc_R\frac{c_R^\alpha-1}{\alpha+1}  +o_\eps(1)-C R^{ b+N} e^{-R}. \end{align}
By Lemma \ref{subadd}, we get that 
\[ \tilde e_\eps(M)\leq  \tilde e_\eps(a) +\tilde e_\eps\left (M-a\right)-C(a, M),\] where $C(a, M)>0$ for $a<M$ and $C(M,M)=0$. 
This implies in particular that 
\[0>-C(a, M)\geq -(1-c_R)K+
Cc_R\frac{c_R^\alpha-1}{\alpha+1}  +o_\eps(1)-C R^{ b+N} e^{-R}. \]
%%%%%%%%%%%
Recalling that $c_R\to 1$ as $R\to +\infty$, this gives a contradiction, choosing $R$ sufficiently large and $\eps<\eps_0(R)$.  
\end{proof}
%%%%%%%%%%%%%%%%%%%%%%%%}

An immediate  corollary of the previous theorem is the following convergence result.
\begin{corollary}\label{convergenzainnorma}
Let $(\bar u_\eps, \bar m_\eps, \tilde \lambda_\eps)$  and  $(\bar u,  \bar m, \bar \lambda)$ be  as in Proposition \ref{propconv}.
Then,  \begin{equation}\label{mconvergenza} \bar m_\eps \to \bar m \qquad \text{ in }L^1(\R^N) \text { and }L^{\alpha+1}(\R^N).  \end{equation} 

Finally for all $\eta>0$, there exist $R>0$ and $\eps_0$ such that 
for all $\eps\leq \eps_0$,
\begin{equation}\label{intmnuovo}\int_{B(0,R)} \bar m_\eps dx \geq M-\eta.\end{equation}
\end{corollary} 
\begin{proof} By Proposition \ref{propconv} we get that $\bar m_\eps\to \bar m$ almost everywhere, and by Theorem \ref{nodicotomy}, $\int_{\R^N} \bar m_\eps=M=\int_{\R^N}\bar m$. This implies the convergence in $L^1(\R^N)$. Indeed, by Fatou lemma
\[2M\leq  \liminf_\eps \int_{\R^N} \bar m_\eps+\bar m-|\bar m_\eps-\bar m| dx\leq 2M- \limsup_\eps \int_{\R^N} |\bar m_\eps-\bar m|\, dx.\]
Moreover, recalling \eqref{minfinitytilda}, we get that
\[\|\bar m_\eps-\bar m\|_{L^{\alpha+1}(\R^N)}^{\alpha+1} \leq \|\bar m_\eps-\bar m\|_{L^{1}(\R^N)} (\|\bar m\|_{L^{\infty}(\R^N)}+  \|\bar m_\eps\|_{L^{\infty}(\R^N)})\to 0.\]

Finally observe that for all $R$, by  Remark \ref{remarkcrescita}, 
\[ \int_{B_R(0)} \bar m_\eps dy\geq \int_{B_R(0)} \bar m dy-\int_{B_R(0)} |\bar m_\eps-\bar m|dy\geq  
M- CR^{N-1} e^{-R} -\int_{\R^N} |\bar m_\eps-\bar m| dy.\]
So, using    the $L^1$ convergence we conclude the desired estimate. 
\end{proof}

\subsection{Existence of ground states.}

In this subsection we aim at proving that as $\eps$ goes to zero, $(\bar u_\eps, \bar m_\eps, \tilde \lambda_\eps)$ converges to a solution of the limiting MFG system \eqref{mfglimit2}, without potential terms. In particular, we will prove Theorem \ref{fullconv}. 

We first need a  $\Gamma$-convergence type result, proved in the following lemma.
\begin{lemma}\label{asin} 
Let $(m_\eps, w_\eps), (m,w)\in \mathcal{K}_{1,M}$  be such that $m_\eps\to m$ in $L^1\cap L^{\alpha+1}(\R^N)$ and $w_\eps\rightharpoonup w$ weakly in $L^q(\R^N)$ for some $q > 1$. Then 
\begin{equation}\label{en1} \lim\inf_\eps \mathcal{E}_\eps(m_\eps, w_\eps)\geq \mathcal{E}_0( m, w),\end{equation}
where $\mathcal{E}_0$ is defined in \eqref{energia0}. 

Let $(m,w)\in \mathcal{K}_{1,M}$ be such that $m(1+|y|^b)\in L^1(\R^N)$. Then 
\begin{equation}\label{en1gamma} \lim_\eps \mathcal{E}_\eps(m(\cdot-y_\eps),w(\cdot-y_\eps))\leq \mathcal{E}_0(m,w).\end{equation}
\end{lemma} 
\begin{proof}
We recall  that $L_\eps(q)\to C_L|q|^{\gamma'}$ uniformly in $\R^N$ by \eqref{Lassrescaled}  and $F_\eps(m)\to -\frac{1}{\alpha+1}m^{\alpha+1}$ uniformly in $[0, +\infty)$ by \eqref{Fresc}. Moreover we observe that    the energy $\mathcal{E}_0$ is lower semicontinuous with respect to weak $L^q$ convergence of $w$ and strong $L^{\alpha+1} \cap L^1$  convergence of $m$.
Since $V\geq 0$, we get  that 
\begin{multline*}
\lim\inf_\eps \mathcal{E}_\eps(m_\eps, w_\eps)
\geq \lim\inf_\eps \int_{\R^N} m_\eps L_\eps\left(-\frac{ w_\eps}{m_\eps}\right) +F_\eps( m_\eps)\, dx \\ 
\geq \lim\inf_\eps \int_{\R^N}C_L  m_\eps^{1-\gamma'} |  w_\eps|^{\gamma'} -\frac{C_f}{\alpha+1} m_\eps^{\alpha+1}\, dx\\
\geq \int_{\R^N}C_L  m^{1-\gamma'} |  w|^{\gamma'} -\frac{C_f}{\alpha+1}  m^{\alpha+1}\, dx=\mathcal{E}_0(m, w).
 \end{multline*}

Now we observe that for all $m$ such that $m(1+|y|^b)\in L^1(\R^N)$, using \eqref{vassresc},  we get that \begin{equation}\label{en2}\lim_{\eps\to 0} \int_{\R^N} m(y+y_\eps)V_\eps(y+y_\eps)dy\leq \lim_\eps 
C_V  \eps^{\frac{N \alpha \gamma'}{\gamma'-\alpha N}} \int_{\R^N} (1+|y|)^{b}m(y)dy=0.\end{equation}
Therefore, recalling again the uniform convergence of $L_\eps(q)\to C_L|q|^{\gamma'}$ and $F_\eps(m)\to -\frac{1}{\alpha+1}m^{\alpha+1}$, we conclude (noting that if we translate $m,w$ of $y_\eps$ the energy $\mathcal{E}_0$ remains the same)  \[\lim_\eps \mathcal{E}_\eps(m(\cdot-y_\eps),w(\cdot-y_\eps))=\mathcal{E}_0(m,w)+\lim_{\eps\to 0} \int_{\R^N} m(y+y_\eps)V_\eps(y+y_\eps)dy\leq \mathcal{E}_0(m,w). \]

\end{proof}

\begin{proof}[Proof of Theorem \ref{fullconv}]
We first show that $(\bar u, \bar m)$ obtained in Proposition \ref{propconv} are associated to minimizers of an appropriate energy, without potential term, so that \eqref{min2} holds. 

Note that $(\bar m, \bar w)\in \mathcal{K}_{1,M}$ where $\bar w=-\bar m \nabla H_0(\nabla \bar u)$, due to Proposition \ref{propconv} and Theorem 
\ref{nodicotomy} and    $\bar m(1+|y|^b)\in L^1(\R^N)$ by the exponential decay \eqref{improvedintbarra}. Moreover $\bar m_\eps\to \bar m$ in $L^1\cap L^{\alpha+1}$ by Corollary \ref{convergenzainnorma} and $\bar w_\eps=-\bar m_\eps \nabla H_\eps(\nabla \bar u_\eps)\to \bar w=-\bar m\nabla H_0(\nabla \bar u)$  locally uniformly (by Proposition \ref{propconv}) and  weakly in $L^{\frac{\gamma'(\alpha+1)}{\gamma'+\alpha}}$ by the same argument as in the proof of Theorem \ref{existence}.

Let now $(m,w)\in \mathcal{K}_{1,M}$ be such that $m(1+|y|^b)\in L^1(\R^N)$.
Using the minimality of $(\bar m_\eps, \bar w_\eps)$, \eqref{en1} and \eqref{en1gamma}, we conclude that 
\[\mathcal{E}_0(m,w)\geq \lim_\eps  \mathcal{E}_{\eps}(m(\cdot-y_\eps),w(\cdot-y_\eps)) \geq \lim_\eps  \mathcal{E}_{\eps}(\bar m_\eps, \bar w_\eps)\geq \mathcal{E}_0(\bar m,\bar w).\]
This implies  \eqref{min2}.
\smallskip

To obtain the first part of the theorem, that is the existence of a solution to \eqref{mfglimit2}, we need to prove that the function $g$ appearing in Proposition \ref{propconv} is actually zero on $\R^N$. 
To do that, we derive a better estimate on the term $V_\eps(y+y_\eps)$, in particular we show that  $V_\eps(y+y_\eps)\to 0$ locally uniformly in $\R^N$. 

By minimality of $(\bar m_\eps, \bar w_\eps)$ and $(\bar m, \bar w)$, \eqref{Lassrescaled}, \eqref{Fresc}  and \eqref{en2} we get  that
\begin{multline*} \mathcal{E}_\eps(\bar m_\eps, \bar w_\eps)\leq \mathcal{E}_\eps (\bar m(\cdot+y_\eps), \bar w(\cdot+y_\eps)) \\
\leq \mathcal{E}_0(\bar m, \bar w) +\int_{\R^N} \bar m(y+y_\eps)V_\eps(y+y_\eps)dy+C \eps^{\frac{N\alpha \gamma'}{\gamma'-N\alpha}}
\leq \mathcal{E}_0(\bar m_\eps, \bar w_\eps) +C_1 \eps^{\frac{N\alpha \gamma'}{\gamma'-N\alpha}}. \end{multline*}
Again using \eqref{Fresc} and \eqref{Lassrescaled}  we get
\[ \mathcal{E}_0(\bar m_\eps, \bar w_\eps) +C_1 \eps^{\frac{N\alpha \gamma'}{\gamma'-N\alpha}}\leq    \int_{\R^N} \bar m_\eps L_\eps \left(-\frac{\bar w_\eps}{\bar m_\eps}\right)+F_\eps (\bar m_\eps)dy +C\eps^{\frac{N\alpha \gamma'}{\gamma'-\alpha N}}M+C \eps^{\frac{N\alpha \gamma'}{\gamma'-N\alpha}}.\] 
So, putting together the last two inequalities, we conclude that
\begin{equation}\label{intvv}\int_{\R^N} \bar m_\eps V_\eps(y+y_\eps)dy\leq C \eps^{\frac{N\alpha \gamma'}{\gamma'-N\alpha}}.\end{equation}
Recalling \eqref{riscalata}, this implies that for all $R>0$, we get
\[C_V^{-1} (\max\{ \eps^{\frac{\gamma'}{\gamma'-\alpha N}}|y_\eps|-\eps^{\frac{\gamma'}{\gamma'-\alpha N}}R- C_V, 0\})^{b} \int_{B(0,R)}\bar m_\eps dy\leq C.\]
Using \eqref{intmnuovo}, we conclude that there exists $C>0$ such that
\begin{equation}\label{puntoy} \eps^{\frac{\gamma'}{\gamma'-\alpha N}}|y_\eps|\leq C.\end{equation} 
In turns this gives, recalling again \eqref{riscalata}, that
\[0\leq V_\eps(y+y_\eps)\leq C_V \eps^{\frac{N \alpha \gamma'}{\gamma'-\alpha N}} (1 + \eps^{\frac{\gamma'}{\gamma'-\alpha N}}|y|+\eps^{\frac{\gamma'}{\gamma'-\alpha N}}|y_\eps|)^{b}\leq C  \eps^{\frac{N \alpha \gamma'}{\gamma'-\alpha N}} (1 +|y|)^{b} \]
which implies that $V_\eps(y+y_\eps)\to 0$ locally uniformly. 
\end{proof} 

\begin{remark}\label{remgs} \upshape If $H$ and $f$ satisfy the growth conditions \eqref{Hass} and \eqref{assFlocal}, arguing as before one has that there exists a classical  solution to the potential-free version of \eqref{mfg}, 
\begin{equation}\label{mfgnopot}\begin{cases}
- \Delta u+ H(\nabla u)+ \lambda = f(m)\\
- \Delta  m - \di( \nabla H(\nabla  u)m)=0\\\int_{\R^N}  m=M. \end{cases}
\end{equation}
In addition, $(m, -\nabla H(\nabla  u)m)$ is a minimizer of \[
(m, w) \mapsto \int_{\R^N} mL\left(-\frac{w}{m}\right) + F(m)dx
\]
among $(m,w)\in \mathcal{K}_{1, M},\ m(1+|y|^b)\in L^1(\R^N)$. This can be done as follows: start with a sequence $(u_\delta, m_\delta, \lambda_\delta)$ solving
\begin{equation}\label{mfgdelta}\begin{cases}
- \Delta u_\delta+ H(\nabla u_\delta)+ \lambda_\delta= f(m_\delta) + \delta|x|^b\\
- \Delta  m_\delta - \di( \nabla H(\nabla  u_\delta)m_\delta)=0\\\int_{\R^N}  m_\delta=M. \end{cases}
\end{equation}
with $\delta = \delta_n \to 0$. Such a sequence exists by Theorem \ref{exthm}. The problem of passing to the limit in \eqref{mfgdelta} to obtain \eqref{mfgnopot} is the same as passing to the limit in \eqref{mfgresc}, and it is even simpler: in \eqref{mfgresc}, one has to be careful as the Hamiltonian $H_\eps$ and the coupling $f_\eps$ vary as $\eps \to 0$ (still, they converge uniformly), while in \eqref{mfgdelta} they are fixed, and only the potential is vanishing. We observe that $b > 0$ could be chosen arbitrarily, the perturbation $\delta|x|^b$ always disappears in the limit. Still, the limit $m,u$ somehow retains a memory of $b$ in terms of energy properties: $m$ minimizes an energy among competitors satisfying $m(1+|y|^b)\in L^1(\R^N)$.
\end{remark}
 
\begin{remark} \upshape We stress that uniqueness of solutions for \eqref{mfglimit2} does not hold in general; for example, a triple $(u, m, \lambda)$ solving the system may be translated in space to obtain a full family of solutions. On the other hand, a more subtle issue is the uniqueness of $m$ in the second equation (with $\nabla u$ fixed), that is, if $(u, m_1,  \lambda)$ and $(u, m_2, \lambda)$ are solutions, then $m_1 \equiv m_2$. This property is intimately related to the ergodic behaviour of the optimal trajectory $d X_s = -\nabla H_0(\nabla u (X_s)) ds + \sqrt{2 \eps} \, d B_s$ (see, for example, \cite{c14} and references therein). It is well-known that uniqueness for the Kolmogorov equation is guaranteed by the existence of a so-called Lyapunov function; in our cases, it can be checked that $u$ itself (or increasing functions of $u$, as in \eqref{Phidef}) acts as a Lyapunov function, so uniqueness of $m$ and ergodicity holds for \eqref{mfglimit2} and \eqref{mfg}.
\end{remark}
\subsection{Concentration of mass} 
The last problem we address is the localization of the point $y_\eps$, to conclude the proof of Theorem \ref{concthm}. Rewriting \eqref{intmnuovo} in view of \eqref{rescaling} and \eqref{resc2}, we get that for all $\eta>0$ there exist $R, \eps_0$ such that for all $\eps\leq \eps_0$,

\begin{equation}\label{conc1} \int_{B(\eps^{\frac{\gamma'}{\gamma'-\alpha N}}y_\eps, \eps^{\frac{\gamma'}{\gamma'-\alpha N}}R)} m(x)dx\geq M-\eta,\end{equation}
where $m$ is the  classical solution to \eqref{mfg} given in Theorem \ref{exthm}, and $\bar m_\eps(y ) = \eps^{\frac{N \gamma'}{\gamma'-\alpha N}}  m(\eps^{\frac{\gamma'}{\gamma'-\alpha N}} y+\eps^{\frac{\gamma'}{\gamma'-\alpha N}} y_\eps)$. 

By \eqref{puntoy}, we know that, up to subsequences, $\eps^{\frac{\gamma'}{\gamma'-\alpha N}}y_\eps\to \bar x$. Our aim is to locate this point, which is the point where mass concentrates. We need a preliminary lemma stating the existence of suitable competitors that will be used in the sequel.

\begin{lemma} For all $\eps \le \eps_0$, there exists $(\hat m_\eps, \hat w_\eps)\in \mathcal{K}_{1, M}$ that minimize
\begin{equation}\label{energiann} (m, w) \mapsto \int_{\R^N} m L_\eps \left(-\frac{ w}{ m}\right) +F_\eps(m)\, dy
\end{equation}
among  $(m,w)\in \mathcal{K}_{1, M},\ m(1+|y|^b)\in L^1(\R^N)$.
Moreover, for some positive constants $c_1, c_2$ independent of $\eps$,
\begin{equation}\label{expdecay}
\hat m_\eps (y) \le c_1 e^{-c_2|y|} \qquad \text{on $\R^N$}.
\end{equation}
\end{lemma}
\begin{proof} The existence of $(\hat m_\eps, \hat w_\eps)$ is stated in Remark \ref{remgs}, together with a solution $(\hat u_\eps, \hat m_\eps, \hat \lambda_\eps)$ to the associated MFG system as the optimality conditions (see \eqref{sys9} below). To obtain the uniform exponential decay, we can argue by Lyapunov functions as in Proposition \ref{propconv}; here, we have to be careful, since the  argument in  Proposition \ref{propconv} mainly require
\[
f_\eps(\hat m_\eps)-\hat \lambda_\eps -H_\eps(0)\geq - \hat \lambda_\eps /2 > 0
\]
outside some fixed ball $B_r(0)$. This claim can be proved as follows: first, $-\hat \lambda_\eps$ is bounded away from zero for $\eps$ small. Indeed, 
\[
\hat \lambda_\eps M = \int_{\R^N} \hat m_\eps L_\eps \left(-\frac{ \hat w_\eps}{\hat m_\eps}\right) +f_\eps(\hat m_\eps)\hat m_\eps\, dy \le \mathcal E_\eps(\bar m_\eps, \bar w_\eps) + o_\eps(1) \le -C.
\]
The inequality follows by minimality of $(\hat m_\eps, \hat w_\eps)$ and $(\bar m_\eps, \bar w_\eps)$, and (rescaled) \eqref{energybound2}.

We now prove that $\hat m_\eps$ decays as $|x| \to \infty$ uniformly in $\eps$. Note that $\hat w_\eps = -\nabla H_\eps(\nabla  \hat u_\eps)\hat m_\eps$, where $(\hat u_\eps, \hat m_\eps, \hat \lambda_\eps)$ solves
\begin{equation}\label{sys9}
\begin{cases}
- \Delta  \hat u_\eps+ H_\eps(\nabla  \hat u_\eps)+ \lambda = f_\eps(\hat m_\eps)\\
- \Delta  \hat m_\eps - \di( \nabla H_\eps(\nabla   \hat u_\eps)\hat m_\eps)=0\\\int_{\R^N}  \hat m_\eps=M. \end{cases}
\end{equation}
We derive local estimates for $\hat u_\eps$ and $\hat m_\eps$. We shift the $x$-variable so that $\hat u_\eps(0) = 0 = \min_{\R^N} \hat u_\eps$ for all $\eps$. Choose $p > N$ such that
\[
\alpha < \frac{\gamma'}{p} < \frac{\gamma'}{N}.
\]
If one considers the HJB equation solved by $\hat u_\eps$, recalling \eqref{assFlocalr} and \eqref{Hassrescaled}, Theorem \ref{bernstein2}  gives the existence of $C > 0$ such that \[\|\nabla \hat u_\eps\|_{L^{\infty}(B_{2R}(x_0))} \le K(\|\hat m_\eps\|_{L^{\infty}(B_{4R}(x_0))}^\alpha+1)^{\frac{1}{\gamma}}.\] Note that $C > 0$ does not depend on $\eps$ and $x_0$. Turning to the Kolmogorov equation, again by \eqref{Hassrescaled} and Proposition \ref{ell_regularity},
\[
\|\hat m_\eps\|_{W^{1,p}(B_{R}(x_0))}\le C (\|\nabla \hat u_\eps\|^{\gamma-1}_{L^\infty(B_{2R}(x_0))} + 1) \|m_\eps\|_{L^p(B_{2R}(x_0))}.
\]
By the previous $L^\infty$ estimate on $\nabla u_\eps$ and interpolation of the $L^p$ norm of $m$ between $L^1$ and $L^\infty$ we get
\[
\|\hat m_\eps\|_{W^{1,p}(B_{R}(x_0))}\le C(\|\hat  m_\eps\|^{\frac{\alpha}{\gamma'}}_{L^\infty(B_{4R}(x_0))} + 1)\|\hat  m_\eps\|_{L^1(B_{4R}(x_0))}^{1/p}\|\hat  m_\eps\|_{L^\infty(B_{4R}(x_0))}^{1-1/p}.
\]
Recall that $\|\hat m_\eps\|_{L^1(B_{4R}(x_0))} \le M$; then, since $p > N$, by Sobolev embeddings we obtain that for some $\beta > 0$,
\begin{equation}\label{holdere1}
\|\hat m_\eps\|_{C^{0,\beta}(B_{R}(x_0))} \le C(\|\hat m_\eps\|^{\frac{\alpha}{\gamma'}}_{L^\infty(\R^N)} + 1)\|\hat m_\eps\|_{L^\infty(\R^N)}^{1-1/p}.
\end{equation}
First, since $C$ does not depend on $x_0$, this yields $\|\hat m_\eps\|_{L^\infty(\R^N)} \le C$, by the choice of $p < \gamma'/\alpha$. Secondly, plugging back this estimate into \eqref{holdere1}, we conclude $\|\hat m_\eps\|_{C^{0,\beta}(\R^N)} \le C $. 

Then, using these estimates, we get that up to subsequences, $\hat \lambda_\eps\to \hat \lambda$, $\hat u_\eps\to \hat u$ locally uniformly in $C^1$, and  
$\hat m_\eps\to \hat m$ locally uniformly, where $(\hat u, \hat m, \hat \lambda)$ is a solution to \eqref{mfglimit} with $g\equiv 0$. Arguing exactly as in Proposition \ref{propconv}, we get
that $\tilde u$, $\tilde m$ satisfy  the estimates  \eqref{improvedintbarra} (eventually modifying the constants). Moreover $\int_{\R^N} \hat m\,dx=a\in (0, M]$.
Observe now that Lemma \ref{subadd} and Theorem \ref{nodicotomy} hold also for the energy \eqref{energiann}, since it coincides with the energy $\mathcal{E}_\eps$ without  the potential term $\int_{\R^N} V_\eps m\, dx$.  Therefore we can apply Theorem \ref{nodicotomy} to $\hat m$, to conclude that actually $\int_{\R^N}\hat m\, dx=M$. So, by Corollary \ref{convergenzainnorma}, we obtain that for all $\eta>0$, there exist $R>0$ and $\eps_0$ such that 
for all $\eps\leq \eps_0$,
\begin{equation}\label{concmeps}\int_{B(0,R)}  \hat m_\eps dx \geq M-\eta.\end{equation}
By \eqref{holdere1} and \eqref{concmeps}, using Lemma \ref{van_lemma}, we get that 
\[
f_\eps(\hat m_\eps) \ge \frac{\hat \lambda_\eps}{4}
\]
outside a ball $B_r(0)$. Since $H_\eps(0) \to 0$, the claim
\begin{equation}\label{est99}
f_\eps(\hat m_\eps)-\hat \lambda_\eps -H_\eps(0)\geq - \hat \lambda_\eps /2 > 0
\end{equation}
outside a ball $B_r(0)$ follows. As previously mentioned, me may now proceed and conclude as in Proposition \ref{propconv}; basically, \eqref{est99} implies that $x \mapsto e^{k \hat u_\eps(x)}$ acts as a Lyapunov function for $\hat m_\eps$ for some small $k > 0$, giving
\[
c \int_{\R^N} e^{k |x| - k_1} \hat m_\eps \le \int_{\R^N} e^{k \hat u_\eps} \hat m_\eps \le C
\]
for all $\eps$ small, that easily implies the pointwise exponential decay \eqref{expdecay} of $\hat m_\eps$ by H\"older regularity of $\hat m_\eps$ itself.
\end{proof}

For general potentials, the point where mass concentrates is a minimum for $V$.  
\begin{proposition}\label{minimo} Up to subsequences,  $\eps^{\frac{\gamma'}{\gamma'-\alpha N}}y_\eps\to \bar x$, where $V(\bar x)=0$, i.e. $\bar x$ is a minimum of $V$. 
\end{proposition} 
\begin{proof} Fix a generic $z\in \R^N$ and observe that $(\hat m_\eps(\cdot+z), \hat w_\eps(\cdot+z))$ is still a minimizer of $ \int m L_\eps \left(-\frac{ w}{ m}\right) +F_\eps(m)$.  
By  minimality of $(\bar m_\eps, \bar w_\eps)$ and of  $(\hat m_\eps(\cdot+z), \hat w_\eps(\cdot+z))$, 
we get that \begin{multline*} \int_{\R^N} \bar m_\eps L_\eps \left(-\frac{\bar w_\eps}{ \bar m_\eps}\right) +F_\eps(\bar m_\eps)dy +  \int_{\R^N} \bar m_\eps(y) V_\eps(y+y_\eps)dy =  \mathcal{E}_\eps(\bar m_\eps, \bar w_\eps) \\ \leq 
\mathcal{E}_\eps(\hat m_\eps(\cdot + z), \hat w_\eps(\cdot + z) ) \leq \int_{\R^N} \bar m_\eps L_\eps \left(-\frac{\bar w_\eps}{ \bar m_\eps}\right) +F_\eps(\bar m_\eps) +  \int_{\R^N} \hat m_\eps(y+z) V_\eps(y+y_\eps)dy.\end{multline*}
In particular this gives that \begin{equation}\label{dise} \int_{\R^N} \bar m_\eps(y) V_\eps(y+y_\eps)dy\leq \int_{\R^N} \hat m_\eps(y+z) V_\eps(y+y_\eps)dy= \int_{\R^N} \hat m_\eps(y) V_\eps(y+y_\eps-z)dy\qquad \forall z\in\R^N.\end{equation} 
Recalling the rescaling of $V_\eps$  and of $\bar m_\eps$ in \eqref{rescaling}, this is equivalent to 
 \begin{equation}\label{dise2} \int_{\R^N} m (x) V(x)dx 
 \leq \int_{\R^N} \hat m_\eps(y) V(\eps^{\frac{\gamma'}{\gamma'-\alpha N}}y+\eps^{\frac{\gamma'}{\gamma'-\alpha N}}y_\eps-\eps^{\frac{\gamma'}{\gamma'-\alpha N}}z)dy\qquad \forall z\in\R^N\end{equation} 
where $m$ is  the classical solution to \eqref{mfg} given in Theorem \ref{exthm}, such that $\bar m_\eps(y ) = \eps^{\frac{N \gamma'}{\gamma'-\alpha N}}  m(\eps^{\frac{\gamma'}{\gamma'-\alpha N}} y+\eps^{\frac{\gamma'}{\gamma'-\alpha N}} y_\eps)$. 

By \eqref{puntoy}, we get that up to passing to a  subsequence, $\eps^{\frac{\gamma'}{\gamma'-\alpha N}}y_\eps\to \bar x$ for some $\bar x\in\R^N$. Then 
by \eqref{conc1}, we get that 
\begin{equation}\label{alto} \liminf_{\eps\to 0} \int_{\R^N} m(x) V(x)dx \geq \liminf_{\eps\to 0}   \int_{B(\eps^{\frac{\gamma'}{\gamma'-\alpha N}}y_\eps, \eps^{\frac{\gamma'}{\gamma'-\alpha N}}R)} m(x)V(x)dx \geq (M-\eta)V(\bar x).\end{equation} 

We fix $\bar z$ such that $V(\bar z)=0$ and we choose in \eqref{dise2} $z= y_\eps-\eps^{-\frac{\gamma'}{\gamma'-\alpha N}}\bar z$.  
%Observe that the exponential decay of $\bar m$ 
%proved in \eqref{improvedintbarra} implies that  there exist positive constants $c_1, c_2$,  such that
%$\bar m (y) \le c_1 e^{-c_2|y|}$ on $\R^N$, by the 
 %pointwise bounds stated in \cite[Theorem 6.1]{MP05}. 
 We have, by the Lebesgue convergence theorem and \eqref{expdecay}, 
\begin{equation}\label{basso} \limsup_{\eps \to 0}\int_{\R^N} \hat m_\eps(y) V(\eps^{\frac{\gamma'}{\gamma'-\alpha N}}y+\bar z)dy\leq \limsup_{\eps \to 0} c_1 \int_{\R^N}e^{-c_2|y|} V(\eps^{\frac{\gamma'}{\gamma'-\alpha N}}y+\bar z)dy= 0.\end{equation}
By \eqref{alto}, \eqref{basso} and \eqref{dise2}, we conclude $V(\bar x)=0$. 

\end{proof} 

If we assume that the potential   $V$  has  a finite number of minima and polynomial behavior, that is, it satisfies assumption \eqref{Vpoly}, then we get that at the limit $\eps^{\frac{\gamma'}{\gamma'-\alpha N}}y_\eps $ selects at the limit the  more stable minima of $V$, as we will show in the next proposition.

\begin{proposition}\label{teoy} 
Assume that $V$ satisfies assumption \eqref{Vpoly}. Then, up to subsequences, there holds that 
\[\eps^{\frac{\gamma'}{\gamma'-\alpha N}}y_\eps\to x_i\qquad \text{as $\eps\to 0$}\]
where  $i\in \{j=1,\dots, n, \ | \ b_j=\max_{k} b_k\}$. 
\end{proposition} 

\begin{proof} 
By Proposition \ref{minimo}, we know that up to subsequences, $\eps^{\frac{\gamma'}{\gamma'-\alpha N}}y_\eps\to  x_{\iota}$ for some 
$\iota=1,\dots n$.  It remains to prove that $b_{\iota}=\max_i b_i$. Assume by contradiction that it is not true, and then $b_{\iota}<\max_i b_i$. 

 We compute for $j\in 1,\dots n$, recalling  the uniform exponential decay of $\hat m_\eps$ given in \eqref{expdecay},
\begin{multline}\label{uno}
\int_{\R^n} \hat m_\eps(y+y_\eps-\eps^{-\frac{\gamma'}{\gamma'-\alpha N}}x_j) V_\eps(y+y_\eps)dy =
\int_{\R^n} \hat m_\eps(y) V_\eps(y+\eps^{-\frac{\gamma'}{\gamma'-\alpha N}}x_j)dy\\\leq 
C_V\eps^{\frac{\gamma'N\alpha}{\gamma'-N\alpha}}\int_{\R^n}\hat m_\eps(y) \eps^{\frac{b_j\gamma'}{\gamma'-N\alpha}}|y|^{b_j}\prod_{i\neq j} |\eps^{\frac{\gamma'}{\gamma'-N\alpha}}y-x_i+x_j|^{b_i}dy\\ \leq C \eps^{\frac{\gamma'(N\alpha+b_j)}{\gamma'-N\alpha}}\int_{\R^n}\hat m_\eps(y)  |y|^{b_j}\prod_{i\neq j} |y-x_i+x_j|^{b_i}dy  \leq C\eps^{\frac{\gamma'(N\alpha+b_j)}{\gamma'-N\alpha}} \end{multline} 
Note in particular that we can choose in the previous inequality $b_j=\max_i b_i$.

We   get from \eqref{dise} applied to $z=y_\eps-\eps^{-\frac{\gamma'}{\gamma'-\alpha N}}x_j$, where $j$ is such that  $b_j=\max_i b_i$,  and from \eqref{uno}  the following improvement of \eqref{intvv} 
\begin{equation}\label{intvv2}\int_{B(0,R)} \bar m_\eps V_\eps(y+y_\eps)dy\leq \int_{\R^N} \hat m_\eps(y+y_\eps-\eps^{-\frac{\gamma'}{\gamma'-\alpha N}}x_j) V_\eps(y+y_\eps)dy\leq C \eps^{\frac{(N\alpha+\max b_i) \gamma'}{\gamma'-N\alpha}}\end{equation} for all $R\geq 0$.  We choose $R>0$ sufficiently large such that $ \int_{B(0, R)} \bar m_\eps dy \geq \frac{M}{2}$. 
Recalling the rescaling of $V$, \eqref{intvv2} implies  that 
\begin{equation}
\label{m}  C \eps^{\frac{ \max b_j \gamma'}{\gamma'-N\alpha}}\geq \frac{M}{2}C_V^{-1} \min_{y\in B(0,R)} \prod_{j = 1}^{n} |\eps^{\frac{\gamma'}{\gamma'-N\alpha}} y+ \eps^{\frac{\gamma'}{\gamma'-N\alpha}} y_\eps-x_j|^{b_j}.
\end{equation} 
Note that for $\eps$ sufficiently small $ |\eps^{\frac{\gamma'}{\gamma'-N\alpha}} y+ \eps^{\frac{\gamma'}{\gamma'-N\alpha}} y_\eps-x_j|\geq \delta>0$ for all $i\not = \iota$ and all $y\in B(0,R)$. 
So, by \eqref{m} we get that there exists $C>0$ for which 
\[ \min_{y\in B(0, R)}|\eps^{\frac{\gamma'}{\gamma'-N\alpha}} y+ 
\eps^{\frac{\gamma'}{\gamma'-N\alpha}} y_\eps-x_{\iota}|^{b_{\iota}}\leq C \eps^{\frac{ \max b_j \gamma'}{\gamma'-N\alpha}} \] and then 
\begin{equation}\label{m2} | \hat y_\eps- \eps^{-\frac{\gamma'}{\gamma'-N\alpha}} x_{\iota}|^{b_{\iota}}= \min_{y\in B(0, R)}|y+ y_\eps- \eps^{-\frac{\gamma'}{\gamma'-N\alpha}} x_{\iota}|^{b_{\iota}}\leq C \eps^{\frac{ (\max b_j-b_{\iota})  \gamma'}{\gamma'-N\alpha}} \to 0 \end{equation} for some $\hat y_\eps\in B(y_\eps, R)$. Let $z_\eps=\hat y_\eps-y_\eps \in B(0, R)$. Up to subsequences we can assume that $z_\eps\to \bar z\in B(0, R)$. 

We use now \eqref{intvv2}, recalling assumption \eqref{Vpoly}, we get that 
\begin{multline*}  C \eps^{\frac{ \max b_j \gamma'}{\gamma'-N\alpha}}\geq C_V^{-1} \int_{B(0,R)} \bar m_\eps(y) \prod_{j = 1}^{n} |\eps^{\frac{\gamma'}{\gamma'-N\alpha}} y+ \eps^{\frac{\gamma'}{\gamma'-N\alpha}} y_\eps-x_j|^{b_j} dy \\ 
\geq c_1 \eps^{\frac{ b_{\iota} \gamma'}{\gamma'-N\alpha}}\int_{B(0,R)} \bar m_\eps(y) |y-z_\eps +\hat y_\eps-\eps^{-\frac{\gamma'}{\gamma'-N\alpha}} x_{\iota}|^{b_{\iota}}dy.  \end{multline*} 
In particular this implies that 
\begin{equation} \label{tre} \lim_{\eps\to 0} \int_{B(0,R)} \bar m_\eps(y) |y-z_\eps +\hat y_\eps-\eps^{-\frac{\gamma'}{\gamma'-N\alpha}} x_{\iota}|^{b_{\iota}}dy=0.\end{equation}

Recalling that $\bar m_\eps\to \bar m$ locally uniformly (see \eqref{convergenza1}), that $\hat y_\eps- \eps^{-\frac{\gamma'}{\gamma'-N\alpha}} x_{\iota}\to 0$ by \eqref{m2},  and that $z_\eps\to \bar z$, we get 
\[\lim_{\eps\to 0} \int_{B(0,R)} \bar m_\eps(y) |y-z_\eps +\hat y_\eps-\eps^{-\frac{\gamma'}{\gamma'-N\alpha}} x_{\iota}|^{b_{\iota}}dy=\int_{B(0,R)} \bar m(y)|y-\bar z|^{b_{\iota}}dy >0. \] 
This gives a contradiction with \eqref{tre}. \end{proof}

As a consequence of the previous results, we can conclude with the
\begin{proof}[Proof of Theorem \ref{concthm}] Setting $x_\eps = \eps^{\frac{\gamma'}{\gamma'-\alpha N}}y_\eps$, it suffices to recall \eqref{conc1} and Propositions \ref{minimo}, \ref{teoy}. \end{proof}

\small
\addcontentsline{toc}{section}{References}
\bibliography{concMFGpart2}
\bibliographystyle{abbrv}

\medskip

\medskip
%\small
\begin{flushright}
\noindent \verb"annalisa.cesaroni@unipd.it"\\
Dipartimento di Scienze Statistiche\\ Universit\`a di Padova\\
Via Battisti 241/243, 35121 Padova (Italy)

\smallskip

\noindent \verb"cirant@math.unipd.it"\\
Dipartimento di Matematica ``Tullio Levi-Civita''\\ Universit\`a di Padova\\
via Trieste 63, 35121 Padova (Italy)
\end{flushright}

\end{document}